\documentclass[10 pt,pdftex]{article}
\pdfoutput=1
\usepackage[margin=2.5cm]{geometry}
\usepackage[all]{xy}
\usepackage{arydshln,colortbl,float,longtable,xcolor,graphicx,amsmath,amsthm,amssymb,amscd,enumerate,booktabs,bookmark}

\graphicspath{{images/}}

\newcommand{\V}[1]{\mathcal{V}\!\left({#1}\right)}
\newcommand{\gap}{\hspace{0.5em}}
\newcommand{\vgap}{\vspace{0.25em}}

\newcommand{\padding}{\rule[-1.45ex]{0pt}{0.2em}\gap}
\newcommand{\oddrow}{\rowcolor[gray]{0.95}}
\newcommand{\evnrow}{}
\newlength{\matrixheight}
\newcommand{\calulatepadding}[1]{\settoheight{\matrixheight}{\hbox{#1}}{#1}}

\theoremstyle{plain}

\newtheorem{thm}{Theorem}
\newtheorem{pro}[thm]{Proposition}
\newtheorem{lem}[thm]{Lemma}
\newtheorem{cor}[thm]{Corollary}

\theoremstyle{definition}

\newtheorem{dfn}[thm]{Definition}
\newtheorem{dfnrem}[thm]{Definition-Remark}
\newtheorem{rem}[thm]{Remark}
\newtheorem{exa}[thm]{Example}

\newcommand{\Bmu}{\mbox{$\raisebox{-0.59ex}
  {$l$}\hspace{-0.18em}\mu\hspace{-0.88em}\raisebox{-0.98ex}{\scalebox{2}
  {$\color{white}.$}}\hspace{-0.416em}\raisebox{+0.88ex}
  {$\color{white}.$}\hspace{0.46em}$}{}}

\DeclareMathOperator{\rk}{rk}

\DeclareMathOperator{\Defo}{Def}
\DeclareMathOperator{\Cl}{Cl}

\DeclareMathOperator{\codim}{codim}
\DeclareMathOperator{\Hom}{Hom}

\DeclareMathOperator{\hcf}{hcf}
\DeclareMathOperator{\Image}{Im}

\DeclareMathOperator{\Nef}{Nef}
\DeclareMathOperator{\Amp}{Amp}
\DeclareMathOperator{\Spec}{Spec}
\DeclareMathOperator{\Sing}{Sing}
\DeclareMathOperator{\wG}{wGr}
\DeclareMathOperator{\Irr}{Irr}
\DeclareMathOperator{\Pf}{Pf}

\newcommand{\bv}{\mathbf{v}}
\newcommand{\oo}{\mathcal{O}}
\newcommand{\QQ}{\mathbb{Q}}
\newcommand{\RR}{\mathbb{R}}
\newcommand{\PP}{\mathbb{P}}
\newcommand{\ZZ}{\mathbb{Z}}
\newcommand{\NN}{\mathbb{N}}
\newcommand{\FF}{\mathbb{F}}
\newcommand{\CC}{\mathbb{C}}
\newcommand{\GG}{\mathbb{G}}
\newcommand{\LL}{\mathbb{L}}
\newcommand{\TT}{\mathbb{T}}
\newcommand{\Cstar}{\CC^\times}

\newcommand{\Xsurf}[2]{X_{#1, \, #2}}
\newcommand{\Bsurf}[2]{B_{#1, \, #2}}
\newcommand{\Ssurf}[2]{S_{#1, \, #2}}

\begin{document}
\bibliographystyle{plain}

\title{Del Pezzo surfaces with $\frac1{3}(1,1)$ points}

\author{Alessio Corti \thanks{a.corti@imperial.ac.uk}\\
 Department of Mathematics\\
 Imperial College London\\
 London, SW7 2AZ, UK\\
\and
 Liana Heuberger \thanks{liana.heuberger@imj-prg.fr}\\
 Institut Mathematique de Jussieu\\
 4 Place Jussieu\\
 75005 Paris, France}

\date{7\textsuperscript{th} September, 2015}

\maketitle

\begin{abstract} We classify non-smooth del Pezzo surfaces with $\frac1{3}(1,1)$
  points in $29$ qG-deformation families grouped into six unprojection
  cascades (this overlaps with work of Fujita and Yasutake
  \cite{FY14}), we tabulate their biregular invariants, we give good
  model constructions for surfaces in all families as
  degeneracy loci in rep quotient varieties, and we prove that
  precisely $26$ families admit qG-degenerations to toric
  surfaces. This work is part of a program to study mirror symmetry
  for orbifold del Pezzo surfaces \cite{surfaces15}.
\end{abstract}

\tableofcontents{}
 
\section{Introduction}
\label{sec:introduction}


In this paper we:
\begin{enumerate}[(I)]
\item Classify non-smooth del Pezzo surfaces with $\frac1{3}(1,1)$
  points in precisely $29$ qG-deformation families. We further
  structure the classification into six unprojection cascades,
  determine their biregular invariants and their directed MMP together
  with a distinguished configuration of negative curves on the minimal
  resolution. This overlaps with work of Fujita and
  Yasutake~\cite{FY14}.
\item Construct good models for surfaces in all families as
  degeneracy loci in rep quotient varieties. In all but two
  cases, the rep quotient variety is a simplicial toric variety.
\item Prove that precisely $26$ of the $29$ families admit a qG-degeneration
  to a toric surface. 
\end{enumerate}

The classification is summarised in table~\ref{table:2} and
table~\ref{table:1}, which also plot invariants and provide good model
constructions of surfaces in all families. 

This work is part of a program to understand mirror symmetry for
orbifold del Pezzo surfaces \cite{surfaces15, KNP14, OP14,
  tveiten14:_maxim_lauren,prince:_smoot_toric_fano_surfac_using,
  coates:_lauren} and it is aimed specifically at giving evidence
for the conjectures made in \cite{surfaces15}.

The rest of the introduction is organised as follows: in
\S~\ref{sec:our-results} we give precise statements of our main
results; in \S~\ref{sec:context} we say a few words about the context
of \cite{surfaces15}; in \S~\ref{sec:structure-paper} we outline the
structure of the paper.

\subsection{Our results}
\label{sec:our-results}

\subsubsection{The classification and its cascade structure}
\label{sec:class-casc-struct}

\begin{dfn}
  \label{dfn:5}
  A $\frac1{n}(a,b)$ \emph{point} is a surface cyclic quotient singularity
  $\CC^2/\Bmu_n$ where $\Bmu_n$ acts linearly on $\CC^2$ with weights
  $a$, $b\in \bigl( \frac1{n}\ZZ\bigr)/\ZZ$. We always assume no
  stabilisers in codimension $0$, $1$, that is,
  $\hcf(a,n)=\hcf(b,n)=1$.

  A \emph{del Pezzo surface} is a surface $X$ with cyclic quotient
  singularities and $-K_X$ ample.
 
  The \emph{Fano index} of $X$ is the largest positive integer $f>0$
  such that $-K_X=fA$ in the Class group $\Cl X$.
\end{dfn}

\begin{rem}
  \label{rem:11}
  In this paper we view a del Pezzo surface $X$ with quotient
  singularities as a variety. Such a surface is in a natural way a
  smooth orbifold (or DM stack), but we mostly ignore this
  structure. Thus for us $\Cl X$ is the Class group of Weil divisors
  on $X$ modulo linear equivalence. In particular, although $K_X$ is a
  Cartier divisor on the orbifold, we think of it as a $\QQ$-Cartier
  divisor on the underlying variety (the coarse moduli space of the
  orbifold) and then to say that it is ample is to say that a positive
  integer multiple is Cartier and ample.
\end{rem}

See \cite{surfaces15} for a discussion of qG-deformations of del Pezzo
surfaces with cyclic quotient singularities. In particular, it is
explained there that the singularity $\frac1{3}(1,1)$ is qG-rigid and
the degree $d=K^2$ is locally constant in qG-families.

We classify qG-deformation families of del Pezzo surfaces with
$k\geq 1$ $\frac1{3}(1,1)$ points.  It follows for example from the
proof of \cite[Lemma~6]{surfaces15} that del Pezzo surfaces $X$ with
fixed number $k$ of $\frac1{3}(1,1)$ points and $d=K_X^2$ form an
algebraic stack $\mathfrak{M}_{k,d}$ of dimension
\[
\dim \mathfrak{M}_{k,d} = -\chi (X,\Theta_X)=-h^0(X,\Theta_X)+h^1(X,\Theta_X)
\]
where $\Theta_X$ is the sheaf of derivations of $X$. Note that this
dimension can be negative when $X$ has a continuous group of
automorphisms. It is always the case that $H^2(X,\Theta_X)=(0)$ but in
some cases both $H^0(X,\Theta_X)$ and $H^1(X,\Theta_X)$ are non-zero. 

The following two theorems, together with theorem~\ref{thm:6} below,
are our main results:

\begin{thm}
  \label{thm:3}
  There are precisely $3$ qG-deformation families of del Pezzo surfaces
  with $k\geq 1$ $\frac1{3}(1,1)$ points and Fano index $f\geq
  2$. They are:
  \begin{enumerate}[(1)]
  \item $\Ssurf{1}{25/3}=\PP(1,1,3)$ with $k=1$, $K^2=\frac{25}{3}$ and $f=5$;
  \item $\Bsurf{1}{16/3}$: the family of weighted hypersurfaces $X_4\subset
    \PP(1,1,1,3)$ with $k=1$, $K^2=\frac{16}{3}$ and $f=2$;
  \item $\Bsurf{2}{8/3}$: the family of weighted hypersurfaces $X_6\subset
    \PP(1,1,3,3)$ with $k=2$, $K^2=\frac{8}{3}$ and $f=2$. 
  \end{enumerate}
\end{thm}

\begin{thm}
  \label{thm:1}
  There are precisely $26$ qG-deformation families of del Pezzo surfaces
  with $k\geq 1$ $\frac1{3}(1,1)$ points and Fano index
  $f=1$. The numerical invariants of these surfaces are shown in
  table~\ref{table:1} in \S~\ref{sec:tables}. In that table
  $X_{k,\,d}$ denotes the unique family with
  $k$ $\frac1{3}(1,1)$ points, $K^2=d$ and $f=1$. The table also
  gives a good model construction of a surface $X$ in all families. 
\end{thm}

Detailed information on how to read the table is given in
\S~\ref{sec:tables}. In that section, we also introduce several
invariants and explain how to compute some that are not shown in the
table. For example, denoting by $X^0=X^{\text{nonsing}}=X\setminus
\Sing X$ the nonsingular locus of $X$, proposition~\ref{pro:3}(b)
states that $\pi_1(X^0)=(0)$ for all families except $\Xsurf{6}{1}$
and $\Xsurf{6}{2}$, for which $\pi_1(X^0)=\ZZ/3\ZZ$.

Next we discuss the finer structure of the classification.

\begin{dfn}
  \label{dfn:4} 
  A negative curve on $X$ is a compact curve $C\subset X$ with
  negative self-intersection number $C^2<0$.
  We say that a compact curve $C\subset X$ is a
  $(-m)$-curve if $C^2=-m$. Note that in general $m\in \QQ$.
  Let $P_1,\dots,P_k\in X$ be the singular points and denote by
  $X^0=X^\text{nonsing}=X\setminus \{P_1,\dots P_k\}$ the nonsingular locus
  of $X$. A $(-1)$-curve $C\subset X^0$ is called a \emph{floating}
  $(-1)$-curve.
\end{dfn}

Theorem~\ref{thm:3} and theorem~\ref{thm:1} are a straightforward
logical consequence of the minimal model program and the
following, which is proved in \S~\ref{sec:trees}:

\begin{thm}
  \label{thm:5}
  Let $X$ be a del Pezzo surface with $k\geq 1$ $\frac1{3}(1,1)$
  points. If $X$ has no floating $(-1)$-curves, then $X$ is one of the
  following surfaces, all constructed in table~\ref{table:2}
  and~\ref{table:1} and in the statement and proof of theorem~\ref{thm:4}:
\begin{enumerate}[(1)]
\item $k=1$ and either $X$ is a surface of the family of weighted
  hypersurfaces $\Bsurf{1}{16/3} = X_4\subset \PP(1,1,1,3)$, or
  $X=\Ssurf{1}{25/3}=\PP(1,1,3)$;
\item $k=2$ and either $X=\Xsurf{2}{17/3}$, or
  $X$ is a surface of the family of weighted hypersurfaces $\Bsurf{2}{8/3} = X_6\subset \PP(1,1,3,3)$;
\item $k=3$ and $X=\Xsurf{3}{5}$;
  \item $k=4$ and $X=\Xsurf{4}{7/3}$; 
\item $k=5$ and $X=\Xsurf{5}{5/3}$; 
\item $k=6$ and $X=\Xsurf{6}{2}$.
  \end{enumerate}
\end{thm}

\begin{rem}
  \label{rem:unique_surfaces}
  With the exception of families $\Bsurf{1}{16/3}$, $\Bsurf{2}{8/3}$, all the surfaces in
  theorem~\ref{thm:5} are qG-rigid: in other words, they are the only
  isomorphism class of surfaces in that family.  
\end{rem}

Theorem~\ref{thm:4} of \S~\ref{sec:trees} is a more precise version of
theorem~\ref{thm:5} just stated. In particular, the statement of
theorem~\ref{thm:4} in \S~\ref{sec:trees} has images showing the
directed MMP for $X$ that provide a birational construction of $X$, and
pictures of a distinguished configuration of negative curves in the
minimal resolution $f\colon Y \to X$.

In all cases, we could have pushed our analysis to the point where we
could have made a list of all negative curves on $Y$ and $X$, and computed
generators of the nef cones $\Nef Y$, $\Nef X$. We did not pursue this as we
don't have a compelling reason to do so.

Surfaces with a given $k$ are all obtained by a \emph{cascade}---the
terminology is due to~\cite{MR2053462}---of blow-ups of smooth points
starting with one of the surfaces in theorem~\ref{thm:5}:

\begin{cor}
  \phantomsection \label{cor:1}
 \begin{enumerate}[(1)]
 \item A surface of the family $\Xsurf{1}{d}$ is the blow-up of
   $25/3-d\leq 8$ nonsingular points on $\PP(1,1,3)$. If $d< 16/3$,
   then it is also the blow-up of a surface of the family $\Bsurf{1}{16/3}$ in
   $1\leq 16/3 -d \leq 5$ nonsingular points;
 \item A surface of the family $\Xsurf{2}{d}$ is the blow-up of
   $17/3-d \leq 5$ nonsingular points on $\Xsurf{2}{17/3}$. If
   $d<8/3$, then it is also the blow-up of a surface of the family $\Bsurf{2}{8/3}$ in
   $1\leq 8/3-d\leq 2$ nonsingular points;
  \item A surface of the family $\Xsurf{3}{d}$ is the blow-up of $5-d
    \leq 4$ nonsingular points on $\Xsurf{3}{5}$;
  \item A surface of the family $\Xsurf{4}{d}$ is the blow-up of
    $7/3-d \leq 2$ nonsingular points on $\Xsurf{4}{7/3}$;
  \item A surface of the family $\Xsurf{5}{2/3}$ is the blow-up of a
    nonsingular point on $\Xsurf{5}{5/3}$;
  \item $\Xsurf{6}{1}$ is the blow-up of a nonsingular point on
    $\Xsurf{6}{2}$.
  \end{enumerate}\qed
\end{cor}

\begin{rem}
  \label{rem:cascades}
  In the cases $k=1$ and $k=2$, corollary~\ref{cor:1} is not an
  immediate consequence of theorem~\ref{thm:5}. Indeed, given a
  surface $X$, it is clear that a sequence of contractions of floating
  $(-1)$-curves leads to one of the surfaces listed in
  theorem~\ref{thm:5}. We need to show, in addition, that:
  \begin{enumerate}[(1)]
  \item If $X\to \Bsurf{1}{16/3}$ is the
  blow-up of a nonsingular point, there is an alternative
  sequence of $4$ blow-downs of floating $(-1)$-curves starting from
  $X$  and ending in $\PP(1,1,3)$; 
\item If $X\to \Bsurf{2}{8/3}$ is the
  blow-up of a nonsingular point, then there is an alternative
  sequence of $4$ blow-downs of floating $(-1)$-curves starting from
  $X$  and ending in the surface $\Xsurf{2}{17/3}$.
  \end{enumerate}
 These facts are easy to verify from the explicit birational
 constructions given in theorem~\ref{thm:4}.
\end{rem}

\subsubsection{Good model constructions}
\label{sec:good-model-constr}

We summarise all of the key features of the constructions provided by
table~\ref{table:1}.
 
\begin{dfn}
  \label{dfn:7}
  \begin{itemize}
  \item A \emph{rep quotient variety} is a geometric quotient $F=A/\!\!/G$
where $G$ is a complex Lie group and $A$ a representation of $G$.
\item Let $L_1,\dots,L_c$ be line bundles on $F$ constructed from characters
$\rho_i \colon G \to \Cstar$ ($i=1,\dots,c$).  A subscheme
$X\subset F$ of pure codimension $c$ is a \emph{complete intersection}
of type $(L_1,\dots, L_c)$ on $F$ if $X=Z(\sigma)$ is the zero-scheme of a
section
\[
\sigma \in H^0(F; L_1\oplus \cdots \oplus L_c)
\] 
\item Let $E_1$, $E_2$ be vector bundles on $F$ constructed from two
  representations of $G$. A subscheme $X\subset F$ is a
  \emph{degeneracy locus} on $F$ if $X$ is the locus where a vector
  bundle homomorphism $s\colon E_1\to E_2$ drops rank, provided that
  this locus has the expected codimension.
  \end{itemize}
\end{dfn}

Examples of rep toric varieties are toric varieties, where $G$ is a torus, but also the
weighted Grassmannians of \cite{MR1954062}.

For all of the 26 families $\Xsurf{k}{d}$ of theorem~\ref{thm:1}, we
list in table~\ref{table:1} a \emph{good model construction} of a
surface of the family as a degeneracy locus in a rep quotient
variety. 

For all but one pair $(k,d)$, table~\ref{table:1} shows a rep
quotient variety $F$ and line bundles $L_1,\dots, L_c$ with the
following properties:
\begin{enumerate}[(a)]
\item the line bundles $L_i$ are nef on $F$,
\item $-K_F-\Lambda$ is ample on $F$,\footnote{By construction the
    line bundle $-K_F-\Lambda$ is $G$-linearised, and this uniquely
    specifies the GIT problem of which $F$ is the
    solution. Table~\ref{table:1} also gives a complete description of
    the cone $\Nef F$ of stability conditions.} where
  $\Lambda=\sum_{i=1}^c L_i$,
\end{enumerate}
such that a general complete intersection $X$ of type
$(L_1,\dots, L_c)$ on $F$ is a surface of the family\footnote{In fact, one can
  verify, a general surface of the family. We do not claim, however,
  that every surface of the family has such a description. This may
  be true, and it is definitely true for rigid surfaces, but we did
  not check it in general.} $\Xsurf{k}{d}$. 

Since, by the adjunction formula, $-K_X=-(K_F+\Lambda)_{|X}$ is ample,
the constructions make it manifest that $X$ is a Fano variety.  In all
cases it is easy to verify that a general complete intersection of
type $(L_1,\dots,L_c)$ on $F$ has $k$ $\frac1{3}(1,1)$ singularities,
and compute $k$ and the anticanonical degree $d=K_X^2$.

In 24 out of 25 cases $F$ is a toric variety. In 23 of the 24 cases,
$F$ is a well-formed simplicial toric variety, and the surface itself
is quasi-smooth and well-formed. (These notions are recalled in
\S~\ref{sec:well-form-simpl}.)

\S~\ref{sec:constructions} summarises conventions and facts about
toric varieties and gives model computations demonstrating all of
these statements.

\begin{rem}
  \label{rem:4}
  Our model construction are not unique. In many cases, several
  similar constructions exist. It would be nice to understand all
  model constructions systematically.
\end{rem}

We say a few words about the three exceptions:

\paragraph{Family $\Xsurf{1}{7/3}$}

In \S~\ref{sec:model-wg-2} we describe a weighted Grassmannian
$F=\wG(2,5)$ and a complete intersection $X$ in it which is a surface in
this family. $F$ is a well-formed orbifold, and $X$ is quasi-smooth
and well-formed.

\paragraph{The surface $\Xsurf{5}{5/3}$}

This family is qG-rigid and it consists of a single surface. This
surface has a simple birational construction: blow up the vertices of
a pentagon of $5$ lines on a smooth del Pezzo surface of degree $5$,
and blow down the strict transforms of the $5$ lines. In
\S~\ref{sec:x_5-53} we construct a model of this surface as
codimension~3 degeneracy locus of an antisymmetric vector bundle
homomorphism $s\colon E\otimes L \to E^\vee$ where $E$ is a rank~5
split vector bundle on a simplicial toric 5-fold.

\paragraph{Family $\Xsurf{5}{2/3}$}

In \S~\ref{sec:x_5-23} we describe a non-simplicial toric variety $F$
and a complete intersection $X$ in it which is a surface in this
family. We verify that $X$ misses the non-orbifold locus of
$F$. Outside of this locus $F$ is a well-formed orbifold and $X$ is
quasi-smooth and well-formed. We did not succeed in finding a good
model construction for a surface in this family in a simplicial toric
variety. Such a construction may well exist but it it very difficult
computationally to look for it, particularly since this family does
not admit a toric degeneration. There are, in fact, two difficulties:
the software does not exist, and the computations are very large.

\subsubsection{Toric qG-degenerations}
\label{sec:toric-qg-degen}

In \S~\ref{sec:toric-degenerations} we prove the following:

\begin{thm}
  \label{thm:6}
  With the exception of $\Xsurf{4}{1/3}$, $\Xsurf{5}{2/3}$ and
  $\Xsurf{6}{1}$ (all of which have $h^0(X,-K_X)=0$), all other
  families admit a qG-degeneration to a toric surface.
\end{thm}

Table~\ref{tab:third_one_one_polygons} and
figure~\ref{fig:third_one_one_polygons} list $26$ lattice
polygons $P$ such that their face-fans $\Sigma(P)$ give toric surfaces
$X_P$ that are qG-degenerations of the families in
theorem~\ref{thm:6}. 

\subsubsection{Comments on the literature and on our proofs}
\label{sec:comments-literature-}

The cascade for the surfaces with $k=1$ was discovered by Reid and
Suzuki in \cite{MR2053462}.

Del Pezzo surfaces with quotient singularities, also known as log del
Pezzo surfaces, are studied in \cite{MR2227002, MR1039966, MR1009466,
  MR993456, MR1787240, MR972094, MR2372472}. 

In two remarkable papers, De-Qi Zhang \cite{MR957874, MR1007096}
classifies log del Pezzo surfaces of Picard rank~1 closely related to
our surfaces and outlines a general strategy to classify all rank~1 log del
Pezzo surfaces. DongSeon Hwang has recently announced a complete
classification of rank~1 log del Pezzo surfaces. 

While we were working on this project, paper \cite{FY14} appeared,
containing a classification of log del Pezzo surfaces with Gorenstein
index $3$.  While we do not classify all log del Pezzo surfaces of
Gorenstein index~3, in some other respect we classify more surfaces
then \cite{FY14}. Indeed, the discussion in \S~\ref{sec:context} shows
that a del Pezzo surfaces with singularities as example~\ref{exa:1}(a)
and~(b) qG-deforms to one of our surfaces. Our classification is done
by similar methods: we determine all possibilities for the directed
MMP of the surface by a detailed combinatorial study of the
configuration of negative curves on the minimal resolution. In our
view, our proof is simpler than that in \cite{FY14}. The cascade
structure, our good model constructions, and the statement about toric
degenerations, are new.

It took us a significant effort to find the good model
constructions. For many of the families, it is comparatively easy to
find a construction of a general surface as a complete intersection in
a toric variety, but very hard to find one that satisfies properties
(a) and (b) of \S~\ref{sec:good-model-constr}. Initially, we
discovered some constructions by hand using birational geometry; then,
we found more with the help of a computer search; finally, we learned
a systematic way \cite{coates:_lauren}. To determine which of the
families admit a toric qG-degeneration we made a computer search for
the relevant Fano lattice polygons. A unified picture comprising both
these facts---the toric complete intersection model and the toric
degeneration---would be desirable. The paper \cite{coates:_lauren}
contains the beginning of such a theory. 

At this time we can not imagine a geometric explanation for the fact
that some del Pezzo surfaces do, and some do not, admit a
qG-degeneration to a toric surface.

\subsection{Context}
\label{sec:context}

We put our results in the context of a general program to understand
mirror symmetry for orbifold del Pezzo surfaces \cite{surfaces15,
  KNP14, OP14,
  tveiten14:_maxim_lauren,prince:_smoot_toric_fano_surfac_using,
  coates:_lauren} and answer the question: Why classify del Pezzo
surfaces with $\frac1{3}(1,1)$ points?

qG-deformations of surface singularities is a technical notion
that ensures that the canonical class is well-behaved in families: in
particular, $K^2$ is locally constant in a qG-family of proper
surfaces. (The Gorenstein index, crucially, is not locally constant in
a qG-family.) Thus, it is natural to restrict our attention to qG-deformations.

\begin{dfn}
  \label{dfn:6}
  \begin{itemize}
  \item A cyclic quotient surface singularity is \emph{class T} if it
    has a qG-smoothing, cf.\ \cite[Definition~3.7]{MR922803};
  \item A cyclic quotient surface singularity is \emph{class R} if it
    is qG-rigid, cf.\ \cite{surfaces15, AK14};
\item A del Pezzo surface $X$ with cyclic quotient singularities is
  \emph{locally qG-rigid} if $X$ has singularities of class $R$ only.
  \end{itemize}
\end{dfn}

It is well-known that a surface cyclic quotient singularity
$(x\in X)\cong \frac1{n}(1,q)$ has a unique qG-deformation component
and that the general surface of the miniversal family has a unique
singularity of class $R$, the \emph{$R$-content} of $(x\in X)$, cf.\
\cite{surfaces15} and \cite[Definition~2.4]{AK14} and the discussion
following it. 

\begin{exa}
  \phantomsection \label{exa:1}
  \begin{enumerate}[(a)]
  \item A singularity is of class $T$ if and only if it is of the form
    $\frac1{r^2w}(1,rma-1)$.
  \item The singularity $\frac1{3}(1,1)$ is qG-rigid. A singularity has $R$-content
$\frac1{3}(1,1)$ if and only if it is of the form $\frac1{3(3m+1)}\bigl(1,
2(3m+1)-1\bigr)$.
  \end{enumerate}
\end{exa}

It is known \cite[Lemma~6]{surfaces15} that, if $X$ is a del Pezzo
surface with cyclic quotient singularities $x_i\in X$, the natural
transformation of qG-deformation functors:
\[
\Defo^{\text{qG}} X \to \prod \Defo^{\text{qG}} (X,x_i)
\]
is smooth: a choice for each $i$ of a local qG-deformation of the
singularity $(X,x_i)$ can always be globalised to a qG-deformation of
$X$. In other words $X$ can be qG-deformed to a surface that has only
the residues of the $(X,x_i)$ as singularities. 

Our point of view here is that, when we classify del Pezzo surfaces,
and study mirror symmetry for them, it is natural to classify first
the locally qG-rigid ones, for these are the generic surfaces that we
are most likely to meet, and study their qG-degenerations as a second
step. (Note that the algebraic stack of qG-families of orbifold del
Pezzo surfaces $X$ with fixed $K_X^2$ is almost always unbounded; for
example it is unbounded when $X=\PP^2$, see for instance
\cite{MR1116920}.) This study is motivated by the fact that
$\frac1{3}(1,1)$ is the simplest class $R$ singularity.

In \cite{surfaces15} mirror symmetry for a locally qG-rigid del Pezzo
surface is stated in terms of a qG-degeneration to a toric surface:
thus it is crucial for us to determine which families admit such a
degeneration. The mirror symmetry conjecture~B of \cite{surfaces15}
computes the quantum orbifold cohomology of a locally qG-rigid surface
$X$ from data attached to the toric qG-degeneration. In order to
compute the quantum orbifold cohomology of a surface $X$ by the known
technology of abelian/nonabelian correspondence and quantum Lefschetz
\cite{MR2367022, MR2276766, MR2578300}, and thus give evidence for
conjecture~B of \cite{surfaces15}, we need a model of $X$ as a
complete intersection in a rep quotient variety. In this context, we
need conditions (a) and (b) of \S~\ref{sec:good-model-constr} to
control the asymptotics of certain $I$-functions, and this motivates
our constructions here. Conditions (a) and (b) are of course also
natural from a purely classification-theoretic
perspective. Paper~\cite{OP14} computes (part of) the quantum orbifold
cohomology of our surfaces.

Part of our motivation in undertaking this classification was to ask
seriously how general mirror symmetry is. Out of the 29 families that
comprise our classification, 26 admit a toric qG-degeneration and
\cite{surfaces15} provides a mirror construction for them. (A mirror
for these surfaces can also, in principle, be constructed by means of
the Gross--Siebert program, see
\cite{prince:_smoot_toric_fano_surfac_using}.) What about the
remaining three families $\Xsurf{4}{1/3}$, $\Xsurf{5}{2/3}$ and
$\Xsurf{6}{1}$?  It turns out that we can construct surfaces in these
families as complete intersections in toric varieties, thus these
families also have mirrors, by the Hori--Vafa construction.

\subsection{Structure of the paper}
\label{sec:structure-paper}

The paper is structured as follows. Section~\ref{sec:tables} contains
the tables that summarise the classification, with instructions
on how to read them. We also introduce several invariants and explain
how to compute some that are not shown in the tables. In particular,
denoting by $X^0=X^{\text{nonsing}}=X\setminus \Sing X$ the
nonsingular locus of $X$, proposition~\ref{pro:3}(b) of
\S~\ref{sec:fano-index-pi_1x0} states that $\pi_1(X^0)=(0)$ for all
families except $\Xsurf{6}{1}$ and $\Xsurf{6}{2}$, for which
$\pi_1(X^0)=\ZZ/3\ZZ$ .

Section~\ref{sec:constructions} summarises some facts from toric
geometry needed to validate the tables: in particular, we give the
notion of quasi-smooth and well-formed complete intersections in a
well-formed simplicial toric variety, and sample computations
verifying that the constructions of table~\ref{table:1} really
construct what they say they do. Sections~\ref{sec:x_5-53}
and~\ref{sec:x_5-23} are extended essay on models of surfaces
$\Xsurf{5}{5/3}$ and $\Xsurf{5}{2/3}$, and \S~\ref{sec:birat-constr-4}
collects special birational constructions in some cases.

Section~\ref{sec:invariants} studies some of the basic invariants
introduced in \S~\ref{sec:invariants_tiny} and uses elementary lattice
theory and covering space theory to obtain almost optimal bounds for
them that we use later on to cut down the number of cases that we need
to consider in the proof of theorems~\ref{thm:5} and~\ref{thm:4}. This
material is not strictly necessary for the proof of
theorems~\ref{thm:5} and~\ref{thm:4} but it does simplify it. Part of
our reason to include it here is that the use of lattice theory and
elementary covering space theory is very effective and we think it may
have applications in other problems of classification of orbifold del
Pezzo surfaces.

Section~\ref{sec:mmp} summarises all that we need from the
Minimal Model Program in the proof of theorems~\ref{thm:5}
and~\ref{thm:4} and introduces the directed MMP that we use to
organise the combinatorics of the proof. 

Section~\ref{sec:trees} is the heart of the paper. We prove theorems~\ref{thm:5}
and~\ref{thm:4}. In particular, the statement of theorem~\ref{thm:4}
has images showing the directed MMP for all the surfaces that provide
birational constructions of them, and pictures of a distinguished
configuration of negative curves in the minimal resolution.

In the final \S~\ref{sec:toric-degenerations} we prove
theorem~\ref{thm:6}. We list and picture $26$ lattice polygons and we
show that the corresponding toric surfaces are qG-degenerations of
the families of theorem~\ref{thm:6}.

We refer to the short summary at the beginning of each section for
more detailed information on the structure and content of that section.

\subsection{Acknowledgments}
\label{sec:acknowledgements}

This work started as one of a number of interconnected projects at the
PRAGMATIC 2013 Research School in Algebraic Geometry and Commutative
Algebra ``Topics in Higher Dimensional Algebraic Geometry'' held in
Catania in September 2013. The work continued at the EMS School ``New
Perspectives on the classification of Fano Manifolds'' held in Udine
in September 2014. We are very grateful to Alfio Ragusa, Francesco
Russo, and Giuseppe Zappal\`a, the organisers of the PRAGMATIC school,
and to Pietro De Poi and Francesco Zucconi, the organisers of the EMS
school, for creating a wonderful atmosphere for us to work in.

We thank Alexander Kasprzyk, who did the computer programming needed
to search for the models of the surfaces as complete intersections in
simplicial toric varieties and the Fano lattice polygons that give the
toric degenerations.

We also thank Mohammad Akhtar, Gavin Brown, Tom Coates, Alessandro
Oneto, Andrea Petracci, Thomas Prince, Miles Reid and Ketil Tveiten
for many helpful conversations.


\section{Tables}
\label{sec:tables}


Tables~\ref{table:2} and~\ref{table:1} summarise the classification,
provide model constructions for a surface in each of the families,
and display some of their numerical invariants. In
\S~\ref{sec:invariants_tiny} we introduce several invariants and state
some of the relations that hold between them. We explain how to
compute some of the invariants that are not displayed in the table
from those that are. In \S~\ref{sec:tables-1} we explain how to read
the tables. In \S~\ref{sec:fano-index-pi_1x0} we discuss the Fano
index and $\pi_1(X^0)$ where $X$ is one of our surfaces and
$X^0=X^\text{nonsing}=X\setminus \Sing X$ is the nonsingular
locus. Proposition~\ref{pro:3} states that: (a) the Fano index of the
families of theorem~\ref{thm:3} really is as claimed and (b)
$\pi_1(X^0)=\ZZ/3\ZZ$ for families $\Xsurf{6}{2}$ and $\Xsurf{6}{1}$,
and it is trivial for all other families.

\subsection{Invariants}
\label{sec:invariants_tiny}

Here $X$ is a surface of one of the $29$ families of Del
Pezzo surfaces with $k\geq 1$ $\frac1{3}(1,1)$ points, and
$X^0=X^\text{nonsing}= X\setminus \Sing X$ is the nonsingular locus of
$X$. We are interested in the following invariants. 

\begin{enumerate}[(i)]
\item $k$, the number of singular points of $X$;
\item $K^2=K_X^2$, the anticanonical \emph{degree} of $X$. It is obvious that $K_X^2>0$
  and $K_X^2\equiv \frac{k}{3} \pmod{\ZZ}$;
\item $h^0(X, -K_X)$, an integer $\geq 0$ and, more generally,
  $h^0(X,-nK_X)$ for all integers $n\geq 0$;
\item $r=\rho(Y)=\rho(X)+k$, the Picard rank of the minimal resolution $f\colon Y
  \to X$;
\item $n=e(X^0)=\widehat{c}_2(X)-k/3=2+\rho(X)-k$, where $e$ is the
  (homological) topological Euler number and
  $\widehat{c}_2(X)=c_2(\widehat{T}_X)$ is the orbifold second Chern class
  of $X$;
\item $\sigma$, the \emph{defect} of $X$, defined as follows: let
  $L=H^2(Y;\ZZ)$, viewed as a unimodular lattice by means of the
  intersection form, let $N=\langle -3\rangle^{\perp \,k} \subset L$ be the
  sublattice spanned by the $(-3)$-curves, and let
  $\overline{N}=\{\bv\in L \mid \exists d\in \ZZ \; \text{with} \;
  d\bv \in N \}$
  be the saturation of $N$ in $L$, then, for some integer $\sigma>0$,
  $\overline{N}/N\cong \FF_3^\sigma$.  Indeed, note that
  $\overline{N}/N\subset N^\ast/N$ where $N^\ast=\Hom (N,\ZZ)$ and
  $N\subset N^\ast$ the natural inclusion given by the quadratic
  form. Note that $N^\ast/N$ is $3$-torsion and isomorphic to $(\ZZ/3\ZZ)^k$, so
  $\overline{N}/N$ is also $3$-torsion. Equivalently,
  $\sigma=k-\rk \Image [N\otimes \FF_3\to L\otimes \FF_3]$. We prove
  in lemma~\ref{lem:3} below that $\FF_3^\sigma \cong H_1(X^0;\ZZ)$;
\item The number of moduli, that is, the dimension $\dim \mathfrak{M}$
  of the moduli stack. This number is
  $-\chi(X,\Theta_X)=h^1(X,\Theta_X)-h^0(X,\Theta_X)$ (it is
  well-known that $H^2(X,\Theta_X)=(0)$, see for example
  \cite{surfaces15}). By the Riemann--Roch theorem it is a topological
  invariant constant on qG-families;
\item The \emph{Fano index} $f$ of $X$, defined as follows: $f$ is the
  largest integer such that $-K_X=f A$ in $\Cl X$, for some divisor
  class $A\in \Cl X$;
\item The fundamental group $\pi_1(X^0)$.
\end{enumerate}

\begin{rem}
  \label{rem:6}
  The Riemann-Roch \cite[\S~3]{MR927963} and Noether formula state:
\[
h^0(X, -K_X)=1+K_X^2-\frac{k}{3}\quad
\text{and}
\quad
K_X^2=12-n-\frac{5k}{3}
\]
so one can compute $h^0(X,-K_X)$, $n$ and $r$ from $k$ and $K_X^2$
(vanishing implies that $h^0(X,-nK_X)=\chi (X,-nK_X)$ for $n\geq 0$).

It is easy from these data to compute the Poincar\'e series
$P_X(t)=\sum_{n\geq 0} t^nh^0(X,-nK_X)$. We state the result even
though it is not logically needed anywhere in the paper:
\[
P_X(t)=\frac{1+\bigl(K_X^2-1-\frac{k}{3}\bigr)t+\bigl(K_X^2+\frac{2k}{3}\bigr)t^2+\bigl(K_X^2-1-\frac{k}{3}\bigr)t^3+t^4}{(1-t)^2(1-t^3)}
\]
\end{rem}

\begin{rem}
  \label{rem:7}
  \begin{itemize}
  \item If $X$ admits a toric qG-degeneration, then $n=e(X^0)\geq
    0$. Indeed, in this case $n$ is the number of $T$-cones of the Fano
    polygon corresponding to the toric degenerate surface,
    see~\cite{AK14}.
  \item If $X$ admits a toric qG-degeneration, then $h^0(X,-K_X)\geq
    1$. Indeed, by the Riemann--Roch formula, $h^0(X,-K_X)$ is
    constant on a qG family and if $X_0$ is toric then $H^0(X_0,
    -K_{X_0})\neq (0)$ since it contains at least the toric boundary
    divisor of $X_0$.
  \item It follows \cite[Chapter~10]{MR1225842} from the generic
    semi-positivity of $\widehat{T}_X$ \cite[1.8~Corollary]{MR1610249} that
    $\widehat{c}_2(X)\geq 0$.
  \end{itemize}
\end{rem}

\begin{rem}
  \label{rem:5}
  Table~\ref{table:2} and table~\ref{table:1} show the invariants $k$,
  $d$, $h^0(X,-K_X)$, $r$ and $\dim\mathfrak{M}$. As we just
  explained, the invariants $\rho(X)$, $n$, $\widehat{c}_2(X)$, and
  all the $h^0(X,-nK_X)$ are easily computed from these.

  The tables do not show the defect $\sigma$ and
  $\pi_1(X^0)$. Proposition~\ref{pro:3} computes $\pi_1(X^0)$ and
  this, together with lemma~\ref{lem:3}, determines $\sigma$:
  $\Xsurf{6}{2}$ and $\Xsurf{6}{1}$ have $\pi_1(X^0)=\ZZ/3\ZZ$ (hence
  $\sigma=1)$, and all other families have $\pi_1(X^0)=(0)$ (hence
  $\sigma=0$).
\end{rem}

\subsection{Tables}
\label{sec:tables-1}

Tables~\ref{table:2} and~\ref{table:1} summarise the classification
and provide constructions for a general surface in each family. We
explain how to read the tables. We focus on table~\ref{table:1} since
table~\ref{table:2} is a straightforward illustration of
theorem~\ref{thm:3}.

The symbol $\Xsurf{k}{d}$ in the first column of table~\ref{table:1}
signifies the family of surfaces $X$ with $k$ singular points, degree
$K_X^2=d$ and $f=1$: the next three columns display the invariants
$h^0(X, -K_X)$, the rank $r=\rk H^2(Y;\ZZ)=\rho(Y)=k+\rho(X)$ where
$Y$ is the minimal resolution of $X$, and the dimension of the family.

The next column, in all but one case, shows a rep quotient variety
$F$ and line bundles $L_1, \dots, L_c$ on $F$ such that a general
complete intersection of type $(L_1,\dots,L_c)$ on $F$ is a surface of
the family. The last column computes the cone $\Nef F$: this
information is necessary to verify that conditions (a) and (b) of
\S~\ref{sec:good-model-constr} hold: the $L_i\in \Nef F$ and
$-K_F-\Lambda \in \Amp F$. We explain in more detail how to read the
information in these last two columns.

In all cases except $\Xsurf{1}{7/3}$, $\Xsurf{5}{5/3}$ and
$\Xsurf{5}{2/3}$, $F$ is a well-formed simplicial toric variety and
$X$ is a quasi-smooth and well-formed complete intersection of nef
line bundles $L_i$. All these notions are recalled in
\S~\ref{sec:well-form-simpl} below. Families $\Xsurf{6}{2}$ and
$\Xsurf{6}{1}$ are slightly anomalous: the simplicial toric ambient
variety $F$ is not in a direct way a rep quotient variety by a torus,
but by a product of a torus and a finite group. We discuss these two
families in greater detail in \S~\ref{sec:xsurf62} and
\S~\ref{sec:xsurf61} below.

In all cases, because $-K_F-\Lambda$ is Fano, the canonically
linearised line bundle $-K_F-\Lambda$ is a stability condition that
unambiguously specifies $F$ as a GIT quotient: see
\S~\ref{sec:well-form-simpl} for more details on this.

\subsubsection{The typical entry}
\label{sec:typical-entry}

All cases except $\Xsurf{1}{7/3}$, $\Xsurf{5}{5/3}$, $\Xsurf{5}{2/3}$,
$\Xsurf{6}{2}$ and $\Xsurf{6}{1}$ are typical. In a typical case, the
table gives the weight matrix of an action of $\CC^{\times\,l}$ on
$\CC^m$ such that $F=\CC^m/\!/(\CC^{\times\, l})$ and, to the right of
this and separated by a vertical line, a sequence of column vectors
representing the line bundles $L_i$.

\paragraph{A typical entry}
\label{sec:typical-entry-1}

For example, the entry for $\Xsurf{4}{4/3}$ shows that an example of a
surface $X$ with $k=4$ singularities and $K^2=\frac{4}{3}$ can be
constructed as a complete intersection of two general sections of the
line bundles $L_1=(2,4)$ and $L_2= (4,2)$ in the Fano simplicial toric
variety $F$ given by weight matrix:
\[
\begin{array}{cccccc}
x_0 & x_1 & x_2 & x_3 & x_4 & x_5  \\
\hline
1 & 2 & 2 & 1 & 1 & 0  \\
0 & 1 & 1 & 2 & 2 & 1 
\end{array}
\]
and $\Nef F=\langle (2,1),(1,2)\rangle$ (the notation is explained
fully in \S~\ref{sec:well-form-simpl} below). Here $\Lambda=L_1+L_2
\sim (6,6)$, $-(K_F+\Lambda)\sim (1,1)$ 
and $-K_F\sim (7,7)$ are all ample. 

\subsubsection{Model for $\Xsurf{1}{7/3}$ in $F=\wG(2,5)$}
\label{sec:model-wg-2}

We refer the reader to \cite{MR1954062} for the definition of weighted
Grassmannians, as well as notation for complete intersections in
them. Consider, as in \cite{MR1954062}, $F=\wG (2,5)$ with weights
$\bigl(\frac1{2}, \frac1{2}, \frac1{2}, \frac1{3},\frac1{3}\bigr)$:
then $F\subset \PP(1^3,2^6,3)$ and it is easy to see by the methods of
\cite{MR1954062} that the vanishing locus $X=Z(s)$ of a general
section $s\in \Gamma (F, \oo_F (2)^{\oplus 4})$ is a surface of this
family.

\subsubsection{Model for $\Xsurf{5}{5/3}$}
\label{sec:model-xsurf553}

This family is qG-rigid and it consists of a single surface. This
surface has a simple birational construction: blow up the vertices of
a pentagon of $5$ lines on a nonsingular del Pezzo surface of degree
$5$, and blow down the strict transforms of the $5$ lines. In
\S~\ref{sec:x_5-53} we construct a model of this surface as
codimension~3 degeneracy locus of an antisymmetric vector bundle
homomorphism $s\colon E\otimes L \to E^\vee$ where $E$ is a rank~5
split vector bundle on a simplicial toric 5-fold.

\subsubsection{Model for $\Xsurf{5}{2/3}$}
\label{sec:model-xsurf523}

Section~\ref{sec:x_5-23} is an extended essay on this case.

\subsubsection{Model for $\Xsurf{6}{2}$}
\label{sec:xsurf62}

This is the toric surface whose fan is the spanning fan of the lattice polygon
in Figure~\ref{fig:X62}:
\begin{figure}[H]
  \caption{The polygon of $\Xsurf{6}{2}$ \protect{\label{fig:X62}}}
  \centering
  \resizebox{4cm}{!}{\includegraphics[width=0.8\textwidth]{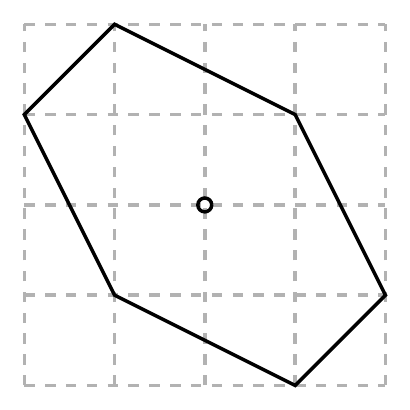}}
\end{figure}
\noindent where the origin of the lattice is labelled with a $\circ$. It is
clear from the picture that the primitive generators $\rho_1,\dots, \rho_6$ of the rays of
the fan generate a sublattice of index $3$. This fact easily implies
that $\pi_1(X^0)=\ZZ/3\ZZ$. 

\subsubsection{Model for $\Xsurf{6}{1}$}
\label{sec:xsurf61}

Every surface in this family is the quotient of a nonsingular cubic surface by a
$\Bmu_3$-action as shown by table~\ref{table:1}. This fact immediately
implies that $\pi_1(X^0)=\ZZ/3\ZZ$ for a surface in this family.

\subsection{The Fano index and \texorpdfstring{$\pi_1(X^0)$}{pi1}}
\label{sec:fano-index-pi_1x0}

\begin{pro}
  \phantomsection \label{pro:3}
  \begin{enumerate}[(a)]
  \item $\PP(1,1,3)$ has Fano index $f=5$. The surface
    $\Bsurf{1}{16/3}$ and every surface of the family
    $\Bsurf{2}{8/3}$ have Fano index $f=2$. All other surfaces with
    $k\geq 1$ $\frac1{3}(1,1)$ points have Fano index $f=1$.
  \item If $X=\Xsurf{6}{2}$ or a surface of the family $\Xsurf{6}{1}$,
    then $\pi_1(X^0)=\ZZ/3\ZZ$. If $X$ is any other surface with
    $k\geq 1$ $\frac1{3}(1,1)$ points, then $\pi_1(X^0)=(0)$.
  \end{enumerate}
\end{pro}

\begin{rem}
  \label{rem:12}
  Our proof of this fact, which we sketch below, is by ad hoc
  computations. For a quasi-smooth and well-formed complete
  intersection $X$ of ample (or nef, or irrelevant) line bundles in a
  well-formed simplicial toric variety $F$, it is natural to imagine
  that there would be some general Lefschetz type results relating
  $\Cl F$ to $\Cl X$ and $\pi_1(F^0)$ to $\pi_1(X^0)$. We could not
  find these results in the literature.
\end{rem}

\begin{proof} We give a sketch of the proof, leaving some of the
  details to the reader.  We start by proving (a). Let $f$ be as
  stated and $A$ the divisor class such that $-K_X=fA$, then we need
  to show that $A$ is primitive in $\Cl X$. It is clearly enough to
  show that the class of $A$ in $H^2(X^0,\ZZ)$ is primitive. In all
  cases, we check this by producing a compact curve $C\subset X^0$
  with $A\cdot C=1$ or a pair of compact curves $C_1$,
  $C_2\subset X^0$ of degrees $d_i=A\cdot C_i\in \NN$ such that
  $\hcf (d_1,d_2)=1$. If $X$ contains a floating $(-1)$-curve
  $C\subset X^0$ then $-K_X\cdot C=1$ and then clearly $f=1$: thus
  from now on we work with the $8$ families of surfaces of
  theorem~\ref{thm:4}. We refer to the constructions and pictures in
  the statement of theorem~\ref{thm:4}. The analysis is very simple
  and we go through each case in turn:

  \paragraph{Case $k=1$} For $X=\Bsurf{1}{16/3}$ $X$ the picture shows
  that $X$ has two rulings of rational curves $C$ with $C^2=0$ and
  $A\cdot C=1$ in both cases.

For $X=\PP(1,1,3)$ if $C$ is a general curve of
self-intersection $3$ then $A\cdot C=1$. 

\paragraph{Case $k=2$} For $X=\Bsurf{2}{8/3}$ $X$ the picture again
shows that $X$ has two rulings of rational curves $C$ with $C^2=0$ and
$A\cdot C=1$ in both cases.

For $X=\Xsurf{2}{17/3}$ the picture shows that $X^0$ contains a curve
$C_1$ with $C_1^2=1$ and a ruling $C_2$ with $C_2^2=0$, so $A\cdot
C_1=3$ and $A\cdot C_2=2$.

\paragraph{Case $k=3$} The surface $X=\Xsurf{3}{5}$ has a morphism
to $\PP(1,1,3)$ and if $C_1\subset X^0$ is the proper transform of a
general curve of self-intersection $3$ then $A\cdot C_1=5$. In
addition $X$ has a ruling $C_2$ with $C_2^2=0$ and $A\cdot C=2$.

\paragraph{Case $k=4$} $X=\Xsurf{4}{7/3}$ has morphisms to $\PP^1$ and
$\PP(1,1,3)$.

\paragraph{Case $k=5$} $X=\Xsurf{5}{5/3}$ has five morphisms to
$\PP^1$ and five morphisms to $\PP(1,1,3)$.

\paragraph{Case $k=6$} The toric surface $X=\Xsurf{6}{2}$ has
morphisms to $\PP^1$ and also to a toric Gorenstein cubic surface $Y$ with
$3\times A_2$ points, hence there are cuves $C_1$, $C_2\subset X^0$
with $A\cdot C_1=2$ and $A\cdot C_2=3$. 

\smallskip

This concludes the proof of (a). For the proof of (b) we need to equip
ourselves with a topological model of $X$. It is clear that if
$X\to X_1$ is the blow-down of a floating $(-1)$-curve, then
$\pi_1(X^0)=\pi_1(X_1^0)$, so again we only need to consider the $8$
families of theorem~\ref{thm:4}.

  To get a topological model of $X^0$ for a surface $X$ of one of these
  $8$ families, we use the toric qG-degenerations of
  \S~\ref{sec:toric-degenerations}: a surface $X$
  qG-degenerates to a toric surface $X_0$ with fan the face-fan of the
  corresponding polygon in table~\ref{tab:third_one_one_polygons},
  pictured in figure~\ref{fig:third_one_one_polygons}.
  Denote by $\widehat{X}_0$ the maximal crepant partial
  resolution of $X_0$. From a direct inspection of the polygons we see
  that, with the exception of
  $\Xsurf{5}{5/3}$, $\widehat{X}_0$ has $k$ $\frac1{3}(1,1)$ points
  $P_i$ ($i=1,\dots, k$) and is elsewhere nonsingular. We treat the
  case of $\Xsurf{5}{5/3}$ separately below. In all other cases $X^0$
  is diffeomorphic to the toric surface $\widehat{X}_0^0$ and we
  determine $\pi_1(X^0)=\pi_1(\widehat{X}_0^0)$ by well-known toric
  methods. 

  The surface $\Xsurf{5}{5/3}$ degenerates to the toric surface $X_0$
  corresponding to polygon $7$: in this case $\widehat{X}_0=X_0$ has
  two $T$ singularities $Q_1$, $Q_2$ of type $\frac1{4}(1,1)$.  In
  this case $X^0$ is obtained as a topological space by attaching to
  $\widehat{X}_0^0$ the Milnor fibres of the singularities $Q_1$,
  $Q_2$ along their boundaries. The result follows from (i) the fact
  that $\pi_1(\widehat{X}_0^0)=(0)$ and (ii) a calculation of
  $\pi_1(X^0)$ from $\pi_1(\widehat{X}_0^0)$. Here (i) holds because
  we can see by direct inspection that the rays of the fan of
  $\widehat{X}_0^0$ generate $N$ as a group. As for (ii), it is stated
  in \cite[Proposition~13]{MR1116920} that if $F$ is the Milnor fibre
  of a qG-smoothing of a primitive class $T$ singularity
  $\frac1{r^2}(1,ar-1)$ with $\hcf(r,a)=1$, then $\pi_1(F)=\ZZ/r\ZZ$
  and $\pi_1(\partial F)=\ZZ/r^2\ZZ$ and the natural homomorphism
  $\pi_1(\partial F)\to \pi_1(F)$ is surjective (we only need the
  $r=2$ case of this). By two applications of the Seifert--van Kampen
  theorem, it follows in this case that $\pi_1(X^0)=(0)$.
\end{proof}

\begin{center}
\begin{longtable}{cccccc}
\caption{del Pezzo surfaces with $1/3(1,1)$ and $f>1$} \label{table:2} \\
\toprule
\multicolumn{1}{c}{Name}&\multicolumn{1}{c}{$h^0(X,-K_X)$}&\multicolumn{1}{c}{$r$}&
       \multicolumn{1}{c}{No. moduli}   &\multicolumn{1}{c}{Model Construction}&\multicolumn{1}{c}{$f$}\\
\midrule
\endfirsthead
\multicolumn{4}{l}{\tiny Continued from previous page.}\\
\addlinespace[1.7ex]
\midrule
\multicolumn{1}{c}{Name}&\multicolumn{1}{c}{$h^0(X,-K_X)$}&\multicolumn{1}{c}{$r$}
&\multicolumn{1}{c}{No. moduli}&\multicolumn{1}{c}{Model
  Construction}&\multicolumn{1}{c}{$f$}\\
\midrule
\endhead
\midrule
\multicolumn{4}{r}{\tiny Continued on next page.}\\
\endfoot
\bottomrule
\endlastfoot
\oddrow $\Ssurf{1}{25/3}$&$9$&$2$&-8&\calulatepadding{$\PP(1,1,3)$}\padding& 5\\
\evnrow $\Bsurf{1}{16/3}$&$6$&$5$&-2& \calulatepadding{$X_4\subset
  \PP(1,1,1,3)$}\padding& 2\\
\oddrow $\Bsurf{2}{8/3}$&$3$&$8$&2&\calulatepadding{$X_6\subset
  \PP(1,1,3,3)$}\padding& 2\\
\addlinespace[1.1ex]
\bottomrule
\end{longtable}
\end{center}

\begin{center}
\setlength{\extrarowheight}{0.3em}
\begin{longtable}{ccccccc}
\caption{del Pezzo surfaces with $1/3(1,1)$ and $f=1$} \label{table:1} \\
\toprule
\multicolumn{1}{c}{Name}&\multicolumn{1}{c}{$h^0(X,-K_X)$}&\multicolumn{1}{c}{$r$}&\multicolumn{1}{c}{No.\ moduli}&\multicolumn{1}{c}{Weights and Line bundles}&\multicolumn{1}{c}{$\Nef F$}\\
\midrule
\endfirsthead
\multicolumn{4}{l}{\tiny Continued from previous page.}\\
\addlinespace[1.7ex]
\midrule
\multicolumn{1}{c}{Name}&\multicolumn{1}{c}{$h^0(X,-K_X)$}&\multicolumn{1}{c}{$r$}&\multicolumn{1}{c}{No. moduli}&\multicolumn{1}{c}{Weights and Line bundles}&\multicolumn{1}{c}{$\Nef F$}\\
\midrule
\endhead
\midrule
\multicolumn{4}{r}{\tiny Continued on next page.}\\
\endfoot
\bottomrule
\endlastfoot
\oddrow $\Xsurf{1}{22/3}$&$8$&$3$&$-6$&\vgap\calulatepadding{$\begin{array}{cccc}
1&1&2&0\\0&1&3&1\\\end{array}$}\padding&\calulatepadding{$\begin{array}{cc}
1&2\\1&3\\\end{array}$}\padding \\
\evnrow $\Xsurf{1}{19/3}$&$7$&$4$&$-4$&\vgap\calulatepadding{$\begin{array}{ccccc}
1&3&3&0&0\\0&2&1&1&0\\1&2&0&0&1\\\end{array}$}\padding&\calulatepadding{$\begin{array}{ccc}
3&0&0\\2&1&0\\2&0&1\end{array}$}\padding\\
\oddrow $\Xsurf{1}{16/3}$&$6$&$5$&$-2$&\vgap \calulatepadding{$\begin{array}{ccccc|c}
1&1&0&0&0&1\\0&0&1&1&3&3\\\end{array}$}\padding&\calulatepadding{$\begin{array}{cc}
1&0\\0&1\\\end{array}$}\padding\\
\evnrow $\Xsurf{1}{13/3}$&$5$&$6$&$0$&\vgap \calulatepadding{$\begin{array}{ccccc|c}
1&1&3&1&0&4\\0&0&0&1&1&1\\\end{array}$}\padding&\calulatepadding{$\begin{array}{cc}
1&1\\0&1\\\end{array}$}\padding\\
\oddrow $\Xsurf{1}{10/3}$&$4$&$7$&$2$&\vgap\calulatepadding{$\begin{array}{cccccc|cc}
1&1&2&1&0&0&2&2\\0&0&1&2&1&1&2&2\\\end{array}$}\padding&\calulatepadding{$\begin{array}{cc}
2&1\\1&2\\\end{array}$}\padding\\
\evnrow
  $\Xsurf{1}{7/3}$&$3$&$8$&$4$&\vgap\calulatepadding{\begin{tabular}{c}
  $F=\wG (2,5)$ and $\oo_F(2)^{\oplus\,4}$\\where $w=\bigl(\frac1{2},\frac1{2},\frac1{2},\frac3{2},\frac3{2}\bigr)$\end{tabular}}\padding&\calulatepadding{$\begin{array}{c}
1\\\end{array}$\padding}\\
\oddrow $\Xsurf{1}{4/3}$&$2$&$9$&$6$&\vgap\calulatepadding{$\begin{array}{ccccc|cc}
1&1&2&2&3&4&4\\\end{array}$}\padding&\calulatepadding{$\begin{array}{c}
1\\\end{array}$\padding}\\
\evnrow $\Xsurf{1}{1/3}$&$1$&$10$&$8$&\vgap\calulatepadding{$\begin{array}{cccc|c}
1&2&3&5&10\\\end{array}$}\padding&\calulatepadding{$\begin{array}{c}
1\\\end{array}$}\padding\\
\oddrow $\Xsurf{2}{17/3}$&$6$&$5$&-4&\vgap\calulatepadding{$\begin{array}{ccccc|c}
1&1&2&3&0&4\\-1&0&1&3&1&2\\\end{array}$}\padding&\calulatepadding{$\begin{array}{cc}
2&1\\1&1\\\end{array}$}\padding\\
\evnrow $\Xsurf{2}{14/3}$&$5$&$6$&-2&\vgap\calulatepadding{$\begin{array}{ccccc|c}
1&1&0&0&-1&0\\0&1&1&3&1&4\\\end{array}$}\padding&\calulatepadding{$\begin{array}{cc}
1&0\\1&1\\\end{array}$}\padding\\
\oddrow $\Xsurf{2}{11/3}$&$4$&$7$&0&\vgap\calulatepadding{$\begin{array}{cccccc|c}
1&0&0&1&0&1&2\\0&1&0&0&1&1&2\\0&0&1&1&1&4&4\\\end{array}$}\padding&\calulatepadding{$\begin{array}{ccc}
3&1&1\\1&3&1\\4&4&4\\\end{array}$}\padding\\
\evnrow $\Xsurf{2}{8/3}$&$3$&$8$&2&\vgap\calulatepadding{$\begin{array}{ccccc|c}
1&1&1&1&0&3\\0&0&1&3&1&3\\\end{array}$}\padding&\calulatepadding{$\begin{array}{cc}
1&1\\1&3\\\end{array}$}\padding\\
\oddrow $\Xsurf{2}{5/3}$&$2$&$9$&4&\vgap\calulatepadding{$\begin{array}{ccccc|c}
1&1&2&1&0&4\\0&1&3&3&1&6\\\end{array}$}\padding&\calulatepadding{$\begin{array}{cc}
2&1\\3&3\\\end{array}$}\padding\\
\evnrow $\Xsurf{2}{2/3}$&$1$&$10$& 6&\vgap\calulatepadding{$\begin{array}{ccccc|cc}
1&2&2&3&3&4&6\\\end{array}$}\padding&\calulatepadding{$\begin{array}{c}
1\\\end{array}$}\padding\\
\oddrow $\Xsurf{3}{5}$&$5$&$6$&-4&\vgap\calulatepadding{$\begin{array}{ccccc}
1&0&0&1&0\\0&1&-1&1&0\\0&0&1&1&3\\\end{array}$}\padding&\calulatepadding{$\begin{array}{ccc}
1&1&0\\1&0&0\\1&1&1\\\end{array}$}\padding\\
\evnrow $\Xsurf{3}{4}$&$4$&$7$&-2&\vgap\calulatepadding{$\begin{array}{cccccc}
1&-1&1&0&0&0\\-1&1&0&0&0&1\\2&1&0&1&0&0\\1&2&0&0&1&0\\\end{array}$}\padding&\calulatepadding{$\begin{array}{ccccccc}
0&1&0&1&0&1&2\\0&0&1&0&1&2&1\\1&2&1&2&2&2&2\\1&1&2&2&2&4&4\\\end{array}$}\padding\\
\oddrow $\Xsurf{3}{3}$&$3$&$8$&0&\vgap\calulatepadding{$\begin{array}{ccccc|c}
1&1&1&0&0&2\\0&0&1&1&3&3\\\end{array}$}\padding&\calulatepadding{$\begin{array}{cc}
1&0\\1&1\\\end{array}$}\padding\\
\evnrow $\Xsurf{3}{2}$&$2$&$9$&2&\vgap\calulatepadding{$\begin{array}{ccccc|c}
1&3&2&0&-1&4\\0&0&1&1&1&2\\\end{array}$}\padding&\calulatepadding{$\begin{array}{cc}
2&0\\1&1\\\end{array}$}\padding\\
\oddrow $\Xsurf{3}{1}$&$1$&$10$&4&\vgap\calulatepadding{$\begin{array}{cccccc|c}
1&0&0&2&1&1&4\\0&1&0&1&2&1&4\\0&0&1&1&1&2&4\\\end{array}$}\padding&\calulatepadding{$\begin{array}{ccc}
2&1&1\\1&2&1\\1&1&2\\\end{array}$}\padding\\
\evnrow $\Xsurf{4}{7/3}$&$2$&$9$&0&\vgap\calulatepadding{$\begin{array}{ccccc|c}
1&0&0&-1&-1&0\\0&3&3&2&1&6\\\end{array}$}\padding&\calulatepadding{$\begin{array}{cc}
0&-1\\1&2\\\end{array}$}\padding\\
\oddrow $\Xsurf{4}{4/3}$&$1$&$10$&2&\vgap\calulatepadding{$\begin{array}{cccccccc|cc}
1&0&0&0&2&1&1&1&2&3\\0&1&0&0&1&2&1&1&2&3\\0&0&1&0&1&1&2&1&3&2\\0&0&0&1&1&1&1&2&3&2\\\end{array}$}\padding&\calulatepadding{$\begin{array}{cccc}
2&1&1&1\\1&2&1&1\\1&1&2&1\\1&1&1&2\\\end{array}$}\padding\\
\evnrow $\Xsurf{4}{1/3}$&$0$&$11$&4&\vgap\calulatepadding{$\begin{array}{ccccc|cc}
2&2&3&3&3&6&6\\\end{array}$}\padding&\calulatepadding{$\begin{array}{c}
1\\\end{array}$}\padding\\
\oddrow $\Xsurf{5}{5/3}$&$1$&$10$&0&\vgap\calulatepadding{\begin{tabular}{c}$F$ and
                                     $D(s)$ where \\ $s\colon E\otimes L
                                     \to E^\vee$ as in \S~\ref{sec:x_5-53}\end{tabular}}\padding&\calulatepadding{}\padding\\
\evnrow $\Xsurf{5}{2/3}$&$0$&$11$&2&\vgap\calulatepadding{$F$ and $L^{\oplus
                                     2}$ as in \S~\ref{sec:x_5-23}}\padding&\\
\oddrow
  $\Xsurf{6}{2}$&$1$&$10$&-2&\vgap\calulatepadding{\begin{tabular}{c}
                                                     $F/\Bmu_3$ (see
                                                     \S~\ref{sec:xsurf62})
                                                     where $F$ has
                                                     weights\\ 
$\begin{array}{cccccc}
1&-1&1&0&0&0\\1&0&0&1&0&0\\0&1&0&0&1&0\\-1&1&0&0&0&1\end{array}$\end{tabular}}\padding&\calulatepadding{$\begin{array}{ccccc}1&0&0&1&0\\1&1&0&1&1\\0&1&1&1&1\\ 
0&0&1&0&1\\\end{array}$}\padding\\
\evnrow $\Xsurf{6}{1}$&$0$&$11$&0&\vgap\calulatepadding{\begin{tabular}{c}
 $\PP^3/\Bmu_3$ and $\oo(3)$ where \\ $\Bmu_3$ acts with weights $\frac1{3}, \frac1{3}, \frac{2}{3}, \frac{2}{3}$\end{tabular}}\padding&\\
\addlinespace[1.1ex]
\bottomrule
\end{longtable}
\end{center}


\section{Constructions}
\label{sec:constructions}


This section is organised as follows: In \S~\ref{sec:well-form-simpl}
we recall basic facts about toric varieties and well-formed complete
intersections in them; in \S~\ref{sec:sample-computations} we give an
extended example of the computations needed to see that the
constructions given in table~\ref{table:1} really do what they say
they do. \S~\ref{sec:x_5-53} and \S~\ref{sec:x_5-23} give good model
constructions of the surface $\Xsurf{5}{5/3}$ and of a surface in the
family $\Xsurf{5}{2/3}$. The final \S~\ref{sec:birat-constr-4}
contains some Italian-style birational constructions in certain
special cases.

\subsection{Simplicial toric varieties and well-formed complete
  intersections in them}
\label{sec:well-form-simpl}

As usual $\TT\cong \CC^{\times\, d}$ is a $d$-dimensional torus,
$M=\Hom (\TT,\CC^\times)$ is the group of characters, and $~{N=\Hom (M,\ZZ)}$.

\paragraph{From a fan to a GIT quotient}
\label{sec:from-fan-git}

We recall how to interpret a toric variety as a GIT quotient. To a
rational fan $\Sigma$ in $N_\RR$ one associates a toric variety
$X_\Sigma$. The toric variety is proper if and only if $\Sigma$ is a
complete fan, that is, the support $|\Sigma|$ of the fan is all of
$N_\RR$; the variety is called simplicial if and only if all cones of
the fan are simplicial cones. All toric varieties in this section are
proper (and, in fact, projective, see below) and simplicial.

A simplicial toric variety $X$ is in a natural way a smooth
Deligne-Mumford stack, called the canonical stack of a simplicial
toric variety in \cite{MR2774310}. Indeed, $X$ is the union of affine
open subsets $X_\sigma=\Spec k[M\cap \sigma^\vee]$, where
$\sigma\subset N$ is a simplex of maximal dimension. Denote by
$\rho_1, \dots, \rho_n\in N$ the primitive generators of the rays of
$\sigma$ (a ray is a $1$-dimensional cone), and let
$N_\sigma=\sum_{i=1}^n\ZZ\rho_i$, then the dual lattice
$M_\sigma=\Hom (N_\sigma , \ZZ)$ can be constructed as the
over-lattice
\[
M_\sigma=\{m\in M\otimes \QQ\mid \langle m, n \rangle\in \ZZ\;\forall
\; n\in N_\sigma\}
\]
The group scheme $\Bmu$ with character group the finite group
$N/N_\sigma$ acts on $\CC[M_\sigma\cap \sigma^\vee]=\CC[\NN^r]$ with
ring of invariants $\CC[M\cap \sigma^\vee]$ and this shows that
$X_\sigma$ is the quotient of the smooth variety $\Spec k[M_\sigma
\cap \sigma^\vee]$ by $\Bmu$. 
These charts give $X_\Sigma$ the structure of a smooth
Deligne--Mumford stack with the following two properties:
\begin{enumerate}
\item the stabilizer of the generic point is trivial, that is,
  $X_\Sigma$ is an orbifold;
\item the stabilizers of all codimension $1$ points are also trivial.
\end{enumerate}

Next we see how to interpret $X_\Sigma$ as a global GIT quotient.
Denote by $\rho_1,\dots, \rho_m\in N$ the primitive generators of the
rays of the fan. Here and below we
assume for simplicity that the $\rho_i$ generate $N$ as a group. Then
we have an exact sequence:
\[
(0) \to \LL \to \ZZ^m\overset{\rho}{\to} N \to (0)
\]
and a dual exact sequence\footnote{If the map $\rho$ is not
  surjective but has finite cokernel, then one must work instead with the Gale
dual sequence, cf.\ \cite{MR2114820}.}:
\[
(0) \to M \to \ZZ^{\ast\,m}\overset{D}{\to} \LL^\ast\to (0)
\]
Note that the homomorphism $\rho$ is not enough to reconstruct the fan
$\Sigma$. Below we identify a simplex $~{\sigma=\langle \rho_{i_1},\dots
\rho_{i_k}\rangle}$ of $\Sigma$ with the list of indices $[i_1,\dots, i_k]$. 
\begin{dfn}
  \label{dfn:1}
  It is well-known~\cite{MR1299003} that $\LL^\ast =\Cl X_\Sigma$ is
  the divisor class group of $X_\Sigma$. We call $D$ the \emph{divisor
    homomorphism} of the toric variety.
\end{dfn}
The homomorphism $D\colon \ZZ^{\ast\,m} \to \LL^\ast$ is dual to a
group homomorphism $\GG\to \CC^{\times \, m}$ where $\GG$ is the torus
with character group $\LL^\ast$; $\GG$ acts on $\CC^m$ via this
homomorphism and as we next explain in greater detail $X_\Sigma$ is a
GIT quotient $\CC^m/\!/\GG$. Identify $\LL^\ast$ with the group of
$\GG$-linearised line bundles on $\CC^n$. An element
$\omega \in \LL^\ast$ is also called a stability condition. The choice
of a stability condition $\omega\in \LL^\ast$ determines a GIT
quotient $\CC^m/\!/_\omega \GG$ and we need to state what choices of
$\omega$ produce $X_\Sigma$.  Denoting by $x_i$ the standard basis of
$\ZZ^{\ast\, m}$, and writing $D_i=D(x_i)\in \LL^\ast$, it is
well-known that the nef cone of $X_\Sigma$ is:
\[
\Nef X_\Sigma = \bigcap_{\sigma\in \Sigma}\langle D_i\mid i \not \in
\sigma \rangle
\] 
In this paper we always assume that $X_\Sigma$ is projective; in other
words, $\Nef X_\Sigma$ has nonempty interior $\Amp X_\Sigma$ (the ample
cone).  Then, for all $\omega\in \Amp X_\Sigma$,
we have that $X_\Sigma=\CC^n/\!/_\omega \GG$.

In fact, we can be more explicit than this. Thinking of $x_i$ as a
coordinate function on $\CC^m$, define the \emph{irrelevant ideal} as:
\[
\Irr_\Sigma=\Bigl(\prod_{i\not \in \sigma} x_i \mid \sigma \in
\Sigma\Bigr)\subset \CC[x_1,\dots, x_m]
\]
and let $Z_\Sigma=V(\Irr_\Sigma)\subset \CC^m$ the variety of
$I$. Then for all $\omega\in \Amp X_\Sigma$, $Z_\Sigma\subset \CC^m$ is the
set of $\omega$-unstable points, $\CC^m\setminus Z_\Sigma$ is the set of
stable points (there are no strictly semi-stable points in this situation), and
\[
X_\Sigma =\CC^m/\!/_\omega \GG = (\CC^m\setminus Z_\Sigma) /\GG
\]
is a ``bona fide'' quotient. 

The action of $\GG$ on $\CC^m \setminus Z_\Sigma$ has finite stabilizers;
this naturally endows the quotient $(\CC^m\setminus Z_\Sigma)/\GG$ with a structure
of a smooth Deligne-Mumford stack which we denote by $[(\CC^m
\setminus Z_\Sigma)/\GG]$; this stack structure is the same as the canonical stack
structure on $X_\Sigma$. Below we describe in detail an atlas of charts.

\paragraph{From a GIT quotient to a fan}
\label{sec:from-git-quotients}

Conversely, consider a rank $r$ lattice $\LL^\ast\cong \ZZ^r$ and
denote by $\GG$ the torus with character group $\LL^\ast$. Consider
now $\ZZ^{\ast\, m}$, denote by $x_i$ the standard basis elements, let
$D\colon \ZZ^{\ast \, m}\to \LL^\ast$ be a group homomorphism such
that the $D_i=D(x_i)$ span a strictly convex cone $\mathcal{C}\subset
\LL^\ast_\RR$. $D$ dualises to a group homomorphism $\GG \to
\CC^{\times \, m}$ and hence $\GG$ acts on $\CC^m$. 
\begin{dfnrem}
  \label{rem:2}
  It is easy to see~\cite{A13} that:
  \begin{enumerate}[(1)]
  \item Choose a basis of $\LL^\ast \cong \ZZ^r$ and identify $D$ with
    a $r\times m$ matrix, which we call the \emph{weight
      matrix}. $\GG$ acts faithfully if and only if the rows of $D$
    span a saturated sublattice of $\ZZ^r$, if and only if the $\hcf$
    of all the $r \times r$ minors of $D$ is $1$. A matrix satisfying this
    condition is called \emph{standard}.
  \item $\GG$ acts faithfully on the divisor $D_i=(x_i=0)\subset
    \CC^m$ if and only if the matrix $D_{\,\widehat{i}}=(D_1\dots,
    \widehat{D}_i,\dots, D_m)$ obtained from $D$ by removing the
    $i$-th column, is standard.
  \end{enumerate}
\end{dfnrem}

\begin{dfn}
  \label{dfn:2}
The homomorphism $D\colon \ZZ^{\ast\, m}\to \LL^\ast$ is
\emph{well-formed} if both the weight matrix $D$ and the $D_{\,\widehat{i}}$ for
all $i=1, \dots, m$ are standard.
\end{dfn}

\begin{rem}
  \label{rem:3}
  We can take GIT quotients for any $D$; however, if $D$ is the
  divisor homomorphism of some toric variety $X_\Sigma$, then $D$ is
  well-formed. The aim of the considerations that follow is precisely to state that
  the converse is also true.
\end{rem}

Given a stability condition $\omega \in \LL^\ast_\RR$, we can form the
GIT quotient
\[
X_\omega:=\CC^m/\!/_\omega \GG
\]
There is a wall-and-chamber decomposition of $\mathcal{C} \subset
\LL^\ast_\RR$, called the secondary fan, and if stability conditions $\omega_1$, $\omega_2$ lie
in the same chamber then the GIT quotients $X_{\omega_1}$,
$X_{\omega_2}$ coincide. In more detail, the walls of the
decomposition are the cones
of the form $\langle D_{i_1}, \dots, D_{i_k}\rangle\subset \LL^\ast_\RR$ that have
codimension one. The chambers are the connected components of the
complement of the union of all the walls; these are $r$-dimensional
open cones in $\mathcal{C}$. By construction, a chamber is the
intersection of the interiors of the simplicial $r$-dimensional cones $\langle D_{i_1},
\dots, D_{i_r}\rangle\subset \LL^\ast_\RR$ that contain it. 
Choose now a chamber, and pick a stability condition $\omega$ in
it. Given such an $\omega$, the irrelevant ideal $I_\omega \subset
\CC[x_1,\dots, x_m]$ is
\[
I_\omega = \Bigl(x_{i_1}\cdots x_{i_r}\mid \omega \in \langle D_{i_1}, \dots, D_{i_r}\rangle\Bigr)
\]
the unstable locus is $Z_\omega=V(I_\omega)$; and the GIT quotient is
the bona fide quotient
\[
X_\omega = (\CC^m\setminus Z_\omega)/\GG \,.
\]
Note that $I_\omega$, $Z_\omega$ and the quotient $X_\omega$ depend
only on the chamber that $\omega$ sits in and not on $\omega$ itself. Given such
an $\omega$ we can also form a simplicial fan $\Sigma$ where
\[
\sigma\in \Sigma
\quad 
\text{if and only if}
\quad
\omega \in \langle D_i \mid i \not \in \sigma \rangle 
\]
and $\Sigma$ also depends only on the chamber that $\omega$ sits
in. The following Proposition collects a few well-known facts that can
easily be synthesized from the literature \cite{MR1299003, MR2114820,
  MR2774310, A13}:
\begin{pro}
  \label{pro:1}
 Let $D\colon \ZZ^{\ast\,m}\to \LL^\ast$ be a homomorphism as above,
 choose a chamber in the secondary fan as described above, and let $\omega$ be a
 stability condition in it. Let $I_\omega$, $Z_\omega$, $X_\omega$,
 $\Sigma$, be as above:
\begin{enumerate}[(1)]
\item If (as we are assuming) $\mathcal{C}$ is a strictly convex cone,
  then $\Sigma$ is a complete fan;
\item If $D$ is well-formed, then $X_\omega = [(\CC^m\setminus
  Z_\omega)/\GG]$ is isomorphic to $X_\Sigma$ as a Deligne--Mumford
  stack, $D$ is the divisor homomorphism of the toric variety
  $X_\Sigma$, and the chosen chamber is $\Amp X_\Sigma$;
\item If $D$ is not well-formed, there is a natural non-representable morphism of
  Deligne--Mumford stacks $X_\omega =[(\CC^m\setminus
  Z_\omega)/\GG]\to X_\Sigma$ and the two stacks have the same coarse
  moduli space.\footnote{In this case the chamber canonically
    determines a stacky fan in the
    sense of \cite{MR2114820}, such that $X_\omega$ is isomorphic to
    the toric Deligne--Mumford stack associated to the stacky fan. We
    don't need this construction in this paper.} 
\item In general, if $D$ is not well-formed, \cite{A13} describes an
  algorithm to construct a well-formed $D^\prime$ such that $D^\prime$
  is the divisor homomorphism of $X_\Sigma$.
\end{enumerate}
\end{pro}

\paragraph{Charts on GIT quotients}
\label{sec:charts-simpl-toric}

We explain how to set up an explicit atlas of charts on
$X_\omega=[(\CC^m\setminus Z_\omega)/\GG]$, which we use repeatedly in
the calculations needed to validate the entries of
table~\ref{table:1}. Fix a well-formed $D\colon \ZZ^{\ast \, m}\to
\LL^\ast$, choose a basis of $\LL^\ast$, identify $D$ with an integral
$r\times m$ matrix. We have that $\CC^m\setminus Z_\omega$ is a union
of $\GG$-invariant open subsets:
\[
\CC^m\setminus Z_\omega = \bigcup_{\{(i_1,\dots,i_r)\mid \omega \in \langle
D_{i_1},\dots,D_{i_r}\rangle\}} U_{i_1,\dots,i_r}
\quad
\text{where}
\quad
U_{i_1,\dots,i_r}=\{x_{i_1}\neq 0,\dots, x_{i_r}\neq 0\}\subset \CC^m
\]
Let now $V_{i_1,\dots,i_r}=\{x_{i_1}=\cdots = x_{i_r} =1\}\subset
\CC^m$, then
$[U_{i_1,\dots,i_r}/\GG]=[V_{i_1,\dots,i_r}/\Bmu]$ where $\Bmu$ is the finite subgroup
of $\GG$ that fixes $V_{i_1,\dots,i_r}$.  Concretely, $\Bmu$ is the
finite group with character group $A$, the cokernel of the
homomorphism:
\[
D_{i_1,\dots,i_r}=(D_{i_1},\dots,D_{i_r})\colon \ZZ^{\ast\,r} \to \LL^\ast
\]

\paragraph{Complete intersections in toric varieties}
\label{sec:compl-inters-toric}

Consider a well-formed $D\colon \ZZ^{\ast \, m}\to \LL^\ast$ as
above. Fix a chamber of the secondary fan, a stability condition in
it, and let $F=X_\Sigma$ be the corresponding simplicial toric
variety. (In this paper $F$ is always Fano, and we assume
that $\omega =D_1+\dots+D_m$ is the anticanonical divisor of $F$. This
assumption however is irrelevant for the present discussion.) We
consider complete intersections $X\subset F$ of general elements of
linear systems $|L_1|, \dots, |L_c|$ where $L_i\in \LL^\ast$ are
$\GG$-linearised line bundles, that is, line bundles on
$X_\Sigma$.\footnote{More precisely, line bundles on the canonical DM
  stack of $X_\Sigma$.} The space of sections $H^0(F, L_i)$ is the vector subspace
of $\CC[x_1,\dots,x_m]$ with basis consisting
of monomials $x^v\in \CC[x_1,\dots,x_m]$ where $v\in \ZZ^{\ast\,
  m}$ has homogeneity type $L_i$, that is $D(v)=L_i$. Let $f_i\in
H^0(F, L_i)$, then $V(f_1,\dots, f_c)$ is stable under the action of
$\GG$, and we consider the subvariety $X=(V(f_1,\dots, f_c)\setminus
Z_\omega)/\GG\subset F$:
\begin{dfn}
  \label{dfn:3}
  \begin{enumerate}
  \item $X\subset F$ is \emph{quasi-smooth} if either: $V(f_1,\dots, f_c)\subset
    Z_\omega$, or 
\[
V(f_1,\dots,f_c)\setminus Z_\omega \subset \CC^m\setminus Z_\omega 
\]
 is a smooth subvariety of codimension $c$;
\item Suppose that $X\subset F$ is quasi-smooth. We say that $X$ is
  \emph{well-formed} if the following holds: For all toric strata $S\subset
  F$ with nontrivial stabilizer, $S\subset X$ implies $\codim_X S\geq
  2$.
  \end{enumerate}
\end{dfn}

\subsection{Sample computations}
\label{sec:sample-computations}

In the column labelled ``Weights and Line bundles,'' all lines of
table~\ref{table:1}, except those corresponding to families
$\Xsurf{1}{7/3}$, $\Xsurf{5}{5/3}$, $\Xsurf{5}{2/3}$, $\Xsurf{6}{2}$
and $\Xsurf{6}{1}$, list a well-formed weight matrix
\[
D\colon \ZZ^{\ast \, m}\to \LL^\ast = \ZZ^r
\]
for constructing a simplicial toric variety $F$ and, to the right
of it and separated by a vertical line, a list of column vectors
$L_i\in \ZZ^r$, representing line bundles on $F$ such that $X$ is a
complete intersection of general members of the $|L_i|$. The last
column is a list of column vectors in $\ZZ^r$, the generators of
$\Nef F$, which is the chamber of the secondary fan that
contains the stability conditions that give $F$ as GIT quotient. In
all cases it is immediate to verify that the $L_i\in \Nef F$ and that
$-K_F-\Lambda\in \Amp F$ where $\Lambda=\sum L_i$. In particular it
follows from this that $X$ is a Fano variety.

\subsubsection*{Example: family $\Xsurf{1}{10/3}$}
\label{sec:xsurf1,10/3}

As stated in corollary~\ref{cor:1}, a surface $X$ in this family is
either: (i) The blow-up of $\PP(1,1,3)$ at $d=5$ general points; or
(equivalently) (ii) The blow-up of $\Bsurf{1}{16/3}$ at $d=2$ general
points.

According to the table, a surface in this family can be constructed as
a codimension $2$ complete intersection of type $L_1=(2,2)$,
$L_2=(2,2)$ in the (manifestly well-formed) simplicial toric variety
$F$ with weight matrix:
\[
\begin{array}{cccccc}
x_0 & x_1 & x_2 & x_3 & x_4 & x_5 \\
\hline
1 & 1 & 2 & 1 & 0 & 0   \\
0 & 0 & 1 & 2 & 1 & 1  
\end{array}
\]
and $\Nef F=\langle L+2M,2L+M \rangle$, where $L=(1,0)$ and $M=(0,1)$
are the standard basis vectors of $\LL^\ast$. Note that both $L_1$, $L_2$ are ample, 
and $-(K_F+L_1+L_2)\sim L+M$ is ample.

First we examine all the charts of $F$ and verify that $X$ is a
quasi-smooth well-formed complete intersection with
$1\times \frac1{3}(1,1)$ singularities. Finally we calculate
$K_X^2=10/3$.

The chamber is $\langle x_2,x_3\rangle$, so the irrelevant ideal for the given stability condition is
$\Irr=(x_0,x_1,x_2)(x_3,x_4,x_5)$, and the charts are the $U_{ij}$
with $i\leq 2$, $3\leq j$. 

Let us first look at the chart $U_{03}=\{x_0\neq 0, x_3\neq
0\}$. Considering $V_{03}=\{x_0=x_3=1\}\subset \CC^6$ it
is immediate that:
\[
U_{03}=\frac1{2}(0,1,1,1)_{x_1,x_2,x_4,x_5}
\]
the quotient of $V_{03}\cong \CC^4$ with coordinates $x_1,x_2,x_4,x_5$ by
the action of $\Bmu_2$ with weights $(0,1,1,1)$. We see that the
$x_1$-axis $C$ is a curve toric stratum of the \mbox{4-fold} $F$ with stabilizer
$\Bmu_2$ at the generic point. We claim that $C\cap X=
\emptyset$. Indeed $C=\{x_2=x_4=x_5=0\}$ is the toric variety with
weight matrix:
\[
\begin{array}{ccc}
x_0 & x_1 & x_3 \\
\hline
1 & 1 & 1 \\
0 & 0 & 2   
\end{array}
\]
Note, however, that this matrix is not well-formed. Applying the
algorithm in \cite{A13}, we see that $C$, together with the line
bundles $L_{1|C}$, $L_{2|C}$, is the toric variety with well-formed
weight matrix 
\[
\begin{array}{ccc}
x_0 & x_1 & x_3 \\
\hline
1 & 1 & 0 \\
0 & 0 & 1   
\end{array}
\]
and line bundles $L_{1|C}=L_{2|C}=(1,1)$, which is manifestly
the same as $\PP^1$ with $L_{1|C}=L_{2|C}=\oo(1)$. It is clear that
the two restriction maps $H^0(F, L_i)=\langle x_0x_3, x_1x_3\rangle\to
H^0(C, L_{i|C})$ are surjective and thus two general members of $L_1$
and $L_2$ do not intersect anywhere on $C$.

The chart $U_{13}$ is very similar; and the charts $U_{04}$, $U_{05}$,
$U_{14}$, $U_{15}$ are smooth and it is immediate that none of the
strata passing through those charts are contained in the base locus of
$|L_i|$; thus, we only need to look at $U_{23}$.

Considering $V_{23}=\{x_2=x_3=1\}\subset \CC^6$ it
is easy to see that:
\[
U_{23}=\frac1{3}(1,1,1,1)_{x_0,x_1,x_4,x_5}
\]
the quotient of $V_{23}\cong \CC^4$ with coordinates $x_0,x_1,x_4,x_5$ by
the action of $\Bmu_3$ with weights $(1,1,1,1)$. Denote by $f_i\in
H^0(F, L_i)$ general members: the monomials $x_0x_3, x_1x_3,
x_2x_4, x_3x_4$ all appear in $f_i$ with nonzero coefficient, thus the
surface $X$ must contain the origin of this chart, it is quasi-smooth
there, and it has a singularity $1/3(1,1)$ there. This completes the
verification that $X$ is well-formed and has $1\times 1/3(1,1)$
singularities.

We now compute the degree of $X$. The Chow ring of $F$ is generated by
$L=(1,0)$ and $M=(0,1)$ with the relations $L^2(2L+M)=0$,
$(L+2M)M^2=0$ (corresponding to the components $(x_0,x_1,x_2)$,
$(x_3,x_4,x_5)$ of $\Irr$), and, for example, $L^2M^2=1/3$ obtained by looking
at the chart $U_{23}$. From this information, we get that
$L^3M=-(1/2)L^2M^2=-1/6$ and $L^4=-(1/2)L^3M=1/12$ and similarly
$M^4=1/2$, $M^3L=-1/6$ and then it is easy to compute:
\[
 K_X^2=L_1L_2(-K_F-L_1-L_2)^2=(2L+2M)^2(L+M)^2=4(L+M)^4= \\
4\Bigl( \frac1{12}-\frac{4}{6}+\frac{6}{3}-\frac{4}{6}+\frac1{12}\Bigr)=\frac{10}{3} 
\]

\subsection{The surface  \texorpdfstring{$X_{5, 5/3}$}{X(5,5/3)}}
\label{sec:x_5-53}

We construct this surface as the degeneracy locus of an antisymmetric
homomorphism $s\colon E \otimes L \to E^\vee $ where $E$ is a rank 5 (split)
homogeneous vector bundle, and $L$ a line bundle, on a simplicial
toric Fano variety $F$.

Specifically, $F$ is the Fano toric variety with Cox coordinates and weight matrix:
\[
\begin{array}{cccccccccc}
  y_1&y_2&y_3&y_4&y_5&x_1&x_2&x_3&x_4&x_5\\
\hline
1&0&0&0&0&2&1&1&1&1\\
0&1&0&0&0&1&2&1&1&1\\
0&0&1&0&0&1&1&2&1&1\\
0&0&0&1&0&1&1&1&2&1\\
0&0&0&0&1&1&1&1&1&2
\end{array}
\]
The Nef cone of $F$ is the simplicial cone generated by the last 5
vectors of the weight matrix. One can draw the secondary fan; $F$
is covered by 32 charts; it has isolated orbifold singularities
$5\times \frac{1}{2}(1,1,1,1,1)$, $10\times \frac{1}{3}(1,1,2,2,2)$,
$10\times \frac{1}{4}(1,1,1,3,3)$, $5\times \frac{1}{5}(1,1,1,1,4)$, $\frac{1}{6}(1,1,1,1,1)$.

Consider the following line bundles on $F$:
\[
L_1=\oo \begin{pmatrix}2\\2\\3\\2\\3\end{pmatrix},\;
L_2=\oo \begin{pmatrix}2\\3\\2\\2\\3 \end{pmatrix},\; 
L_3=\oo \begin{pmatrix}2\\3\\2\\3\\2 \end{pmatrix},\;
L_4=\oo \begin{pmatrix}3\\2\\2\\3\\2 \end{pmatrix},\;
L_5=\oo \begin{pmatrix}3\\2\\3\\2\\2 \end{pmatrix},\;
\]
and
\[L=\oo\begin{pmatrix}-6\\-6\\-6\\-6\\-6 \end{pmatrix}\]

\textbf{Claim} Writing $E=\oplus_{i=1}^5L_i$, $X\subset F$ is the
degeneracy locus of a general antisymmetric homomorphism
$s\colon E\otimes L \to E^\vee$.

We can take $s$ to be defined by the $5\times 5$ antisymmetric matrix
\[
A=
\begin{pmatrix}
 0&y_1^2y_2y_3y_4^2 & x_1 & x_2 & y_1y_2^2y_4^2y_5 \\
   &0                   &y_1^2y_3^2y_4y_5& x_3 & x_4 \\
   &                     & 0 &y_1y_2y_3^2y_5^2 & x_5\\
   &                     &     & 0                 &y_2^2y_3y_4y_5^2\\
   &                     &    &                     &0
\end{pmatrix}
\]
and the equations of $X$ are the five $4\times 4$ Pfaffians of the matrix
$A$:
\[
\begin{cases}
  -x_3x_5+ x_4y_1y_2y_3^2y_5^2+y_1^2y_2^2y_3^3y_4^2y_5^3=0\\
  -x_5x_2+ x_1y_2^2y_3y_4y_5^2+y_1^2y_2^3y_3^2y_4^2y_5^3=0\\
  -x_2x_4+ x_3y_1y_2^2y_4^2y_5 +y_1^2y_2^3y_3^2y_4^3y_5^2=0\\
  -x_4x_1+ x_5y_1^2y_2y_3y_4^2+y_1^3y_2^2y_3^2y_4^3y_5^2=0\\
  -x_1x_3+ x_2y_1^2y_3^2y_4y_5+y_1^3y_2^2y_3^3y_4^2y_5^2=0
\end{cases}
\]

One can check that $X$ given by these equations avoids all
the singularities of $F$ except $5$ of the points with
$\Bmu_3$-stabilizer, that $X$ is quasismooth and has $5\times \frac{1}{3}(1,1)$
singularities at those points, and that $X$ is nonsingular everywhere
else.

All relevant information about $X$, including the Poincar\'e series
and $K_X^2$, can be obtained from the resolution of $\oo_X$:
\[
(0)\to L \overset{\Pf^\vee}{\to} E\otimes L \overset{A}{\to} E^\vee \overset{\Pf}{\to}
\oo_F \to \oo_X\to  (0)
\]
In particular, the shape of the resolution shows that 
\[
-K_X=\oo_X\begin{pmatrix}1\\1\\1\\1\\1 \end{pmatrix}
\]

\subsection{The family \texorpdfstring{$X_{5, 2/3}$}{X(5,2/3)}}
\label{sec:x_5-23}

We construct a general surface $X$ in this family as a complete
intersection of two hypersurfaces in a $4$-dimensional non-simplicial
toric Fano variety $F$. In \S~\ref{sec:an-embedding-fsubset} we
construct the natural embedding $F\subset \PP(2^5,3^5)$.

\paragraph{Description of $F$} Let us describe the fan of $F$. The
rays of the fan of $F$ are generated by the $10$ vectors
$\rho_{ij}=e_i+e_j\in \ZZ^4$, $i=0,\ldots, 4$, where
$e_1, \dots, e_4\in \ZZ^4$ are the standard basis elements and
$e_0=-e_1-\ldots-e_4$. These vectors are the vertices of a strictly
convex $4$-dimensional lattice polytope $P$ with $10$ facets: $5$
tetrahedra and $5$ octahedra (draw a picture). The fan of $F$ is the
face-fan of $P$. Note that there is an obvious action of the symmetric
group $S_5$ on $F$. Thus $F$ is a nonsimplicial toric Fano
$4$-fold. It is clear from the construction (see below) that $F$ is
the blowdown of the $5$ coordinate divisors $(x_i=0)\subset \PP^4$ on
a blowup $G$ of $\PP^4$ with homogeneous coordinates $x_i$
($i=0,\dots, 4$), along the $10$ coordinate planes
$\Pi_{ij}=(x_i=x_i=0)$.

\paragraph{Description of the family $\Xsurf{5}{5/3}$} Denote by
$D_{ij}\subset F$ the Weil divisor corresponding to the ray
$\rho_{ij}\RR_{\geq 0}$. We will argue that the sheaf
\[L=\oo_F(D_{01}+\cdots +D_{04})
\]
is a line bundle (that is the Weil divisor $D=\sum_{i=1}^4 D_{0i}$ is
Cartier) and that $\Xsurf{5}{2/3}$ is a complete intersection of type
$L^{\oplus \, 2}$ on $F$. 

\smallskip

To verify all of these statements, we examine the local picture at all
the toric charts. Up to $S_5$, there are two types of charts,
corresponding to simplicial and octahedral cones.

In what follows we denote by $y_{ij}$ the Cox coordinates corresponding to the
divisors $D_{ij}$. In terms of these, the monomial basis of $H^0(F,
L)$ is
\begin{equation}
  \label{eq:1}
x_0 = y_{01}y_{02}y_{03}y_{04}, \; x_1 = y_{01}y_{12}y_{13}y_{14},
\;x_2 = y_{02}y_{12}y_{23}y_{24}, \;x_3 = y_{03}y_{13}y_{23}y_{34},
\;x_4 = y_{04}y_{14}y_{24}y_{34}  
\end{equation}

\subsubsection{Simplicial charts}
\label{sec:simplicial-charts}

We denote by $U_0$, ..., $U_4$ the simplicial charts of $F$ where for
example $U_0=(y_{12}=y_{13}=\ldots=y_{34})=1$ is the chart
corresponding to the cone spanned by the vectors $\rho_{01},
\rho_{02}, \rho_{03}, \rho_{04}$. 
These vectors generate a sublattice of index $3$ in $\ZZ^4$ and in
fact 
\[
U_0=\frac1{3}(1,1,1,1)_{y_{01}, y_{02}, y_{03}, y_{04}}
\]
it is clear that $X$ has $5$ isolated $\frac1{3}(1,1)$ singularities,
one in each of these charts.

\subsubsection{Octahedral charts}
\label{sec:octahedral-charts}

We denote by $V_0$, ..., $V_4$ the octahedral charts of $F$ where, for
example, $V_0=(y_{01}=y_{02}=y_{03}=y_{04}=1)$ is the chart
corresponding to the cone $\sigma$ over the octahedron $\left[ \rho_{12}, \ldots
  \rho_{34}\right]$. From the exact sequence
\[
(0)\to \ZZ^2\overset{i}\to \ZZ^6 \overset{\rho}{\to}\ZZ^4
\]
where $\rho=(\rho_{12}, \ldots, \rho_{34})$ and $N_0=\Image
(\rho)=\{x_1, \ldots, x_4\mid x_1+\ldots + x_4 \equiv 0 \pmod{2}\}$, we see that
$V_0=\CC^6_{y_{12},\ldots,y_{34}}/\CC^{\times\,2}\times \Bmu_2$
where $\CC^{\times\,2}$ acts with weights
\[
\begin{pmatrix}
  1 & -1 & 0 & 0 & -1 & 1\\
  1 & 0  & -1 & -1 & 0 & 1 
\end{pmatrix}
\]
and $\Bmu_2$ acts with weights $0$, $0$, $0$, $\frac1{2}$,
$\frac1{2}$, $\frac1{2}$ $\in (\frac1{2}\ZZ)/\ZZ$. It follows that:
\[
V_0=\Spec \CC[y_{12},\ldots,y_{34}]^{\CC^{\times\, 2}\times \Bmu_2}=
\Spec \CC[z_1,\dots, z_8]^{\Bmu_2}
\]
where $z_1=y_{12}y_{13}y_{14}$, $z_2=y_{12}y_{14}y_{24}$,
$z_3=y_{12}y_{23}y_{24}$, $z_4=y_{12}y_{13}y_{23}$, $z_5=y_{13}y_{14}y_{34}$,
$z_6=y_{14}y_{24}y_{34}$, $z_7=y_{23}y_{24}y_{34}$,
$z_8=y_{13}y_{23}y_{34}$ are generators corresponding to the vertices
of a cube in $M_0=\Hom (N_0, \ZZ)$ dual to the octahedron. Finally
\[
V_0=V(z_1z_6-z_2z_5,z_1z_3-z_2z_4,z_1z_8-z_4z_5,z_3z_6-z_2z_7,z_3z_8-z_4z_7,z_6z_8-z_5z_7)\subset \frac1{2}(0,1,0,1,1,0, 1, 0)
\]
where the $6$ relations correspond to the six faces of the cube. It
follows that $V_0\simeq A/\Bmu_2$ where $A\subset \CC^8$ is the affine
cone over $\PP^1\times\PP^1\times\PP^1$ in its Segre embedding. The
$\Bmu_2$-fixed locus $(z_2=z_4=z_5=z_7=0)$ intersects $A$ in the $4$
curves $z_1z_6=z_1z_3=z_1z_8=z_3z_6=z_3z_8=z_6z_8=0$, i.e. the $z_1$-,
$z_3$-, $z_6$-, $z_8$-axes. It follows that $V_0$ has index $2$ at the
origin, where $4$ curves meet where $F$ has generically
$\frac1{2}(1,1,0,0)$ singularities. $L$ is a Cartier divisor at the
origin, thus $X$ also avoids the origin (this was clear from the
start) and $X$ avoids the $1$-dimensional singular stratum in this
chart, which consists of the $4$ curves just mentioned. 

\subsubsection{The embedding $F\subset \PP(2^5, 3^5)$}
\label{sec:an-embedding-fsubset}

Let $x_0,\dots,x_4$ and $z_0,\dots,z_4$ be homogeneous coordinates on
$\PP(2^5,3^5)$. Then with $x_i$ as in equation~\ref{eq:1} and 
\begin{eqnarray*}
z_0 & = y_{12}y_{13}y_{14}y_{23}y_{24}y_{34}\\ 
z_1 & = y_{02}y_{03}y_{04}y_{23}y_{24}y_{34}\\
z_2 & = y_{01}y_{03}y_{04}y_{13}y_{14}y_{34}\\
z_3 & = y_{01}y_{02}y_{04}y_{12}y_{14}y_{24}\\
z_4 & = y_{01}y_{02}y_{03}y_{12}y_{13}y_{23} 
\end{eqnarray*}
$F$ embeds into $\PP(2^5,3^5)$ with 14 equations:
\begin{enumerate}[(1)]
\item $x_0z_0=x_iz_i$ ($i\in \{1,2,3,4\}$), and
\item $z_0z_1=x_2x_3x_4$, $z_0z_2=x_1x_3x_4$, etc ($10$
  equations).\footnote{One can see that there are 35 syzygies between
    these equations, as is typical of codimension~5 Gorenstein ideals
    with 14 generators.}
\end{enumerate}

\subsection{Italian-style constructions}
\label{sec:birat-constr-4}

Here we list a few Italian-style constructions of some of the
surfaces.

\paragraph{$k$-gons} Take a $k$-gon of $(-1)$-curves on a nonsingular del Pezzo
surface $S$ of degree $2\leq k=K_S^2 \leq 6$. Blowing up the vertices
of the $k$-gon and then blowing down the strict transforms of the $k$
$(-1)$-curves---which have now become $(-3)$-curves---one obtains a surface
with $k\times 1/3(1,1)$ singularities and degree $d=k/3$. Hence this
construction gives $\Xsurf{2}{2/3}$, $\Xsurf{3}{1}$, $\Xsurf{4}{4/3}$,
$\Xsurf{5}{5/3}$, $\Xsurf{6}{2}$.

\paragraph{The family $\Xsurf{5}{2/3}$} A surface in this family is
the blow up of the $10=\binom{5}{2}$ points of intersection of pairs
of $5$ general lines in $\PP^2$ followed by contraction of the proper
preimages of the $5$ lines. A floating $(-1)$-curve on this surface can
be seen as follows: choose $5$ of the $10$ points of intersection of
the $5$ general lines in such a way that no three of them are
collinear. The proper transform of the unique conic through these $5$
points is a floating $(-1)$-curve.

\paragraph{The family $\Xsurf{6}{1}$} A surface in this family is the
blow up of the $9$ points of intersections of pairs of lines (one in
each ruling) of a grid of $6=3+3$ lines, three in each ruling, on
$\PP^1\times \PP^1$ followed by contraction of the proper preimages of
the $6$ lines.


\section{Invariants}
\label{sec:invariants}


The main result of this section is proposition~\ref{pro:2} where we
derive an almost exact table of invariants of del Pezzo surfaces with
$\frac1{3}(1,1)$ points from elementary lattice theory and elementary
covering space theory. These methods are surprisingly effective in
producing an almost exact table of invariants and we hope that they
can be useful in other problems of classification of orbifold del
Pezzo surfaces. We use the result in \S~\ref{sec:trees} to cut down on
the cases we need to consider in the proof of theorems~\ref{thm:5}
and~\ref{thm:4}.  We start with a study of the defect invariant
introduced in \S~\ref{sec:invariants_tiny}.

\begin{lem}
  \label{lem:1} Using the notation introduced in \S~\ref{sec:invariants_tiny},
  $k-r/2\leq \sigma \leq k/2$.
\end{lem}

\begin{proof}
  $\overline{N}/N\subset N^\ast /N$ is totally isotropic where
  $N^\ast/N$ is endowed with the discriminant quadratic form, hence
  $\sigma = \dim_{\FF_3} \overline{N}/N\leq \frac1{2}\dim_{\FF_3}
  N^\ast /N=\frac{k}{2}$. Also, $\Image [N\otimes \FF_3 \to L\otimes
  \FF_3]$ is totally isotropic, hence it has dimension $\leq r/2$,
  thus the kernel has dimension $\geq k-r/2$.
\end{proof}

\begin{rem}
  \label{rem:9}
  In fact one can do better, but we won't need to do so here. For
  example, if $k=2$, then the discriminant bilinear form
  $A(x,y)=x^2+y^2$ has no isotropic vector, hence $\sigma=0$ in this
  case.
\end{rem}

\begin{lem}
  \label{lem:3}
  $H_1(X^0;\ZZ)\cong \FF_3^\sigma$.
\end{lem}

\begin{proof}
  Denote by $E=\cup_{i=1}^kE_i\subset Y$ the exceptional divisor of
  the minimal resolution morphism $Y\to X$, and note that of course
  $X^0=Y\setminus E$. Because $Y\setminus E$ is smooth, the Poincar\'e homomorphism
  $H^i_c(Y\setminus E; \ZZ)\to H_{4-i} (Y;\ZZ)$ is an isomorphism. The long exact sequence
  for compactly supported cohomology fits into a commutative diagram:
\[
\xymatrix{
H^2(Y;\ZZ)\ar[r]\ar@{=}[d]&H^2(E; \ZZ)\ar[r]\ar@{=}[d]&H^3_c (Y\setminus E;\ZZ)\ar[r]\ar@{=}[d]
&H^3 (Y;\ZZ)=(0)\ar@{=}[d]\\
L \ar[r] & N^\star \ar[r] & \FF_3^\sigma \ar[r]& (0) \\}
\]
\end{proof}

The following is the main result of this section.

\begin{pro}
  \phantomsection \label{pro:2} 
  $k\leq 6$ and moreover:
  \begin{enumerate}[(1)]
  \item If $k=1$ then $K^2\equiv 1/3 \pmod{\ZZ}$ and $1/3\leq K^2\leq
    25/3$;
  \item If $k=2$ then $K^2\equiv 2/3 \pmod{\ZZ}$ and $2/3\leq K^2\leq
    20/3$;
  \item If $k=3$ then $K^2\equiv 0 \pmod{\ZZ}$ and $1\leq K^2\leq 5$;
  \item If $k=4$ then $K^2\equiv 1/3 \pmod{\ZZ}$ and $1/3\leq K^2\leq
    10/3$;
  \item If $k=5$ then $K^2\equiv 2/3 \pmod{\ZZ}$ and $2/3\leq K^2\leq
    8/3$;
  \item If $k=6$ then $K^2\equiv 0 \pmod{\ZZ}$ and $1 \leq K^2\leq 2$.
  \end{enumerate}
\end{pro}

\begin{rem}
  \label{rem:8}
  It follows from the proof of theorems~\ref{thm:5} and~\ref{thm:4} that the
  possibilities $k=2, K^2=20/3$; $k=4, K^2=10/3$; $k=5, K^2=8/3$ do
  not actually occur. 
\end{rem}

\begin{proof}[Proof of proposition~\ref{pro:2}]
  By \cite[1.8~Corollary]{MR1610249} the orbi-tangent bundle is
  generically semi-positive, and then by \cite[Chapter~10]{MR1225842}
  $\widehat{c}_2=n+k/3 \geq 0$. It follows from this that $K^2=12 -n
  -\frac{5k}{3}\leq 12-\frac{4k}{3}$. By using $K^2\equiv
  k/3\pmod{\ZZ}$ and:
\[
0<K^2\leq 12-\frac{4k}{3}
\quad
\text{and}
\quad
h^0(X,-K)=1+K^2-k/3\geq 0
\]
we immediately conclude that $k\leq 7$. Indeed, if $k=8$ the first set
of inequalities forces $K^2= 2/3$ but this would imply
$h^0(X,-K)=1+K^2-8/3=1+2/3-8/3=-1<0$, a contradiction. Similarly, when
$k=7$ the first inequality gives $1/3\leq K^2 \leq 7/3$, but
$K^2=1/3$ does not occur because it would imply
$h^0(X, -K)=1+1/3-7/3=-1$, again a contradiction. With a bit more work
we can exclude a few more cases: if $k=1$, then $K^2=31/3$ implies
$r=0$, $K^2=28/3$ implies $r=1$ and both are impossible because $r=k+\rho(X)>k$. Similarly,
$k=2$ and $K^2=26/3$ implies $r=2$, impossible; $k=3$ and $K^2=8$
implies $r=3$, also impossible. Thus $k\leq 7$ and we are left with the following
possibilities:
  \begin{enumerate}[(1)]
  \item If $k=1$ then $K^2\equiv 1/3 \pmod{\ZZ}$ and $1/3\leq K^2\leq
    25/3$;
  \item If $k=2$ then $K^2\equiv 2/3 \pmod{\ZZ}$ and $2/3\leq K^2\leq
    23/3$;
  \item If $k=3$ then $K^2\equiv 0 \pmod{\ZZ}$ and $1\leq K^2\leq 7$;
  \item If $k=4$ then $K^2\equiv 1/3 \pmod{\ZZ}$ and $1/3\leq K^2\leq
    19/3$;
  \item If $k=5$ then $K^2\equiv 2/3 \pmod{\ZZ}$ and $2/3\leq K^2\leq
    14/3$;
  \item If $k=6$ then $K^2\equiv 0 \pmod{\ZZ}$ and $1 \leq K^2\leq 4$;
  \item If $k=7$ then $K^2\equiv 1/3 \pmod{\ZZ}$ and $4/3\leq K^2\leq 7/3$.
 \end{enumerate}
 We are still quite some way from proving what we need. We exclude the
 remaining possibilities by studying the invariant $\sigma$. The
 key observation is that, by lemma~\ref{lem:1}, we have that $\sigma \geq k-r/2$ so,
 for example, if $k=2$ and $K^2=23/3$, we must have $r=3$ and then
 $\sigma >0$. It is easy to see that this case does not occur:
 by lemma~\ref{lem:3} $H_1(X^0;\ZZ)\cong \FF_3^\sigma$, so by
 covering space theory there is a $3$-to-$1$ covering $Y\to
 X$, \'etale above $X^0$, from a surface $Y$, necessarily a del Pezzo
 surface, with $1/3(1,1)$ points and degree $K_Y^2=3\times \frac{23}{3}=23$
 and we already know that such a surface does not exist.

As another example, $k=7$, $K^2=4/3$ implies $\sigma\geq 2$ so there is a
$9$-to-$1$ cover $Y\to X$ from a del Pezzo surface $Y$ with $1/3(1,1)$
points and $K_Y^2=12$ and we know that such a surface does not
exist.

 In table~\ref{tab:defective} we summarise the cases where we can
 definitely conclude $\sigma >0$. All but two are excluded at once by
 the same method (the other two cases actually occur) and the result
 follows.

\begin{table}[h,ht]
  \centering
  \begin{tabular}{c|c|c|c|c}
    $k$ & $K^2$ & $r$ & $\sigma$ & Occurs \\
\hline
    2    &23/3    & 3 & $>0$& No \\
    3    &6         & 5 & $>0$& No \\
    3    &7         & 4 & $>0$& No \\
    4    &13/3    & 7 & $>0$& No \\
    4    &16/3    & 6 & $>0$& No \\
    4    &19/3    & 5 & $>1$& No \\
    5    &8/3      & 9 & $>0$& No \\
    5    &11/3    & 8 & $>0$& No \\
    5    &14/3    & 7 & $>1$& No \\
    6    &1         &11& $>0$ & Yes \\
    6    &2         &10& $>0$ & Yes \\
    6    &3         &9  & $>1$ & No \\
    7    &4/3      &11& $>1$ & No \\
    7    &7/3      &10& $>1$ & No
  \end{tabular}
  \caption{Necessarily defective possibilities}
  \label{tab:defective}
\end{table}
All other possibilities are excluded by the same method, except
$k=5$, $K^2=8/3$, $\sigma\geq 1$: this possibility is not excluded at
this point, and it is not excluded by the statement of
proposition~\ref{pro:2}. Table~\ref{tab:defective} states that it does
not occur, but this fact will only follow from the proof of
theorems~\ref{thm:5} and~\ref{thm:4} in \S~\ref{sec:trees}.

\end{proof}


\section{MMP}
\label{sec:mmp}


In our proof of theorem~\ref{thm:5} in \S~\ref{sec:trees} we
systematically use the following elementary result, which we state
without proof. Analogous statements for surfaces with canonical
singularities can be found in \cite{MR787190, MR868434}.

\begin{thm}
  \label{thm:2}
  Let $X$ be a projective surface having $k\times \frac1{3}(1,1)$,
  $n_2\times A_2$, and $n_1\times A_1$ singularities. 

  Assume that $k+2n_2 + n_1\leq 6$. 

  Let $f\colon X \to X_1$ be an extremal contraction. Then exactly one
  of the following holds:

  (E) $f\colon (X, E)\to (X_1, P)$ is a divisorial contraction. Denote
  by $Y\to X$ and $Y_1\to X_1$ the minimal resolutions, and
  $E^\prime \subset Y$ the proper transform of the exceptional
  curve. Then $E^\prime\subset Y$ is a $(-1)$-curve meeting
  transversely at most one exceptional curve of $Y\to X$ above each
  singularity, and one of the following holds:
\begin{itemize}
\item[(E.1)] $E$ is contained in the nonsingular locus. Then $E$ is a
  $(-1)$-curve and we call it a \emph{floating} $(-1)$-curve;
\item[(E.2)] (A1 contraction) $E$ contains one
  $A_1$-singularity, $P\in X_1$ is a nonsingular point;
\item[(E.3)] (A2 contraction) $E$ contains one
  $A_2$-singularity, $P\in X_1$ is a nonsingular point;
\item[(E.4)] $E$ contains one $\frac1{3}(1,1)$-singularity, $P\in X_1$
  is a $A_1$-point;
\item[(E.5)] $E$ contains one $\frac1{3}(1,1)$-singularity and one
  $A_1$ singularity, $P\in X_1$ is a nonsingular point;
\item[(E.6)] $E$ contains two $\frac1{3}(1,1)$-singularities, $P\in
  X_1$ is an $A_2$-point.
\end{itemize}

(C) $X_1=\PP^1$, that is, $f$ is generically a conic bundle. Denote by
$F\subset X$ a special fibre of $f$, and by $Y\to X$ and $Y_1\to X_1$ the
minimal resolutions and $F^\prime \subset Y$ the proper transform of
$F$. Then $F^\prime$ is a $(-1)$-curve and one of
the following holds:
\begin{itemize}
\item[(C.1)] $F$ contains two $A_1$-singularities, and $F^\prime$
  meets each of the $(-2)$-curves transversely;
\item[(C.2)] $F$ contains one $\frac1{3}(1,1)$-singularity and one
  $A_2$ singularity, and $F^\prime$ meets the $(-3)$-curve and one of
  the $(-2)$-curves transversely.
\end{itemize}

(D) $X_1=\{\text{pt}\}$ is a point, that is, $X$ is a del Pezzo surface
of Picard rank one, and $X$ is one of the following surfaces:
\begin{itemize}
\item[(D.1)] $\PP^2$;
\item[(D.2)] $\PP(1,1,2)$ (this surface has exactly one $A_1$ singular point);
\item[(D.3)] $\PP(1,2,3)$ (this surface has exactly one $A_1$ and one $A_2$ singularities);
\item[(D.4)] $\PP^2/\mu_3$ where $\mu_3$ acts with weights
  $1,\omega, \omega^2$. This surface has exactly $3\times A_2$
  singularities;\footnote{$X$ is the toric surface obtained by blowing
    up $3$ vertices on the hexagon of lines of a degree $6$
    nonsingular del Pezzo surface.}
\item[(D.5)] $\PP(1,1,3)$.
\end{itemize} \qed
\end{thm}

\begin{rem}
  Consider the class of projective surfaces $X$ be having
  $k\times \frac1{3}(1,1)$, $n_2\times A_2$, and $n_1\times A_1$
  singularities and $k+2n_2+n_1\leq 6$. It follows from the previous
  statement that a MMP starting from a surface in the class only
  involves surfaces in the class.
\end{rem}

\paragraph{The directed minimal model program}

In the proof of theorems~\ref{thm:5} and~\ref{thm:4} in the following
section, we run the MMP starting with a del Pezzo surface with
$\frac1{3}(1,1)$ points.

In all cases, we perform extremal contractions in the order that they
are listed in theorem~\ref{thm:2} above: that is, we first contract
all the floating $(-1)$-curves, then we contract a ray of type (E.2) if
available, or else one of type (E.3), etc.

We call this the \emph{directed MMP}.

\begin{lem} \label{rmk:curves on Y} Let $X$ be a del Pezzo surface
  with $k\geq 1$ $\frac1{3}(1,1)$ singular points. Assume that $X$
  contains no floating $(-1)$-curves. Denote by 
\[
X=X_0\overset{\varphi_0}{\longrightarrow} \ldots \longrightarrow
X_{i-1}\overset{\varphi_{i-1}}{\longrightarrow} X_i \longrightarrow \ldots
\]
the contractions and surfaces occurring in a MMP for $X$ (not
necessarily directed). 
\begin{itemize}
\item[(1)] All surfaces $X_i$ are del Pezzo surfaces.
\item[(2)] Denote by $f_i:Y_i\rightarrow
X_i$ the minimal resolution of $X_i$ and let $C\subset Y_i$ be a
(reduced and irreducible) curve with negative self-intersection
$C^2=-m$. Then:
\begin{itemize}
\item[(2.1)] if $C$ is $f_i$-exceptional, then $m=2$ or $3$,
\item[(2.2)] if $C$ is not $f_i$-exceptional, then $m=1$ and $C$
  intersects at least one $f_i$-exceptional curve. In particular, none
  of the surfaces $X_i$ contain a floating $(-1)$-curve. 
\end{itemize}
\end{itemize}
\end{lem}

\begin{proof}[Proof of Lemma~\ref{rmk:curves on Y}]
  We prove the statement by induction on $i$.  We first show that
  $X_i$ is a del Pezzo surface. Suppose $X_{i-1}$ is del Pezzo and let
  $E\subset X_{i-1}$ be the effective divisor such that
  $K_{X_{i-1}}=\varphi_{i-1}^\star K_{X_{i}}+aE$, $a>0$. Let
  $\Gamma\subset X_i$ be a curve. Denoting by $\Gamma^\prime \subset
  X_{i-1}$ the proper transform, we have that:
\[
K_{X_i}\cdot \Gamma=K_{X_i}\cdot
\varphi_{i-1\,\star}\Gamma^\prime=\varphi_{i-1}^\star (K_{X_i})\cdot \Gamma^\prime
= (K_{X_{i-1}} -aE)\cdot \Gamma^\prime<0.
\] 
As $K_{X_{i}}^2>K_{X_{i-1}}^2$, by the Nakai-Moishezon criterion we
conclude that $-K_{X_{i}}$ is ample.

Assuming that (2.1) holds for $X_{i-1}$, then it also holds for $X_i$,
by the structure of divisorial contractions listed in theorem~\ref{thm:2}.

Let now $C$ be a $(-m)$-curve on $Y_i$ that is not contracted by
$f_i$, then since $-K_{Y_i}+f_i^\star K_{X_i}\geq 0$ we have that: 
\[
-K_{Y_i}\cdot C=f_i^\star (-K_{X_i})\cdot C + (-K_{Y_i}+f_i^\star
K_{X_i})\cdot C \geq f_i^\star (-K_{X_i})\cdot C=-K_{X_i}\cdot f_{i\,
  \star} C> 0.
\]
 Then $K_C=(K_{Y_i}+C)|_C<0$, therefore $C$ is a
rational curve and $K_{Y_i}\cdot C=m-2<0$ implies $m=1$, that is,
$C\subset Y$ is a $(-1)$-curve, and the image $C_i=f_i(C)\subset X_i$ is
a floating $(-1)$-curve. Now $C_i$ does not contain the image of the
$\varphi_{i-1}$-exceptional curve: otherwise, the proper transform
$C^\prime \subset Y_{i-1}$ would be a curve of negative
self-intersection $C^{\prime\,2}<-1$ not contracted by $f_{i-1}\colon
Y_{i-1}\to X_{i-1}$, contradicting (2.1) for $X_{i-1}$. Thus, $C_i$ is
the image of a floating $(-1)$-curve in $X_{i-1}$ and then in fact, by
descending induction on $i$, $C_i$ is the image of a floating
$(-1)$-curve on $X$, a contradiction to our main assumption that there
are no such things. This shows $(2.2)$.
\end{proof}

\begin{rem} \label{rem:6and3} In section~\ref{sec:trees} we use the
  following type of argument very frequently. Suppose that
\[
X=X_0\overset{\varphi_0}{\longrightarrow} \ldots \longrightarrow
X_{i}\overset{\varphi_{i}}{\longrightarrow} X_{i+1}
\overset{\varphi_{i+1}}{\longrightarrow} X_{i+2} \ldots
\]
is the sequence of contractions and surfaces occurring in a directed
MMP for $X$. If $\varphi_i$ is of type (E.6), then $\varphi_{i+1}$ is
not of type (E.3). Indeed denote by $f_i:Y_i\rightarrow
X_i$ the minimal resolution. If $\varphi_{i+1}$ is of type (E.3), the proper transform $C\subset Y_{i+1}$ of
the curve contracted by $\varphi_{i+1}$ is a $(-1)$-curve, and its proper
transform on $Y_i$ is a $(-1)$-curve that shows that a contraction of
type (E.3) or (E.4) was available on $X_i$ in the first place, and this is a
contradiction.
\end{rem}


\section{Trees}
\label{sec:trees}


This section is the heart of the paper. We prove theorem~\ref{thm:4},
from which theorem~\ref{thm:5} of the introduction immediately
follows. The proof uses proposition~\ref{pro:2}. 

\begin{thm}
  \label{thm:4}
  Let $X$ be a del Pezzo surface with $k\geq 1$ $\frac1{3}(1,1)$
  singular points. If $X$ has no floating $(-1)$-curves, then it is
  one of the following surfaces. The images show the sequence of
  contractions and surfaces of the directed MMP for $X$ providing a
  birational construction of it, followed by---and separated by a
  double horizontal rule---a picture of the minimal resolutions of the
  surfaces of the MMP showing a configuration of curves on them. We
  hope that these are self-explanatory:

  \begin{itemize}
\item In the images showing the sequence of contractions we record the
  singularities on each intermediate surface. For example, ``$2\times
  1/3 + A_2$" signifies a surface with two $\frac1{3}(1,1)$
  singularities and one $A_2$ singularity.
\item In the pictures showing the minimal resolutions the contracted
  curves are in bold and their images are denoted by a bold point.
  \end{itemize}

\begin{minipage}{\textwidth}
  \begin{flushleft}
    (1) $k=1$ and either $X=\Bsurf{1}{16/3}$ (the first case
    pictured), or $X=\PP(1,1,3)$ (the second case pictured):
  \end{flushleft}
  \centering{
  \resizebox{16cm}{!}{\includegraphics[width=0.8\textwidth]{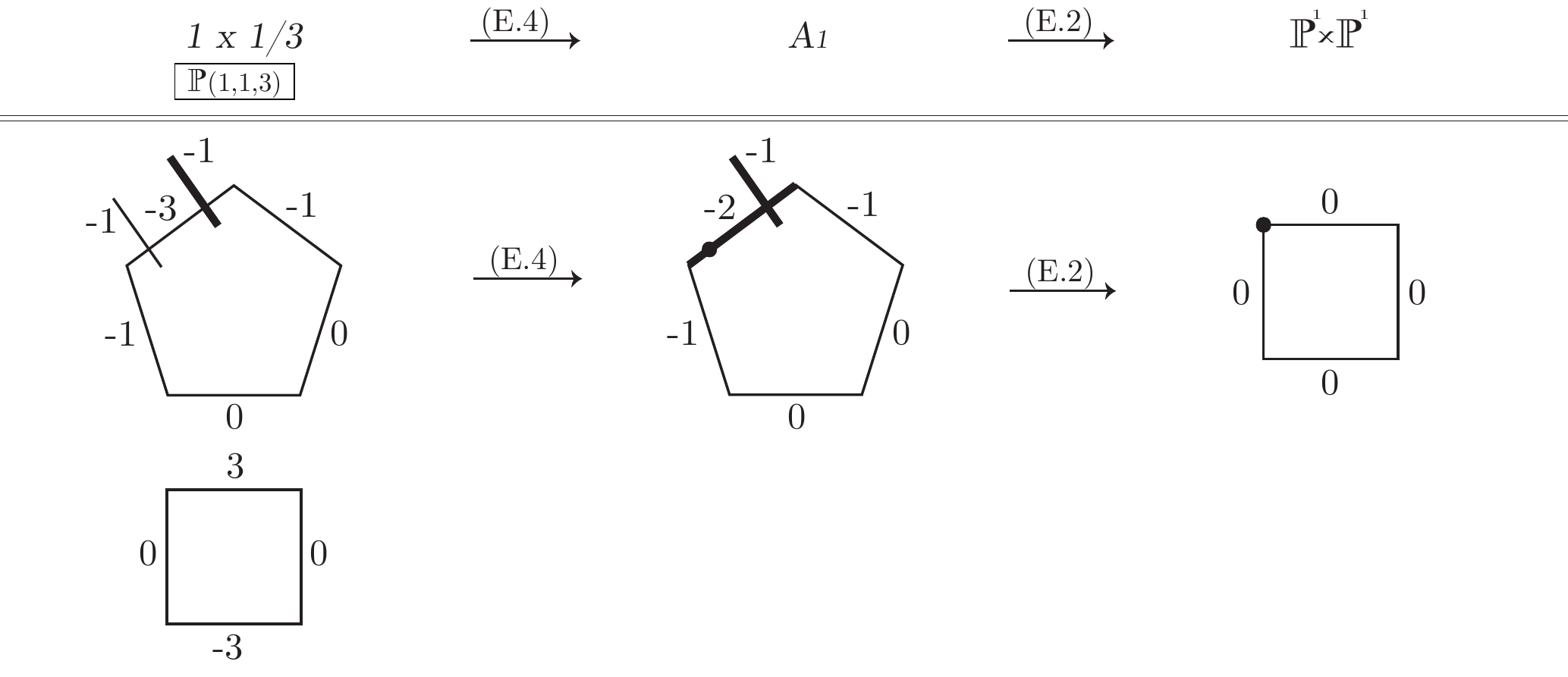}}}
\end{minipage}
\vskip 1truecm
\begin{minipage}{\textwidth}
  \begin{flushleft}
(2) $k=2$ and either $X=\Bsurf{2}{8/3}$ (the first case pictured),
  or $\Xsurf{2}{17/3}$ (the second case):
  \end{flushleft}
  \centering{
  \resizebox{16cm}{!}{\includegraphics[width=0.8\textwidth]{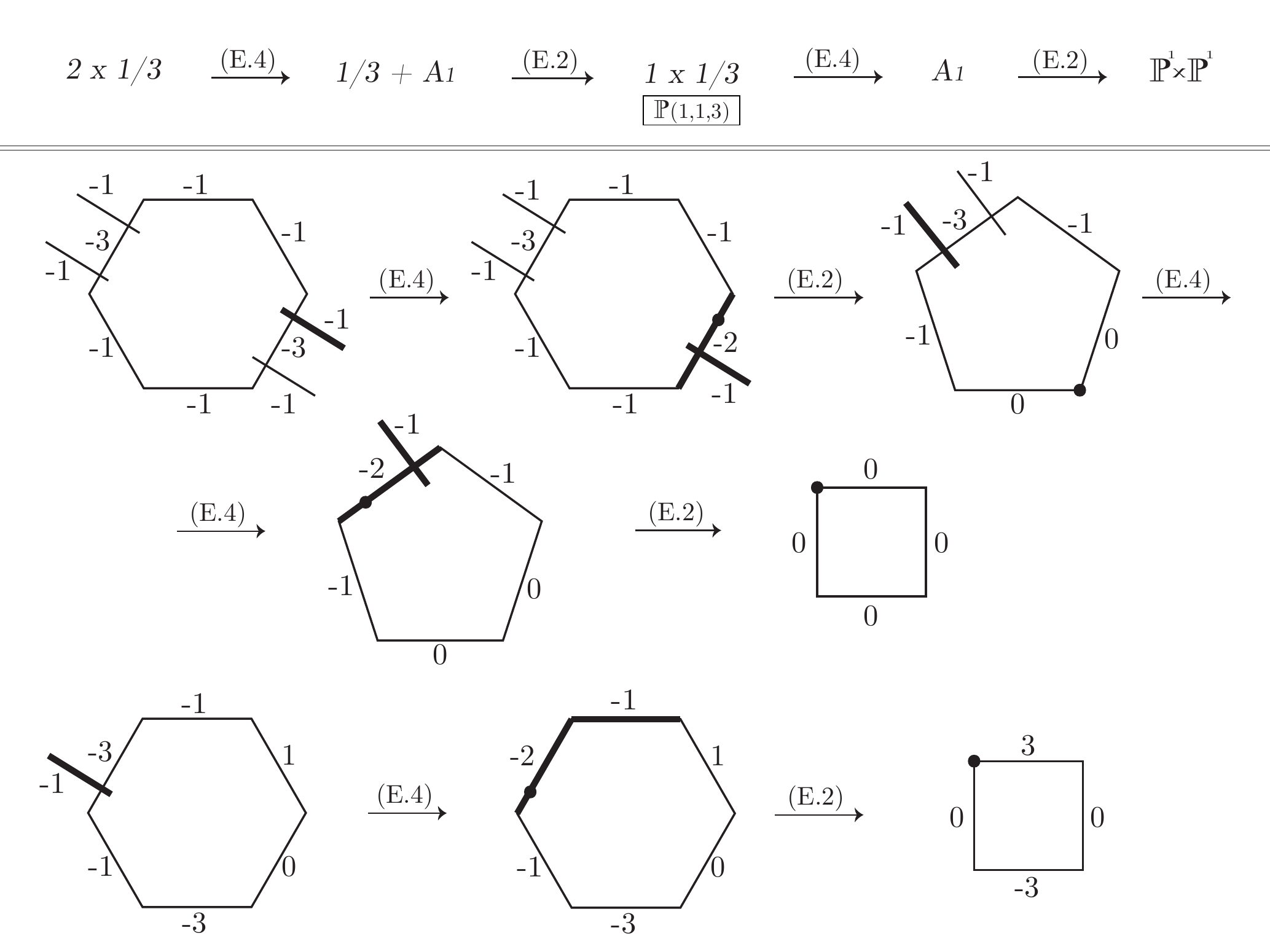}}}
\end{minipage}
\vskip 1truecm
\begin{minipage}{\textwidth}
  \begin{flushleft}
    (3) $k=3$ and $X=\Xsurf{3}{5}$:
  \end{flushleft}
  \centering{
  \resizebox{16cm}{!}{\includegraphics[width=0.8\textwidth]{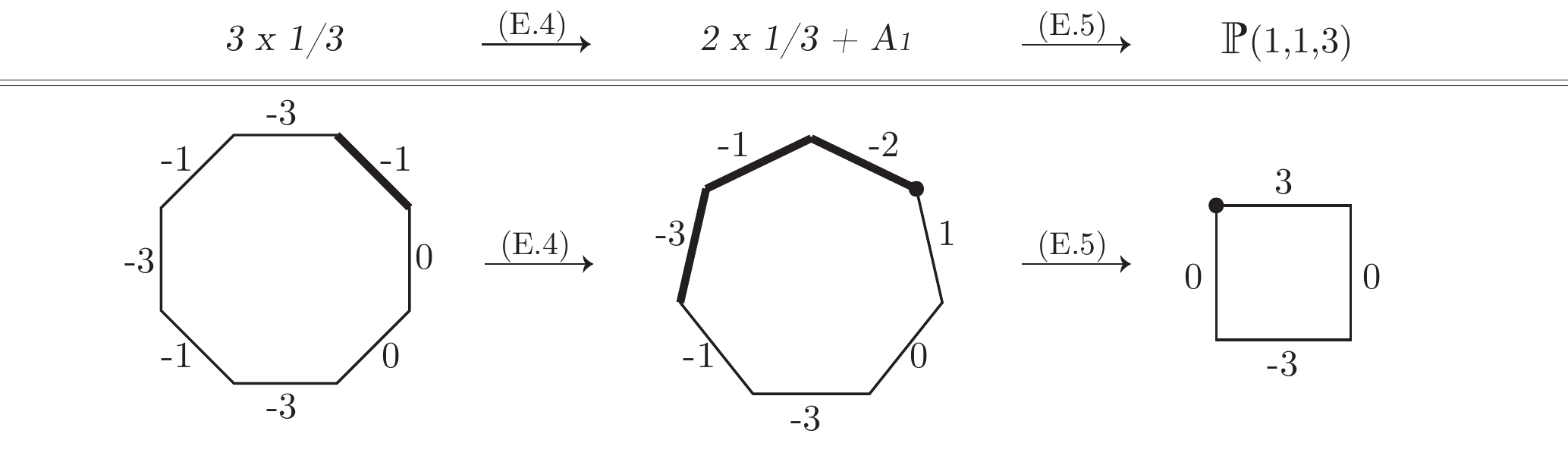}}}
\end{minipage}
\vskip 1 truecm
\begin{minipage}{\textwidth}
  \begin{flushleft}
    (4) $k=4$ and $X=\Xsurf{4}{7/3}$:
  \end{flushleft}
  \centering{
  \resizebox{16cm}{!}{\includegraphics[width=0.8\textwidth]{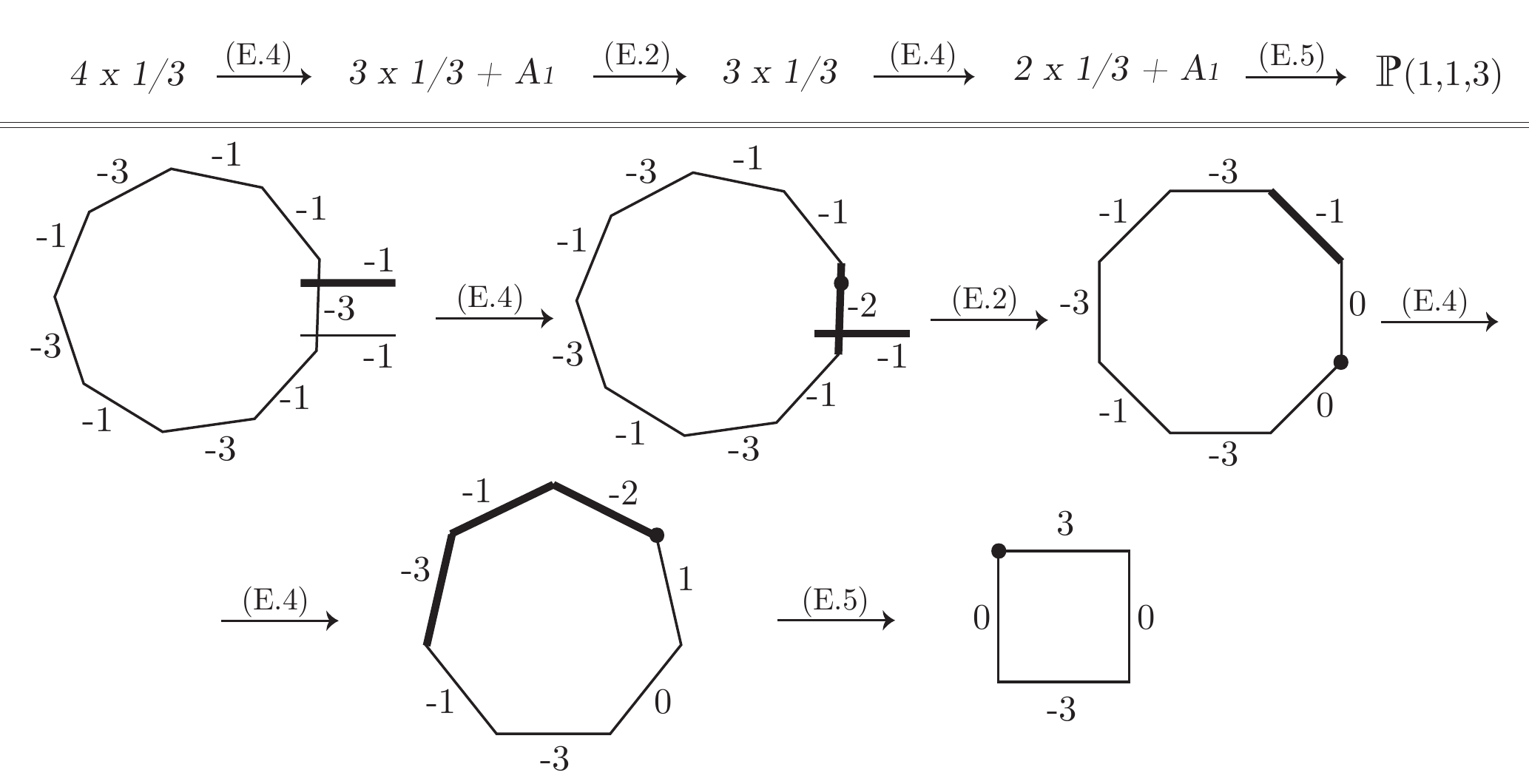}}}
\end{minipage}
\vskip 1 truecm
\begin{minipage}{\textwidth}
  \begin{flushleft}
    (5) $k=5$ and $X=\Xsurf{5}{5/3}$:
  \end{flushleft}
  \centering{
  \resizebox{16cm}{!}{\includegraphics[width=0.8\textwidth]{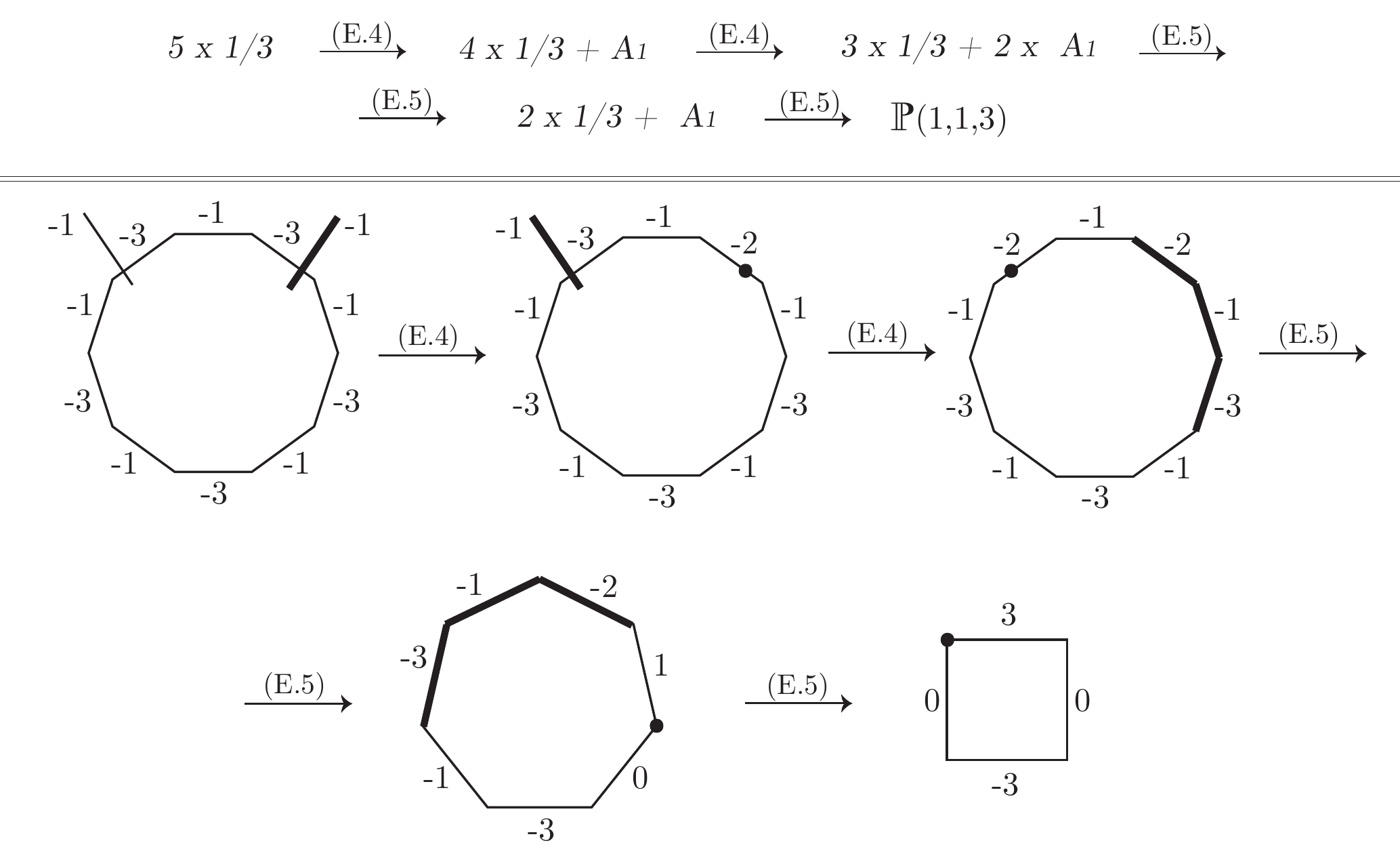}}}
\end{minipage}
\vskip 1truecm
\begin{minipage}{\textwidth}
  \begin{flushleft}
    (6) $k=6$ and $X=\Xsurf{6}{2}$:
  \end{flushleft}
  \centering{
  \resizebox{16cm}{!}{\includegraphics[width=0.8\textwidth]{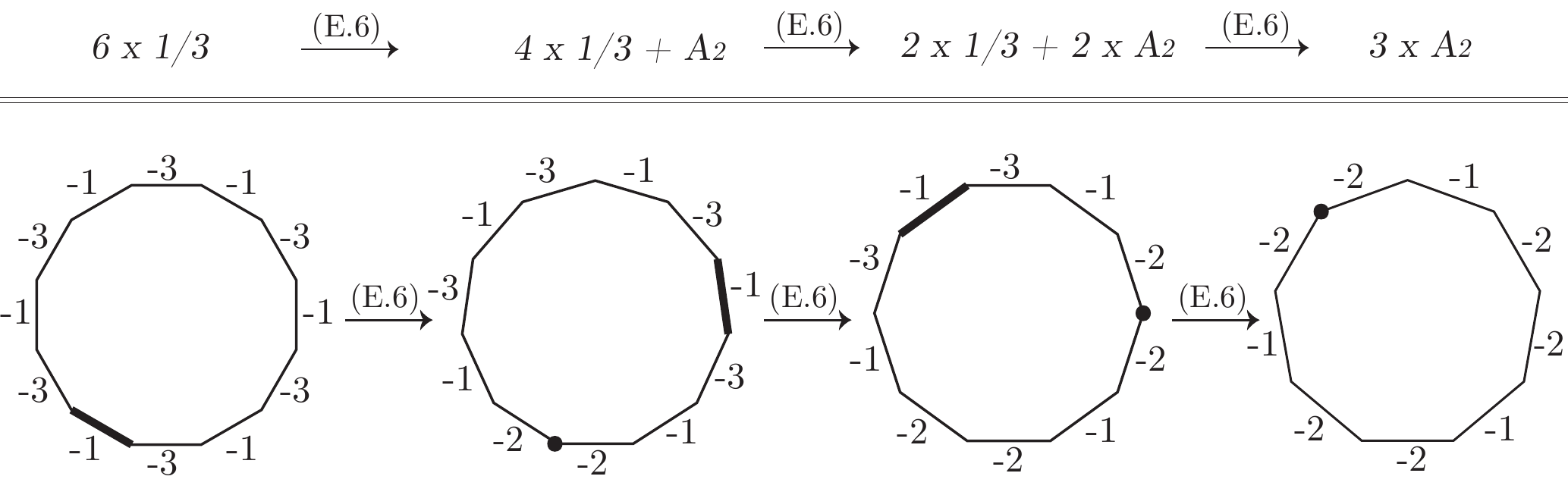}}}
\end{minipage}
\end{thm}
\vskip 1truecm
\begin{proof}
  In all cases we run the directed MMP for $X$. In other words, at
  each step we choose rays exactly in the order that they are listed
  in theorem~\ref{thm:2}. The figures in the statement show the
  sequence of contractions as they occur in the directed minimal model
  program.

  We begin the proof by drawing a tree representing the directed MMPs
  that can potentially occur (figure~\ref{fig:k=4possibtree} is an
  example). For each branch, corresponding to a sequence of
  contractions, we construct a configuration of curves on the minimal
  resolution $Y$. In many cases, the configuration of curves shows
  that at some stage in the MMP there was the option of performing a
  contraction higher up in the list of theorem~\ref{thm:2}: that is,
  the MMP represented by that branch is not directed and hence it does
  not actually occur. At the end we are left with the directed MMPs
  that actually take place.

  Here we only treat in detail the cases $k=4$ and $k=6$; the other
  cases are very similar and can be done by the same
  methods. Figures~\ref{fig:k=1pos} to~\ref{fig:k=5pos} at the
  end of the proof list the remaining trees for all $k$. We leave it up
  to the interested reader to finish the proof. 

\paragraph{The $k=4$ case}

  We first argue that the sequence of extremal contractions of the
  directed MMP must be one of those shown on
  figure~\ref{fig:k=4possibtree}. 

In the argument that follows we denote by 
\[
X=X_0\overset{\varphi_0}{\longrightarrow} \ldots \longrightarrow
X_{i-1}\overset{\varphi_{i-1}}{\longrightarrow} X_i \longrightarrow \ldots
\]
the sequence of contractions and surfaces occurring in a directed MMP
for $X$. Also we denote by $f_i\colon Y_i \to X_i$ the minimal
resolutions.  By theorem~\ref{thm:2}, $\varphi_0$ is either of type
(E.6) or (E.4) and we claim that (E.6) does not occur. 

Suppose for a contradiction that $\varphi_0$ is an (E.6)
contraction. By theorem~\ref{thm:2}, $\varphi_1$ is of type (E.3),
(E.4) or (E.6): indeed, $\varphi_1$ can not be a conic bundle because
$X_1$ has an odd number of singular points and from the classification
of fibres every special fibre has two singularities on it, and it is
clear from the classification that $X_1$ is not a del Pezzo surface
with $\rho =1$.

By remark~\ref{rem:6and3}, (E.3) can not follow (E.6). If $\varphi_1$
is of type (E.4), then this contraction would have been already
available on $X_0$, a contradiction. Finally, if $\varphi_1$ is of
type (E.6), $X_2$ has $2\times A_2$ singularities and then, by
theorem~\ref{thm:2}, $\varphi_2$ is of type (E.3): just as before,
none of the $\rho=1$ del Pezzo surfaces have $2\times A_2$
singularities, and from the classification of fibres $\varphi_2$ can
not be a conic bundle. But again (E.3) can not follow (E.6).

All of this shows that $\varphi_0$ is of type (E.4), therefore $X_1$
has $3\times\frac1{3}(1,1) +A_1$ singularities. Thus $\varphi_1$ can
be of type (E.2), (E.4), (E.5) or (E.6), and we claim that the last
two do not occur. 

Suppose for a contradiction that $\varphi_1$ is of type (E.5). $X_2$
is a del Pezzo surface with $2\times\frac1{3}(1,1)$ singularities. By
the case $k=2$ of the theorem, which we assume to have already proved,
$\varphi_2$ is of type (E.4) and this contraction was available on $X_1$,
a contradiction.

If $\varphi_1$ is of type (E.6) the surface $X_2$ has
$A_1+A_2+\frac1{3}(1,1)$ singularities. The contraction $\varphi_2$ can not
be of type (E.2), (E.4) or (E.5) because otherwise the same
contraction would have been available on $X_1$. It can not be of type
(E.3) either because by remark~\ref{rem:6and3} (E.3) can not follow
(E.6). By theorem~\ref{thm:2} these were the only possibilities thus
this case does not occur.

If $\varphi_1$ is of type (E.2) then $X_2$ is a del Pezzo surface with
$k=3$ and the tree continues as in the $k=3$ case, which we assume
already known. 

If $\varphi_1$ is of type (E.4) then $X_2$ has $2\times A_1 + 2\times
\frac1{3}(1,1)$ singularities. 

The next contraction $\varphi_2$ is not of type (E.2) because it would
have been available earlier.

If $\varphi_2$ were of type (E.6) then $X_3$ would have $A_2+2A_1$
singularities. The next contraction $\varphi_3$ is not of type (E.2)
because it would have been available earlier; it is not of type (E.3)
because by remark~\ref{rem:6and3} (E.3) does not follow (E.6); it is
not of fibering type because $X_3$ has an odd number of
singularities; and $X_3$ is not a del Pezzo surface with $\rho =1$ by
the classification of theorem~\ref{thm:2}.

Thus, $\varphi_2$ is not of type (E.2) or (E.6) and it can be only of
type (E.4) or (E.5), which can be shown to lead respectively to the
two remaining possibilities in figure~\ref{fig:k=4possibtree}.

\begin{figure}[H]
    \caption{$k=4$ tree of possibilities}
    \label{fig:k=4possibtree}
  \centering
  \resizebox{16cm}{!}{\includegraphics[width=0.8\textwidth]{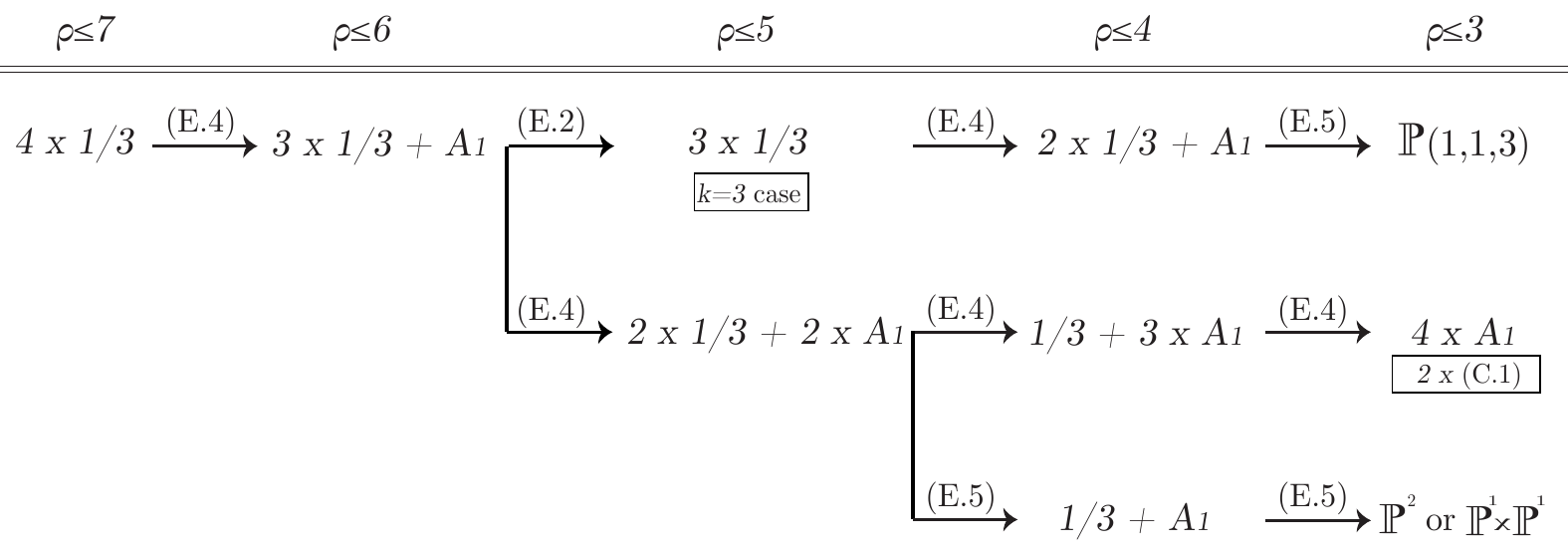}}
\end{figure}

We now explore the branches of this tree one at a time and show that
only one actually occurs.

\paragraph{Case 1}
\label{sec:case-1}

$(E.4)+(E.2)+(E.4)+(E.5)$
 
\medskip

  If this sequence of contractions occurs, then $Y$ must contain the
  configuration of curves depicted in figure~\ref{fig:k=4case1cv-cfg}
  below.  The figure shows the effect of the contractions of the MMP
  on the minimal resolutions: the contracted curves are in bold, as are the
  points onto which they map. 

\begin{figure}[H]
    \caption{A picture of the configuration of negative curves for
      $k=4$, Case~1}
    \label{fig:k=4case1cv-cfg}
  \centering
  \resizebox{16cm}{!}{\includegraphics[width=0.8\textwidth]{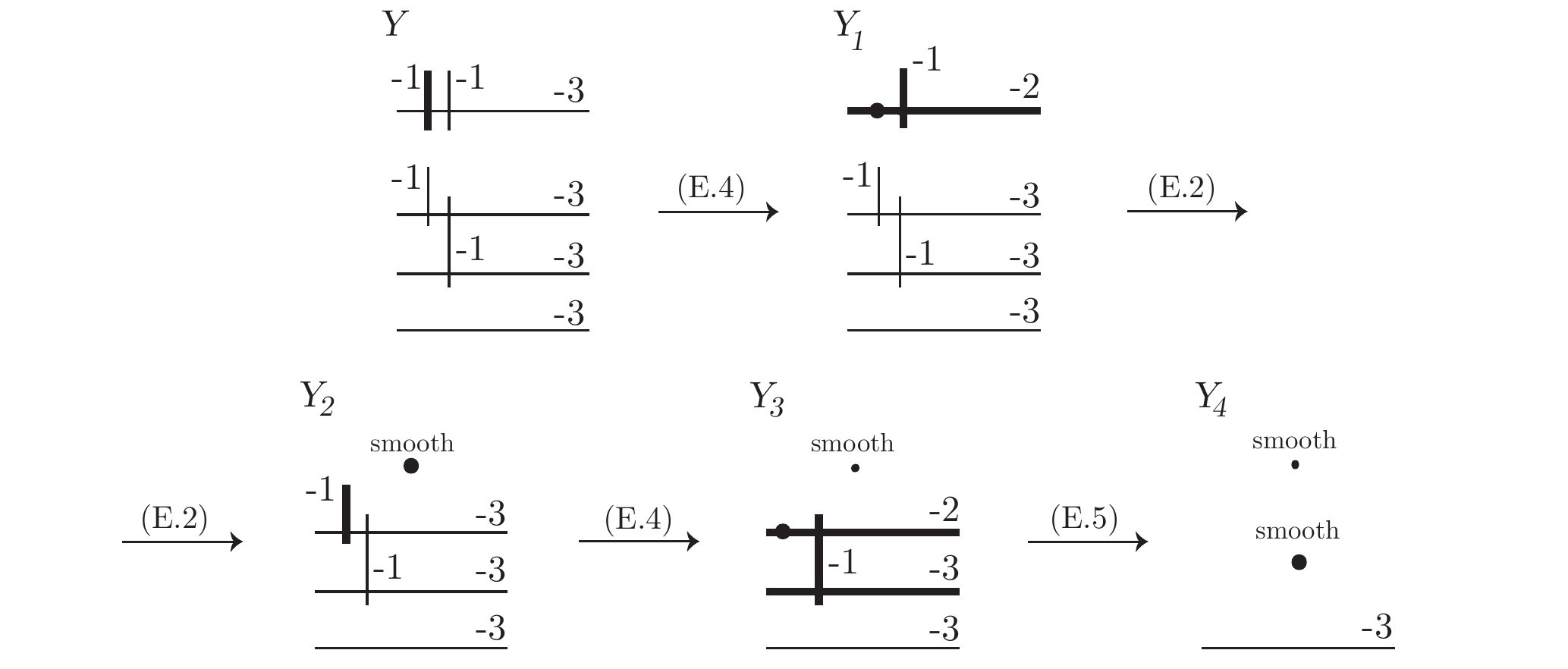}}
\end{figure}

Looking more closely at how $Y$ is built from $Y_4=\FF_3$ by a sequence
of blow-ups, we argue that $Y$ must have more negative curves,
shown in figure~\ref{fig:k=4case1exp.pdf} below.

We use the following result:

  \begin{rem}
    \label{rem:1}
    Let $X$ be a del Pezzo surface and $C\subset X$ an irreducible
    rational curve with positive self intersection. Then $C$ moves in
    a free linear system. Indeed, the map
    $H^0(X,\mathcal{O}_X(C))\rightarrow H^0(C,\mathcal{O}_C(C))$ is
    surjective because $-K_X$ is ample. Since $C^2>0$ we have that
    $\mathcal{O}_C(C)$ is base point free and the conclusion
    immediately follows from the vanishing of $H^1(X, \oo_X)$.
  \end{rem}

\begin{figure}[H]
  \caption{A better picture of the configuration of negative curves
    for $k=4$, Case~1}
  \label{fig:k=4case1exp.pdf}
  \centering
  \resizebox{16cm}{!}{\includegraphics[width=0.8\textwidth]{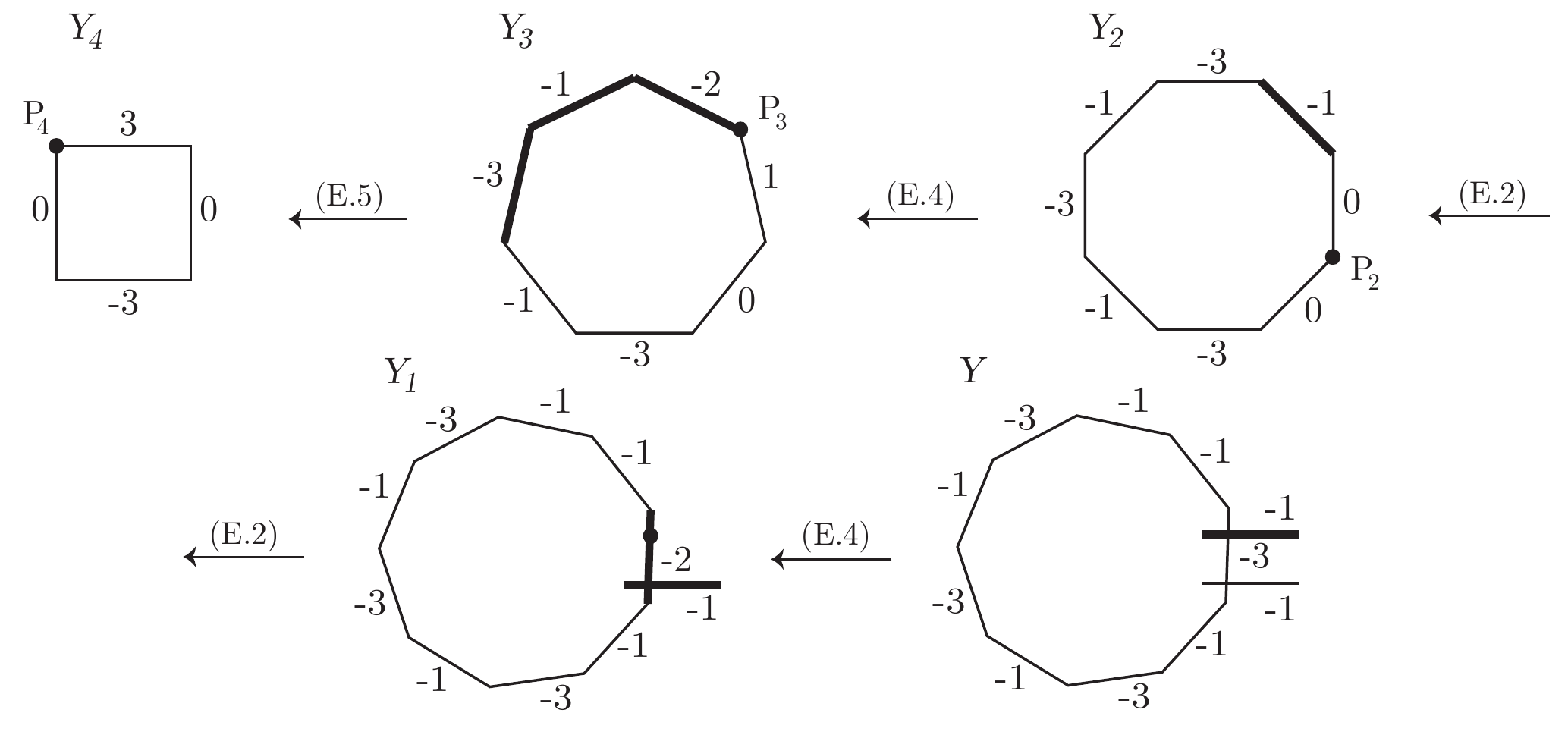}}
\end{figure}

Indeed, the point $P_4\in Y_4$ that is the image of the exceptional curve
in $Y_3$ does not lie on the $(-3)$-curve and then by
remark~\ref{rem:1} we can choose a configuration of curves as
shown in the Figure displaying $P_4$ as the intersection of a fibre and a curve
of self-intersection $+3$. At the next step we need to blow up a point
$P_3\in Y_3$ on the $(-2)$-curve and not contained in any other negative
curve. By remark~\ref{rem:1} again, we can ``move'' the curve with
self-intersection $1$ until it contains $P_3$ as in the figure. At the
next step again we need to blow up a nonsingular point $P_2\in Y_2$
not lying on any negative curve, and we use remark~\ref{rem:1} to
``move'' the two curves with self-intersection $0$ until they both
contain $P_2$.  

We are left with the minimal resolution of the surface $X_{4,7/3}$.

\paragraph{Case 2}
\label{sec:case-2}

$(E.4)+(E.4)+(E.4)+(E.4)$
 
\medskip

We contract one $(-1)$-curve intersecting each $(-3)$-curve on $Y$ and
end up with a surface fibering over $\PP^1$, denoted by $Y_4$,
corresponding to an extremal contraction $X_4\to \PP^1$ having two
singular fibers of type (C.1). As $Y_4$ is a nonsingular surface, we
next run the classical Minimal Model Program for $Y_4$ relative to the
existing fibration $Y_4\to \PP^1$, which ends in a Segre surface
$\FF_k$, and we claim that we can assume that $k=0$:

\begin{figure}[H]
  \centering
  \resizebox{16cm}{!}{\includegraphics[width=0.8\textwidth]{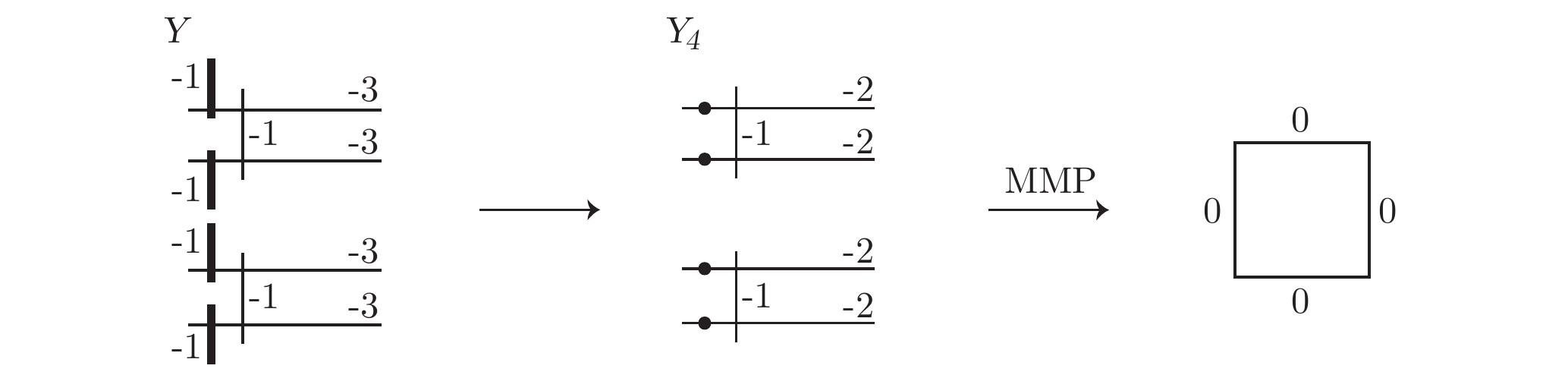}}
  \label{fig:k=4case2cv-cfg}
\end{figure}

Indeed, since all negative curves left on $Y_4$ have self-intersection
$\geq -2$, the same is true about $\FF_k$. We are left with
$k\in\{0,1,2\}$. If $k=2$, this leads to a $(-2)$-curve on $Y$ not
contracted by the morphism to $X$, contradicting lemma \ref{rmk:curves
  on Y}. If on the other hand $k=1$ then by choosing the last
contraction differently we would have landed on $\FF_0$.

\begin{figure}[H]
  \centering
  \resizebox{16cm}{!}{\includegraphics[width=0.8\textwidth]{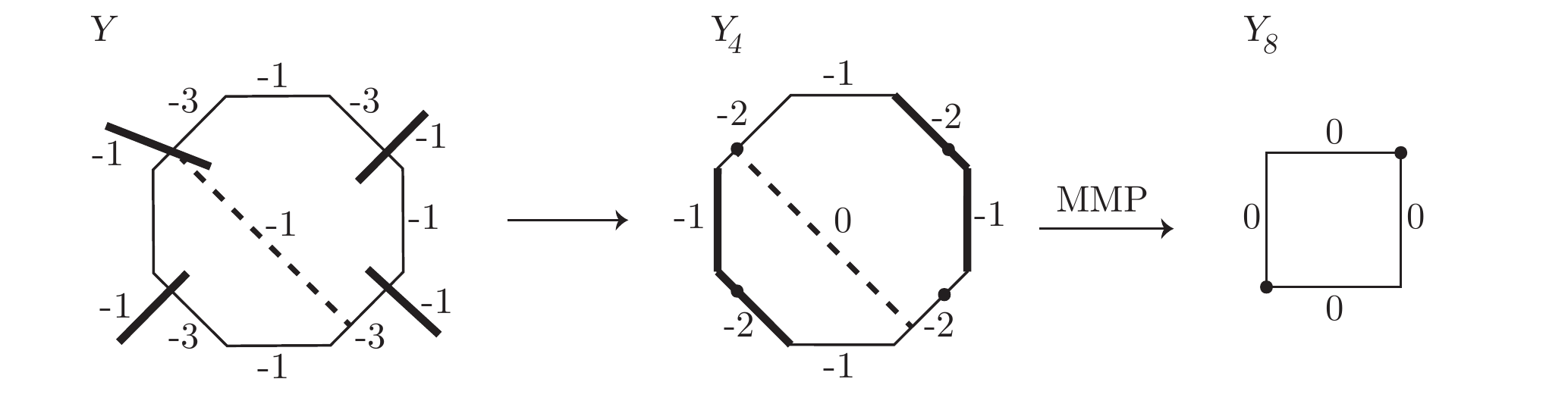}}
  \label{fig:k=4case2exp}
\end{figure}

The horizontal ruling of $\FF_0$ transforms to a free linear system of
$(0)$-curves intersecting two of the opposing $(-2)$-curves in the
special fibers on $Y_4$, depicted as a dashed line in the figure. As
both of these $(-2)$-curves contain a point that is to be blown up in
the process of building $Y$, we denote by the dotted line the
$(0)$-curve passing through the point $P_4\in Y_4$ that is the image of
the exceptional curve in $Y_3$ corresponding to the last (E.4)-type
contraction (as well as its strict transform on $Y$). Exactly before
this contraction is performed, in $Y_3$ (and, ultimately, also in
$Y$), the dotted line is a $(-1)$-curve intersecting only the
$(-1)$-curve that is to be contracted and the $(-2)$-curve on the
other special fiber. This means that an (E.2)-type contraction was
available on $X_3$, which contradicts the fact that a directed MMP was
used in obtaining $X_4$. Therefore this case
does not occur.

\paragraph{Case 3}
\label{sec:case-3}

$(E.4)+(E.4)+(E.5)+(E.5)$
 
\medskip

In this case the minimal resolutions must contain the following
configurations of curves:

\begin{figure}[H]
  \centering
  \resizebox{16cm}{!}{\includegraphics[width=0.8\textwidth]{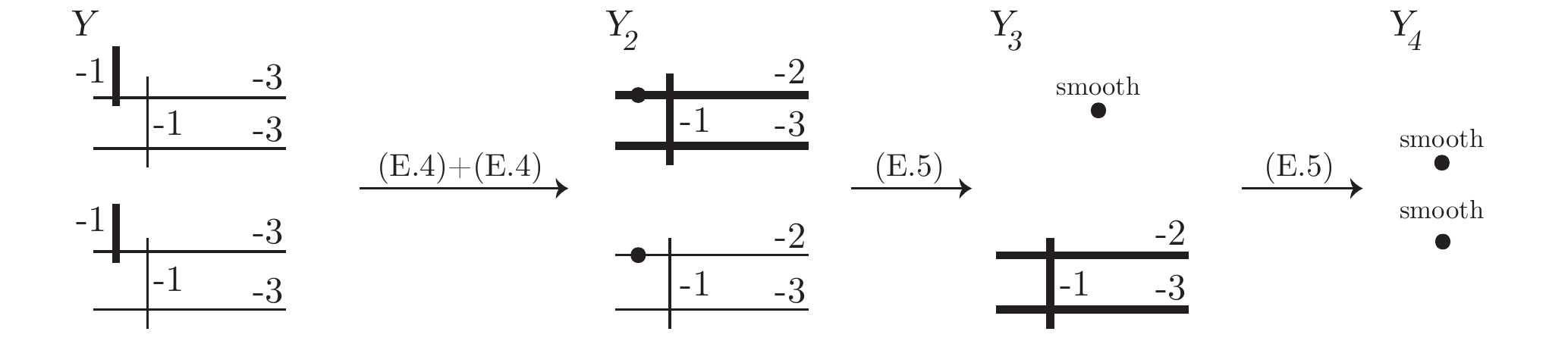}}
  \label{fig:k=4case3cv-cfg}
\end{figure}

Since the resulting surface $X_4$ is smooth and rationally connected
and, by lemma~\ref{rmk:curves on Y}, it contains no negative curves,
it is either $\PP^2$ or $\PP^1\times \PP^1$. In the case of the
projective plane, looking more closely at how $Y$ is built from
$Y_4=\PP^2$ by a sequence of blow-ups, we see that $Y$ must have the
negative curves shown in the following figure:

\begin{figure}[H]
  \centering
  \resizebox{16cm}{!}{\includegraphics[width=0.8\textwidth]{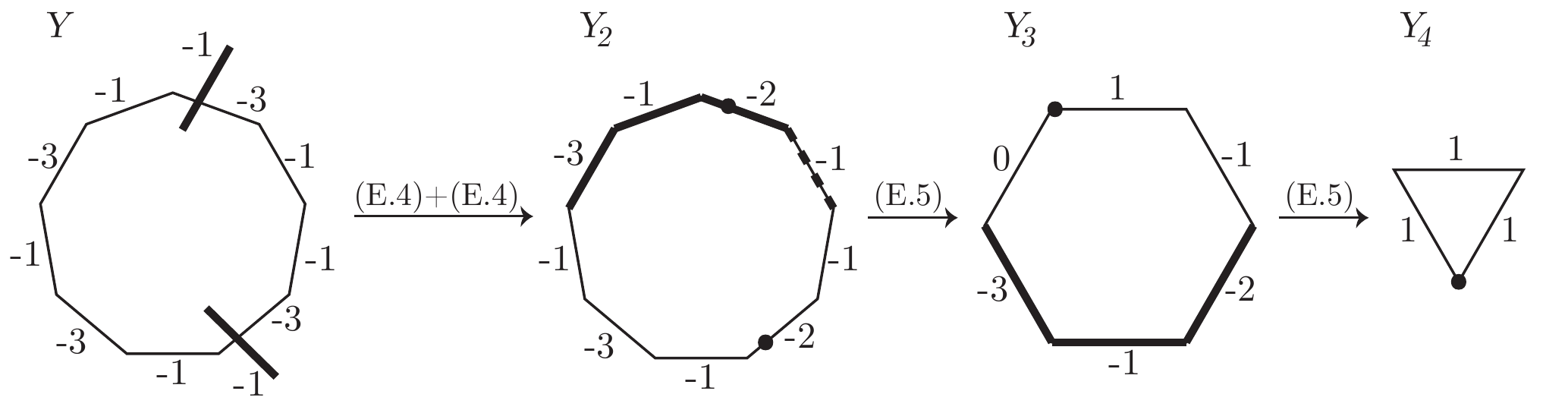}}
  \label{fig:k=4case3P2exp}
\end{figure}

Note that after having done the first two steps of the MMP, $Y_2$
contains two $(-1)$-curves showing that two (E.2)-type contractions
were available on $X_2$, which proves that the MMP is not directed and
that this case does not occur.

If $Y_4=\PP^1\times\PP^1$ we reach a contradiction without having to study the
minimal resolution of $X$. Indeed, the following figure shows the
contraction $Y_3\to Y_4$:

\begin{figure}[H]
  \centering
  \resizebox{16cm}{!}{\includegraphics[width=0.8\textwidth]{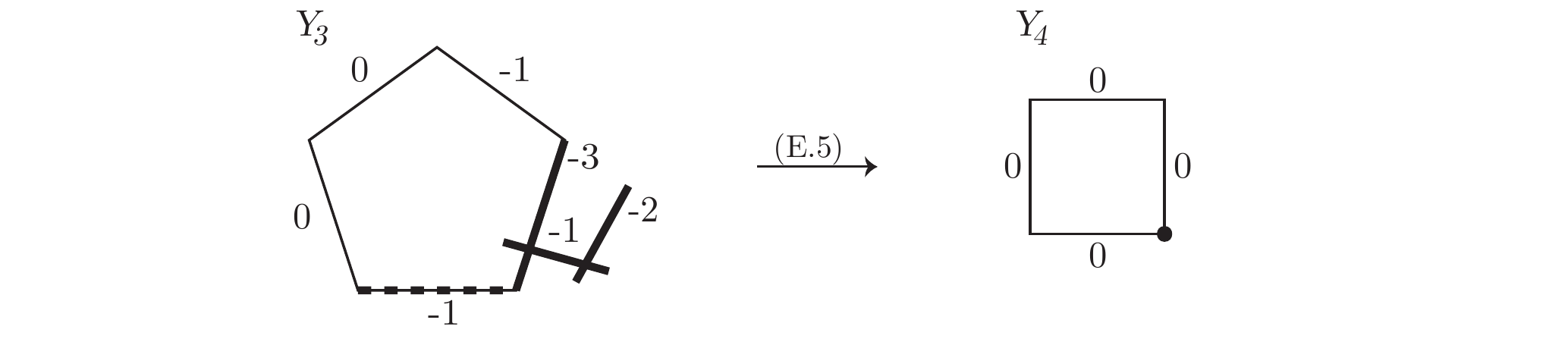}}
  \label{fig:k=4case3P1xP1exp}
\end{figure}

From the picture it is clear that an (E.4)-type contraction was
available on $X_3$, a contradiction. This case also does not occur,
and the only surface with four $\frac1{3}(1,1)$ points is the one
described in Case 1.

\paragraph{The $k=6$ case}

We argue that the sequence of extremal contractions of the directed
MMP is one of the two shown in figure~\ref{fig:k=6possibtree}.

We use the same notations as in the $k=4$ case, i.e we denote by 
\[
X=X_0\overset{\varphi_0}{\longrightarrow} X_1 \overset{\varphi_1}{\longrightarrow}\ldots \longrightarrow
X_{i-1}\overset{\varphi_{i-1}}{\longrightarrow} X_i \longrightarrow \ldots
\]
the sequence of contractions and surfaces occurring in a directed MMP
for $X$ and by $f_i\colon Y_i \to X_i$ the minimal
resolutions. Proposition~\ref{pro:2}, which we use repeatedly in the
course of the proof, implies that $\rho(X)\in \{1\ldots 5\}$. By
theorem~\ref{thm:2}, $\varphi_0$ is either of type (E.4) or (E.6) and
we claim that (E.4) does not occur.

Suppose for a contradiction that $\varphi_0$ is an (E.4) contraction,
then $X_{1}$ has $5\times \frac1{3}(1,1)+A_1$
singularities. Theorem~\ref{thm:2} implies that $\varphi_1$ is either
of type (E.2), (E.4), (E.5) or (E.6): indeed, $\varphi_1$ can not be a
conic bundle because $X_1$ contains singularities of type
$\frac1{3}(1,1)$ but none of type $A_2$, and it is clear from the
classification that $X_1$ is not a del Pezzo surface with $\rho =1$.

If $\varphi_1$ is of type (E.2), then $X_2$ is a del Pezzo surface
with $k=5$ singularities. As we assume the case $k=5$ is known, this
implies that $X_2=X_{5,5/3}$, a contradiction since $\rho(X_2)\leq3$
while $\rho(X_{5,5/3})=5$.

If $\varphi_1$ is a type (E.4) contraction, the surface $X_2$ has
$4\times \frac1{3}(1,1)+2\times A_1$ singularities and its Picard
number is at most three. From this surface, only contractions of type
(E.2), (E.4) and (E.6) are possible: once again $\varphi_2$ cannot be
of fibering type because of the presence of $\frac1{3}(1,1)$
singularities and the absence of those of type $A_2$, and by theorem
\ref{thm:2} it is not a del Pezzo surface of $\rho=1$.

We show that none of these possibilities occur. Indeed, if
$\varphi_2$ is of type (E.2), the del Pezzo surface $X_3$ has $\rho\leq 2$
and $4\times \frac1{3}(1,1)+A_1$ singularities. The following
contraction $\varphi_3$ cannot be a fibration because the number of singularities
is even, and $X_3$ cannot be of Picard rank one by the classification
in theorem \ref{thm:2}. Since a second (E.2) contraction should have
already been performed as $\varphi_1$, $\varphi_3$ can only be of
type (E.4), (E.5) or (E.6). These lead to surfaces of $\rho=1$ and
singularities of type $3\times \frac1{3}(1,1)+2\times
A_1$, $3\times \frac1{3}(1,1)$ and $2\times \frac1{3}(1,1)+A_1+A_2$
respectively, none of which appear in the list in theorem
\ref{thm:2}. 

The same reasoning holds if $\varphi_2$ is of type (E.4). In this
case, the del Pezzo surface $X_3$ has $\rho\leq 2$ and
$3\times \frac1{3}(1,1)+3 \times A_1$ singularities. By the
classification in theorem \ref{thm:2}, $X_3$ is clearly not a conic
bundle over $\PP^1$, nor does it have $\rho=1$. The next contraction
$\varphi_3$ can be either of type (E.2), (E.4), (E.5) or (E.6) and
leads to a surface $X_4$ of Picard rank one. The first three cases
result in at least two points of type $\frac1{3}(1,1)$ on $X_4$, while
the contraction of type (E.6) means that $X_4$ has
$\frac1{3}(1,1)+A_2+3\times A_1$ singularities. None of these
correspond to one of the $\rho=1$ surfaces in theorem \ref{thm:2},
since the only surface of this type containing a $\frac1{3}(1,1)$
point is $\PP(1,1,3)$.

Finally, if $\varphi_2$ is of type (E.6), then $X_3$ has
$2\times \frac1{3}+2\times A_1 + A_2$ singularities. Theorem
\ref{thm:2} implies that the only possible contractions on $X_3$ are
of type (E.3) and (E.6). Indeed, the odd number of singularities and
the presence of two $\frac1{3}(1,1)$ points allow us to respectively
eliminate the possibilities of $\varphi_3$ being of fibering type and
$X_3$ having $\rho=1$. A contraction of type (E.2) would have been
available earlier in the directed MMP since by performing $\varphi_2$
no new singularities of type $A_1$ were created. The same is true for
contractions of type (E.4) and (E.5) which, if available, should have
been done prior to the one of type (E.6). As before, both contractions
would lead to a non-existent del Pezzo surface of $\rho=1$: if
$\varphi_3$ is of type (E.3), then $X_4$ has
$2\times \frac1{3}+2\times A_1$ singularities and if $\varphi_3$ is of
type (E.6) we obtain $2\times A_1 + 2\times A_2$ singularities, a
contradiction to theorem \ref{thm:2}.

If $\varphi_1$ is of type (E.5), the del Pezzo surface $X_2$ has
exactly four $\frac1{3}(1,1)$ points. From our discussion in the case
$k=4$ we obtain that $X_2=X_{4,7/3}$ and since $\rho(X_2)\leq2$ and
$\rho(X_{4,7/3})=5$, this is a contradiction. 

If $\varphi_1$ is of type (E.6), $X_2$ has $3\times \frac1{3}+A_1+A_2$
singularities. This surface is not of Picard rank one, and by theorem
\ref{thm:2} it is not a conic fibration, as it has an odd number of
singularities and too many $\frac1{3}(1,1)$ points. Since by remark
\ref{rem:6and3} a contraction of type (E.3) cannot follow one of type
(E.6) and considering the order of the contractions in the directed
MMP, the only possibility left is that $\varphi_2$ is of type
(E.6). $X_3$ now has singularities of type $\frac1{3}+A_1+2\times A_2$
and no further divisorial contractions are available, again because we
cannot follow with one of type (E.3) and all other possiblities would
have been performed earlier. The singularities on $X_3$ do not however
correspond to any of the surfaces of $\rho=1$ in theorem \ref{thm:2},
nor can they be paired on singular fibers of a conic fibration, for
instance because there is only one point of type $A_1$. We have thus
exhausted all the possibilities for $\varphi_1$.

All of this shows that $\varphi_0$ can not be of type (E.4), thus it
is a contraction of type (E.6) and all the divisorial contractions
that follow must be of the same type. The surface $X_1$ can not be of
$\rho=1$ or of fibering type since it has too many $\frac1{3}(1,1)$
points and an odd number of singularities, thus $\varphi_1$ is also an
(E.6) contraction. As depicted in figure \ref{fig:k=6possibtree}, by
theorem \ref{thm:2} we have two possiblities: either $X_2$ is a conic
bundle with two special fibres of type (C.2), or one last divisorial
contraction is available, leading to the $\rho=1$ surface (D.4).

\begin{figure}[H]
    \caption{$k=6$ tree of possibilities}
    \label{fig:k=6possibtree}
  \centering
  \resizebox{16cm}{!}{\includegraphics[width=0.8\textwidth]{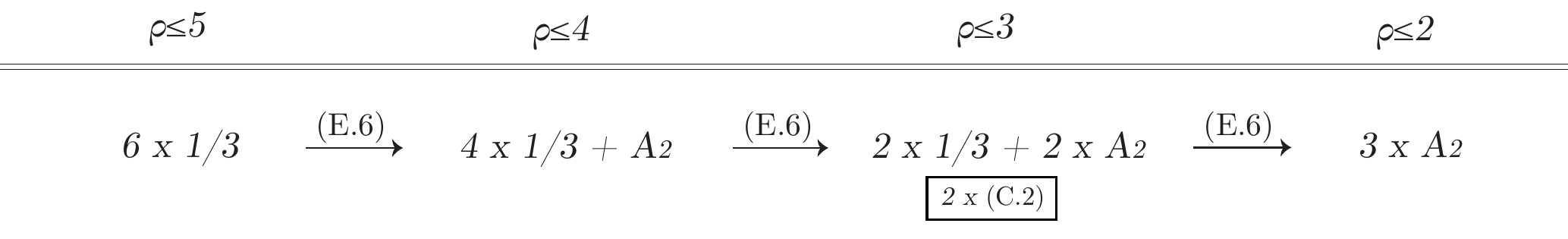}}
\end{figure}

Both instances occur and, as we will see, they lead to the same
surface $X$. This makes sense since at the very beginning of the MMP
there are a total of six contractions of the same type available and
at each step we choose one at the expense of two others. Depending on
their configuration we stop after either two or three contractions,
thus obtaining two end products of the directed MMP for the same $X$.

\paragraph{Case 1}
\label{sec:case1}

$(E.6)+(E.6)+(E.6)$
\medskip

The curve configuration for this case is:
\begin{figure}[H]
  \centering
  \resizebox{16cm}{!}{\includegraphics[width=0.8\textwidth]{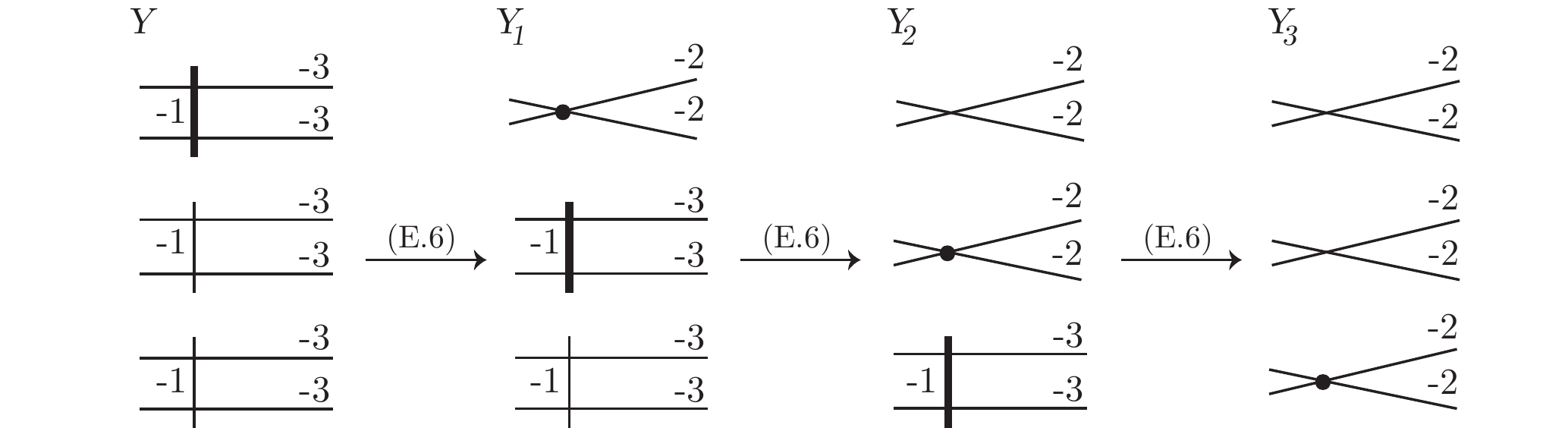}}
  \label{fig:k=6case1cv-cfg}
\end{figure}

Looking more closely at how $Y$ is built from
$Y_3$---the surface of
Picard rank one having three singular points of type $A_2$ described
in theorem~\ref{thm:2} ---by a sequence of blow-ups, we see that $Y$ must have the
negative curves shown in the following figure:

\begin{figure}[H]
  \centering
  \resizebox{16cm}{!}{\includegraphics[width=0.8\textwidth]{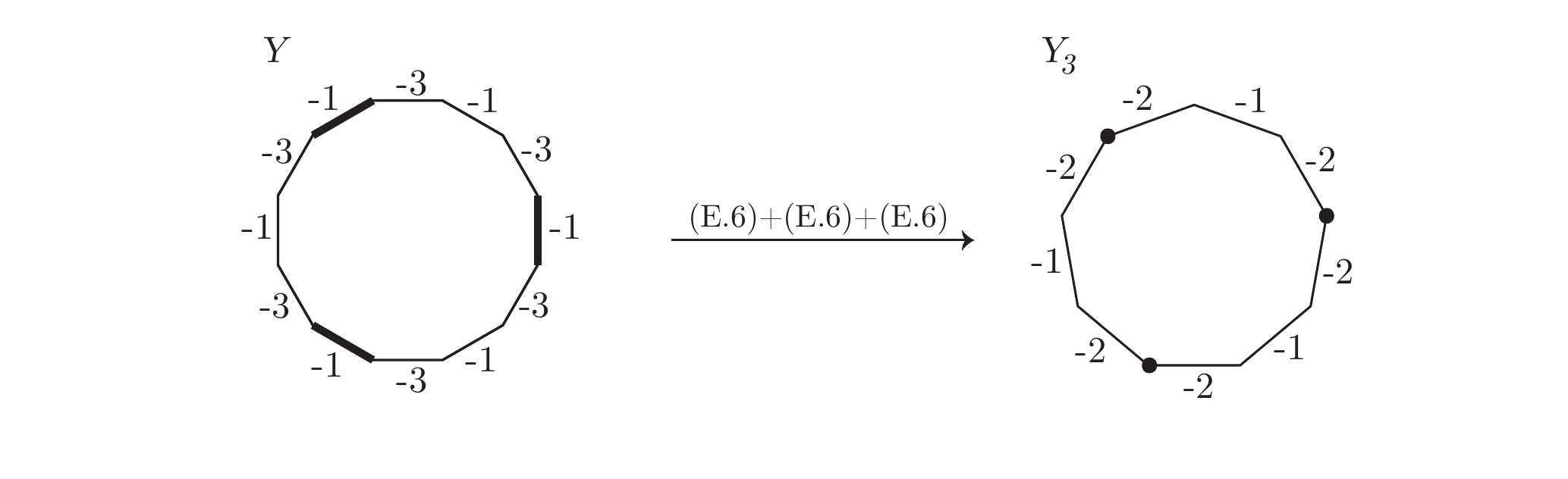}}
  \label{fig:k=6case1exp}
\end{figure}

\paragraph{Case 2}
\label{sec:case2}

$(E.6)+(E.6)$

This sequence ends with a conic fibration having two
singular fibers. On the minimal resolutions, the contractions are
the following:
\begin{figure}[H]
  \caption{Configuration of curves for $k=6$, Case~2}
  \label{fig:k=6caseAcv-cfg}
  \centering
  \resizebox{16cm}{!}{\includegraphics[width=0.8\textwidth]{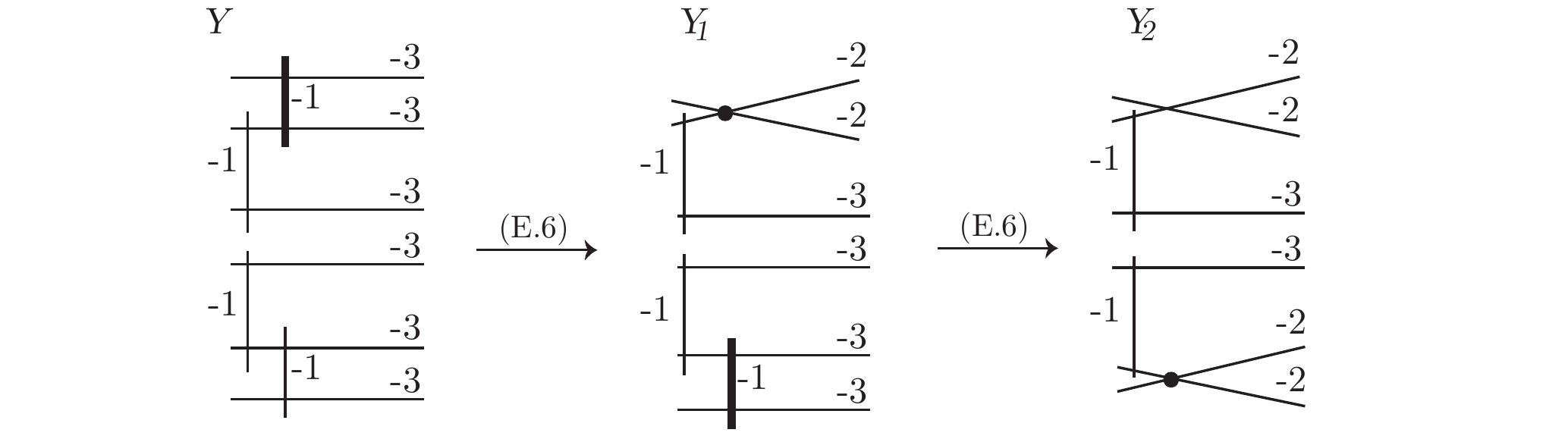}}
\end{figure}

We again proceed to run the nonsingular minimal model program for the
surface $Y_2$ relative to the current fibration over $\PP^1$. As
before, we eventually reach a surface $\FF_k$, where $k\in\{0,1,2\}$,
and we choose our sequence of contractions such that $k$ is
maximal. Figure \ref{fig:k=6caseAcv-cfg} shows that all $(-3)$-curves
on $Y$ either already existed on the special fibers of $Y_2$ or they
come from blowing up points inside these fibers. Thus if $k=2$, the
$(-2)$-section remains as such even on $Y$, which is impossible
according to lemma \ref{rmk:curves on Y}. By maximality, $k=1$, and
then the minimal resolutions must have the negative curves shown in
the following figure:

\begin{figure}[H]
  \centering
  \resizebox{16cm}{!}{\includegraphics[width=0.8\textwidth]{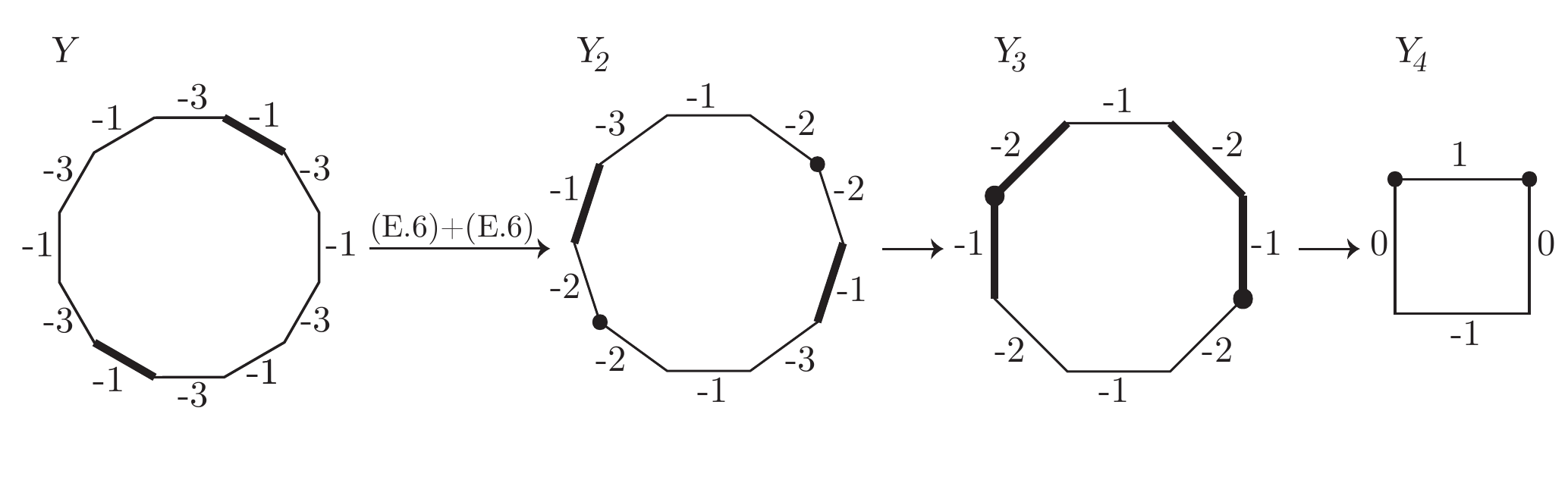}}
  \label{fig:k=6case2exp}
\end{figure}

Cases~1 and~2 are two different directed MMPs starting from the del
Pezzo surface $X_{6,2}$. 

We further give the trees of possibilities for the four remaining
cases. When constructing the tree for a surface with $k_0$ singularities,
we are allowed to use the final working cases for the surfaces with
$k<k_0$ points of type $\frac1{3}(1,1)$. This not only shortens the
process, but also allows us to further exclude certain
sequences of contractions. Indeed, suppose that in the $k_0$ tree a
branch leads to a del Pezzo surface with singularity content
$k \times \frac1{3}(1,1)$, where $k<k_0$. If the sequence in the statement of
theorem \ref{thm:4} for $k$ doesn't correlate with the directed
MMP obtained thus far in the $k_0$ tree, then the entire branch can be
removed. 

\begin{figure}[H]
  \centering
  \caption{$k=1$ Tree}
  \resizebox{16cm}{!}{\includegraphics[width=0.8\textwidth]{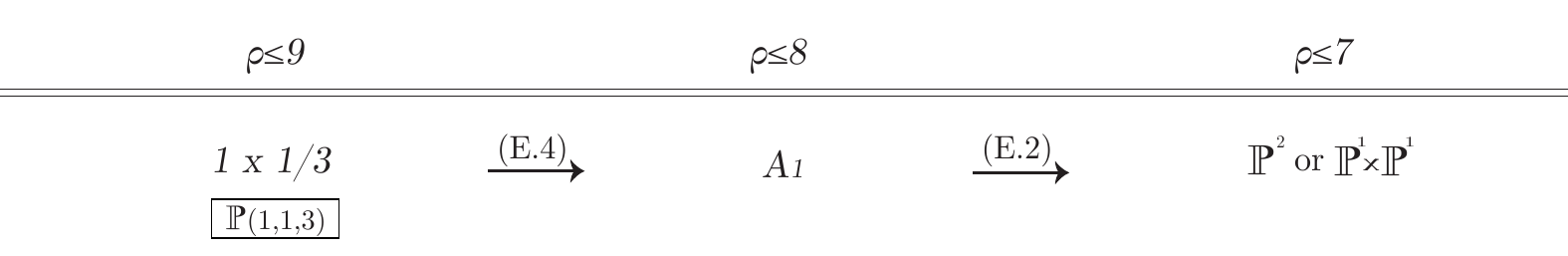}}
  \label{fig:k=1pos}
\end{figure}\begin{figure}[H]
\centering
  \caption{$k=2$ Tree}
  \resizebox{16cm}{!}{\includegraphics[width=0.8\textwidth]{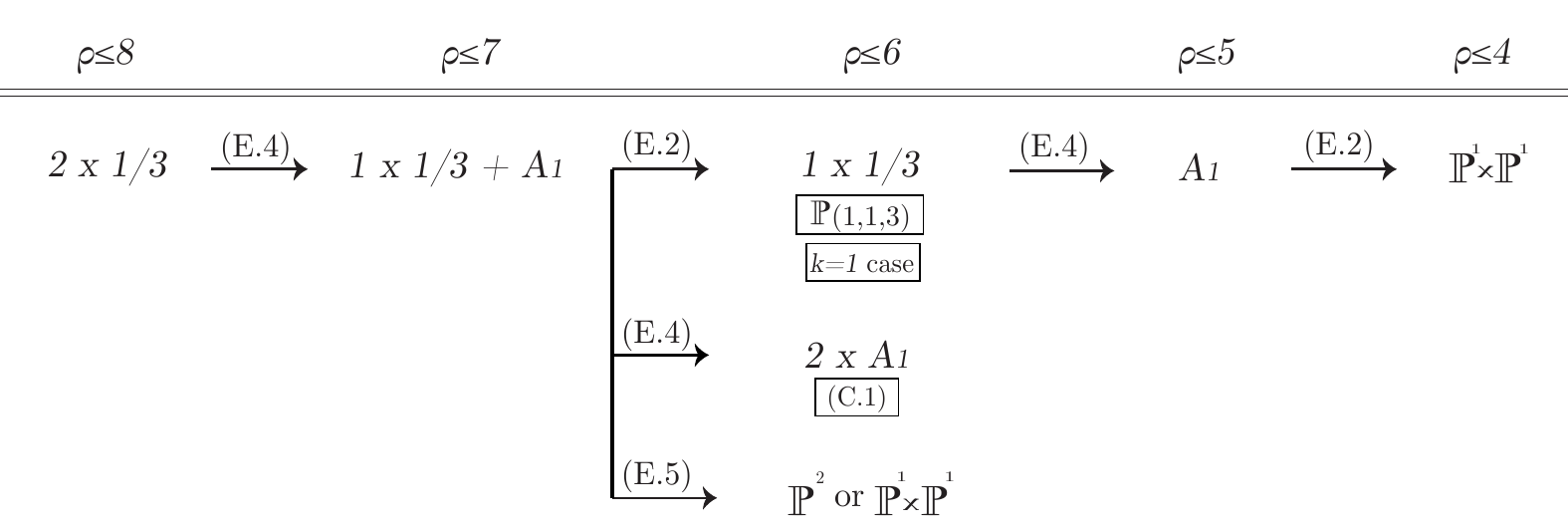}}
  \label{fig:k=2pos}
\end{figure}\begin{figure}[H]
\centering
  \caption{$k=3$ Tree}
  \resizebox{16cm}{!}{\includegraphics[width=0.8\textwidth]{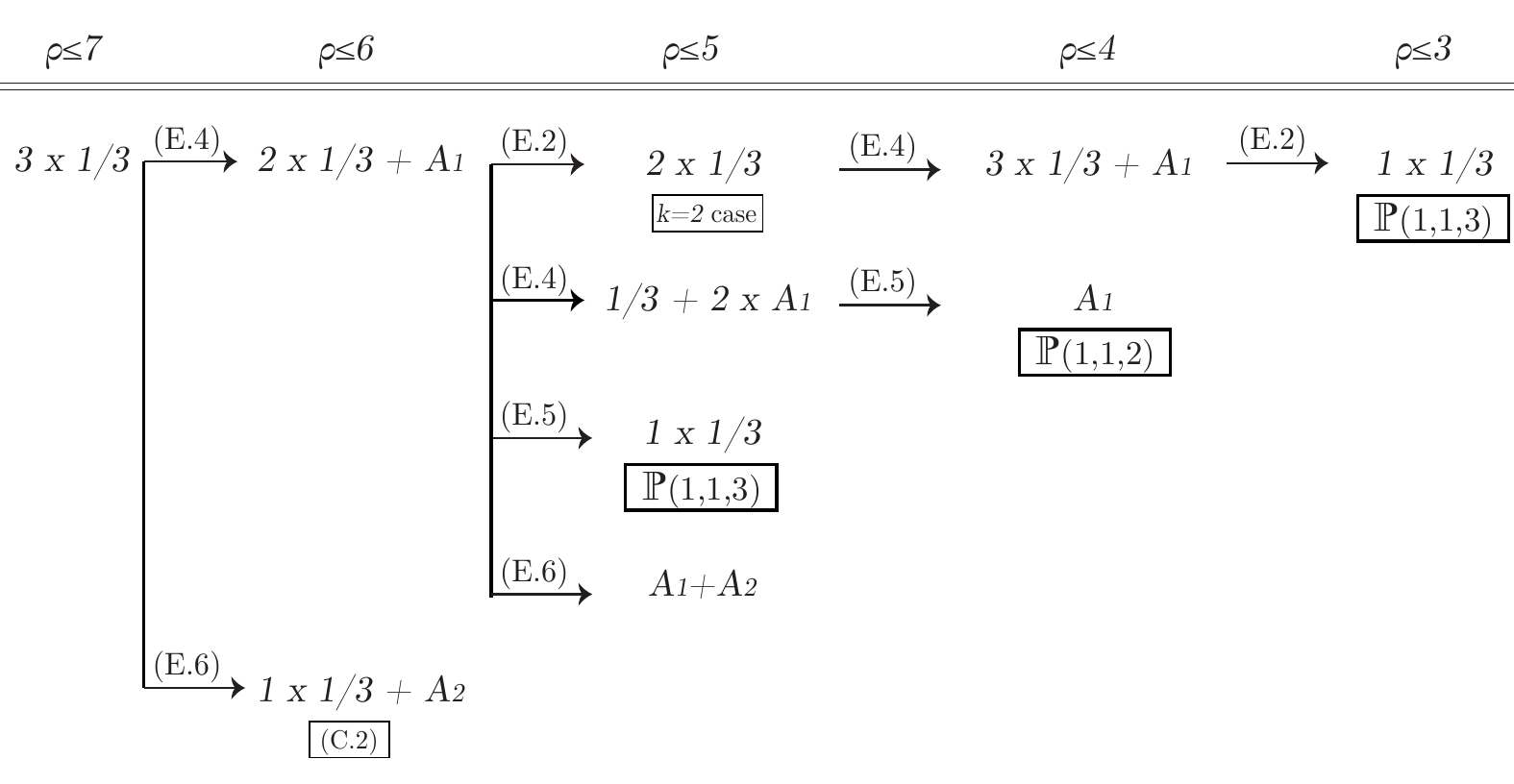}}
  \label{fig:k=3pos}
\end{figure}\begin{figure}[H]
  \centering
  \caption{$k=5$ Tree}
  \resizebox{16cm}{!}{\includegraphics[width=0.8\textwidth]{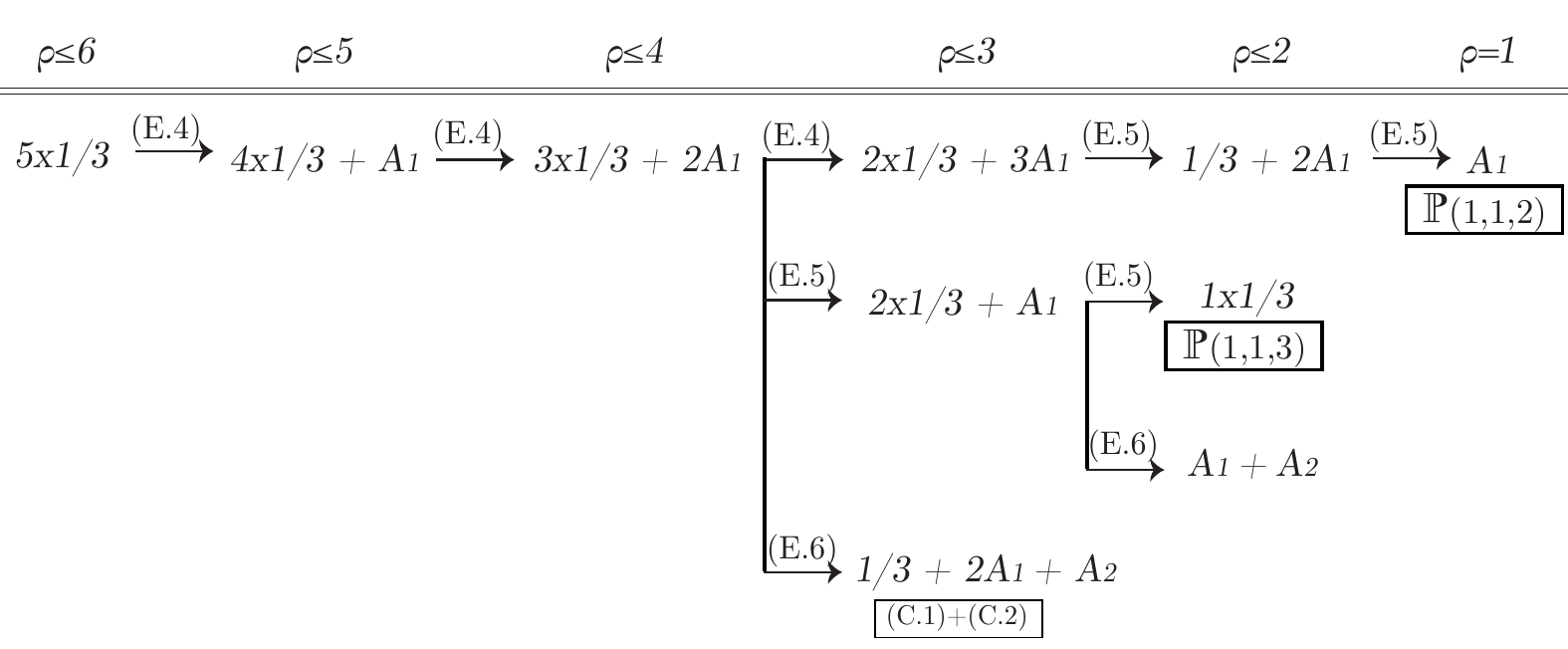}}
  \label{fig:k=5pos}
\end{figure}

Finally, using the techniques presented so far in this section will lead
to a systematic ellimination of the branches so that we are
left with exactly the surfaces in the statement.

\end{proof}


\section{Toric Degenerations}
\label{sec:toric-degenerations}


\begin{proof}[Proof of theorem~\ref{thm:6}] 

  We first argue that $\Xsurf{4}{1/3}$, $\Xsurf{5}{2/3}$ and
  $\Xsurf{6}{1}$ do not admit a toric qG-degeneration.
  For all of these $h^0(X,-K_X)=0$. But we know that $h^0(X,-K_X)$ is
  a qG-deformation invariant, and  if $X_0$ is a a toric surface then
  $H^0(X,-K_{X_0})\neq (0)$ as $|-K_{X_0}|$ always contains at least the
  toric boundary divisor.  We show that the remaining $26$ families do
  indeed admit a toric qG-degeneration.

  Let $N= \ZZ^2$. A Fano polygon is a convex lattice polygon
  $P\subset N_\QQ$ such that: (a) the origin lies strictly in the
  interior of $P$ and (b) the vertices of $P$ are primitive lattice
  vectors.  Table~\ref{tab:third_one_one_polygons} lists $26$ Fano
  polygons---each specified by the list of its
  vertices---$P\subset N_\QQ$ with singularity content
  $(n,\{k\times\frac{1}{3}(1,1)\})$, and
  figure~\ref{fig:third_one_one_polygons} shows pictures of all these
  polygons in turn and the location of the origin in their
  interiors. The singularity content of a polygon is defined in
  \cite[Definition~2.4 and~3.1]{AK14}. It is easy to see that the
  singularity content of $P$ is what it says on the table by looking
  at the picture of the polygon in
  fig.~\ref{fig:third_one_one_polygons}: for all cones of the polygon,
  all you have to do is eyeball the residue of that cone and persuade
  yourself that is is either empty or a $\frac1{3}(1,1)$ cone, and
  count the primitive $T$-cones. Then $n$ is the total number of
  $T$-cones over the whole polygon, and $k$ is the number of nonempty
  residue cones.

  For all $P$ in the table, let $\Sigma(P)$ be the face-fan of $P$ and
  denote by $X_P$ the toric surface constructed from $\Sigma(P)$. Then
  $-K_{X_P}^2=12-n-\frac{5k}{3}$.

  It is explained for instance in \cite[Lemma~6]{surfaces15} that
  qG-deformations of del Pezzo surfaces with cyclic quotient
  singularities are unobstructed. It follows from this that if $P$ is
  a polygon in the table, and it has singularity content
  $(n,\{k\times \frac{1}{3}(1,1)\})$, then $X_P$ qG-deforms to a
  locally qG-rigid del Pezzo surface $X$ with $k$ $\frac1{3}(1,1)$
  points and $K_X^2=K_{X_P}^2$, that is, one of the $26$ remaining
  families that can, in principle, have a toric
  qG-degeneration.

  To prove the theorem, all that is left to do is to determine which
  toric surfaces qG-deform to which locally qG-rigid families. For all
  $i\in \{1,\dots, 26\}$, denote by $Pi$ the $i$\textsuperscript{th}
  polygon of the list. Looking at the table, the only cases where
  there is any ambiguity are $P12$, $P13$---they both have singularity
  content $\bigl(6, 2\times \frac1{3}(1,1)\bigr)$---and $P21$,
  $P22$---they both have singularity content
  $\bigl(5, 1\times \frac1{3}(1,1)\bigr)$. We argue that:
  \begin{enumerate}[(a)]
  \item $X_{P12}$ qG-deforms to $\Bsurf{2}{8/3}$ and $X_{P13}$ to $\Xsurf{2}{8/3}$;
  \item $X_{P21}$ qG-deforms to $\Bsurf{1}{16/3}$ and $X_{P22}$ to $\Xsurf{1}{16/3}$.
  \end{enumerate}
  The best way to show this here is to write down explicitly the
  qG-deformation inside the ambient variety $F$ of
  table~\ref{table:1}. This is not hard to do by hand. 

  For instance $X_{P21}=\PP(1,3,4)$. Denote by $u,v,w$ be the co-ordinates
  of weights $1,3,4$ on $\PP(1,3,4)$ and by $x_0,x_1,x_2,y$ the
  coordinates of weights $1,1,1,3$ on $\PP(1,1,1,3)$. We define an embedding
  $i\colon \PP(1,3,4)\hookrightarrow \PP(1,1,1,3)$, such that $i^\star
  \oo(1)=\oo(4)$, as follows:
\[
 i^\sharp(x_0) = u^4, \; i^\sharp(x_1) = uv, \; i^\sharp(x_2) = w ,
 \quad i^\sharp (y) = v^4
\]
 and it is immediate that the image is the degree $4$ weighted hypersurface given by the
 binomial equation $yx_0-x_1^4=0$. This shows that $X_{P21}$ belongs
 to the family $\Bsurf{1}{16/3}$. 

 The case of $X_{P12}$ is similar and only slightly harder. The vertices
\[
u_0=
\begin{pmatrix}
  3\\1
\end{pmatrix},\;
u_1=
\begin{pmatrix}
  -3\\1
\end{pmatrix},\quad
v=
\begin{pmatrix}
  0\\-1
\end{pmatrix}
\]
 of $P12$ satisfy the relation $u_0+u_1+2v=0$ but together they only
 generate an index $3$ subgroup of $N=\ZZ^2$. This identifies $X_{P12}$
 with the quotient $\PP(1,1,2)/\Bmu_3$ where $\Bmu_3$ acts on the
 homogeneous coordinates $u_0,u_1,v$ with weights
 $\frac1{3},0,\frac1{3}\in \bigl(\frac1{3}\ZZ\bigr)/\ZZ$. Note,
 indeed, that $X_{P12}$ has $2\times \frac1{3}(1,1)$ points at $(1:0:0)$,
 $(0:1:0)$, and $1\times \frac1{6}(1,-1)$ at $(0:0:1)$. Let $L$ be the line
 bundle on $X_{P12}$ of integer weight $2$ and $\Bmu_3$-weight
 $\frac1{3}$. We define an embedding $i\colon X_{P12}\hookrightarrow
 \PP(1,1,3,3)$ with homogeneous coordinates $x_0$, $x_1$, $y_0$, $y_1$, such that $i^\star \oo(1)=L$, as follows:
\[
i^\sharp (x_0)=v,\;i^\sharp (x_1)=u_0u_1,\quad i^\sharp (y_0)=u_0^6,\;i^\sharp (y_1)=u_1^6
\]
and it is immediate that the image $i(X_{P12})\subset \PP(1,1,3,3)$ is the
degree $6$ weighted hypersurface given by the binomial equation
$y_0y_1-x_1^6=0$. This shows that $X_{P12}$ belongs
 to the family $\Bsurf{2}{8/3}$. 

Consider now $P13$. The vertices
\[
u_0=\begin{pmatrix}-1\\2\end{pmatrix},\;
u_1=\begin{pmatrix}2\\-1\end{pmatrix},\quad
v_0=\begin{pmatrix}-1\\-1\end{pmatrix},\;
v_1=\begin{pmatrix}1\\1\end{pmatrix}
\]
of $P13$ satisfy the relations:
\[
u_0 + u_1 + v_0 =0 \quad \text{and}\quad
v_0+v_1=0
\]
but together they only generate an index $3$ subgroup of
$N=\ZZ^2$. This identifies $X_{P13}$ with the quotient $\FF_1/\Bmu_3$ where
$\Bmu_3$ acts on $\FF_1$ as follows. Choose Cox coordinates and weight matrix for $\FF_1$ as:
\[
\begin{array}{cccc}
u_0 & u_1 & v_0 & v_1 \\
\hline
1 & 1 & 1 & 0 \\
0 & 0 & 1 & 1   
\end{array}
\]
The action of $\Bmu_3$ on $\FF_1$ is by weights
$\frac1{3},0,0,\frac1{3}\in \bigl(\frac1{3}\ZZ\bigr)/\ZZ$ on the Cox
coordinates. Also denote by $M_1$ the line bundle on $\FF_1$ with of
bidegree $(1,0)$ and trivial $\Bmu_3$-weight $\frac1{3}$ (any
$\Bmu_3$-weight will do here), and by $M_2$ the line bundle of
bidegree $(0,1)$ and $\Bmu_3$-weight $\frac1{3}$ (we need this
particular $\Bmu_3$-weight here). We construct an embedding
$i\colon X_{P13}\hookrightarrow F$ where $F$ is the Fano simplicial
toric \mbox{3-fold} with weight matrix and Cox coordinates:
\[
\begin{array}{ccccc}
s_0 & s_1 & x_0 & y & x_1 \\
\hline
1 & 1 & 1 & 1 & 0 \\
0 & 0 & 1 & 3 & 1   
\end{array}
\]
where we also denote by $L_1$, $L_2$ the standard basis of $\Cl F$,
such that $i^\star (L_1)=3M_1$ and $i^\star L_2=M_2$, as follows:
\[
i^\sharp(s_0)=u_0^3,\;
i^\sharp(s_1)=u_1^3,\quad
i^\sharp(x_0)=u_0u_1v_0,\;
i^\sharp(y)=v_0^3,\;
i^\sharp(x_1)=v_1
\]
and it is immediate that the image $i(X_{P13})\subset F$ is the
hypersurface of bidegree $(3,3)$ given by the binomial equation
$s_0s_1y-x_0^3=0$. This shows that $X_{P13}$ belongs to the family $\Xsurf{2}{8/3}$.

 The final case $P22$ is very similar and, in fact, easier. The vertices
\[
u_0=\begin{pmatrix}-1\\2\end{pmatrix},\;
u_1=\begin{pmatrix}0\\-1\end{pmatrix},\quad
v_0=\begin{pmatrix}1\\1\end{pmatrix},\;
v_1=\begin{pmatrix}-1\\-1\end{pmatrix}
\]
of $P22$ generate $N=\ZZ^2$ and satisfy the relations:
\[
u_0 + 3u_1 - v_1 =0 \quad \text{and}\quad
v_0+v_1=0
\]
This identifies $X_{P22}$ with the toric surface with Cox coordinates and weight matrix:
\[
\begin{array}{cccc}
u_0 & u_1 & v_0 & v_1 \\
\hline
1 & 3 & 0 & -1 \\
0 & 0 & 1 & 1   
\end{array}
\]
Denote by $M_1$, $M_2$ the standard basis of $\Cl X_{P22}$. We construct an embedding
$i\colon X_{P22}\hookrightarrow \PP^1\times \PP(1,1,3)$ such that
$i^\star \oo(1,0)=3M_1$, $i^\star \oo(0,1)=M_2$ and,  choosing 
homogeneous coordinates $s_0$, $s_1$ on $\PP^1$ and $x_0$, $x_1$, $y$ on
$\PP(1,1,2)$:
\[
i^\sharp(s_0)=u_0^3,\;
i^\sharp(s_1)=u_1,\quad
i^\sharp(x_0)=v_0,\;
i^\sharp(x_1)=u_0v_1,\;
i^\sharp(y)=u_1v_1^3
\]
and it is immediate that the image $i(X_{P22})\subset \PP^1\times \PP(1,1,3)$ is the
hypersurface of bidegree $(1,3)$ given by the binomial equation
$s_0y-s_1x_1^3=0$. This shows that $X_{P22}$ belongs to the family $\Xsurf{1}{16/3}$.
\end{proof}

\begin{rem}
  \label{rem:10} Our proof of theorem~\ref{thm:6} is by very efficient
  ad hoc considerations.  The Gross--Siebert program is a systematic
  approach to constructing qG-deformations of toric surfaces, see
  \cite{prince:_smoot_toric_fano_surfac_using}. Paper
  \cite{coates:_lauren} gives a systematic approach to constructing
  qG-deformations as toric complete intersections, when they exist.
\end{rem}

\begin{table}[ht]
\centering
\caption{$26$ Fano polygons $P\subset N_\QQ$ with singularity content $(n,\{k\times\frac{1}{3}(1,1)\})$}
\label{tab:third_one_one_polygons}
\setlength{\extrarowheight}{0.2em}
\begin{tabular}{rlccc}
\toprule
\multicolumn{1}{c}{\#}&\multicolumn{1}{c}{$\V{P}$}&$n$&$k$&Deforms to\\
\cmidrule(lr){1-1} \cmidrule(lr){2-2} \cmidrule(lr){3-5}
\oddrow \padding $1$&\gap $(7,5),(-3,5),(-3,-5)$&$10$&$1$&$\Xsurf{1}{1/3}$\\ 
\evnrow \padding $2$&\gap $(3,2),(-3,2),(-3,-2),(3,-2)$&$8$&$2$&$\Xsurf{2}{2/3}$\\ 
\oddrow \padding $3$&\gap $(3,1),(3,2),(-1,2),(-2,1),(-2,-3),(-1,-3)$&$6$&$3$&$\Xsurf{3}{1}$\\ 
\evnrow \padding $4$&\gap $(3,2),(-1,2),(-2,1),(-2,-3)$&$9$&$1$&$\Xsurf{1}{4/3}$\\ 
\oddrow \padding $5$&\gap $(2,1),(1,2),(-1,2),(-2,1),(-2,-1),(-1,-2),(1,-2),(2,-1)$&$4$&$4$&$\Xsurf{4}{4/3}$\\ 
\evnrow \padding $6$&\gap $(3,2),(-1,2),(-2,1),(-2,-1),(-1,-2)$&$7$&$2$&$\Xsurf{2}{5/3}$\\ 
\oddrow \padding $7$&\gap $(2,1),(1,2),(-1,2),(-2,1),(-2,-1),(-1,-2),(1,-1)$&$2$&$5$&$\Xsurf{5}{5/3}$\\ 
\evnrow \padding $8$&\gap $(2,1),(1,2),(-1,2),(-2,1),(-2,-1),(-1,-2)$&$5$&$3$&$\Xsurf{3}{2}$\\ 
\oddrow \padding $9$&\gap $(1,1),(-1,2),(-2,1),(-1,-1),(1,-2),(2,-1)$&$0$&$6$&$\Xsurf{6}{2}$\\ 
\evnrow \padding $10$&\gap $(1,1),(-1,2),(-1,-2),(1,-2)$&$8$&$1$&$\Xsurf{1}{7/3}$\\ 
\oddrow \padding $11$&\gap $(1,1),(-1,2),(-2,1),(-1,-1),(2,-1)$&$3$&$4$&$\Xsurf{4}{7/3}$\\ 
\evnrow \padding $12$&\gap $(3,1),(-3,1),(0,-1)$&$6$&$2$&$\Bsurf{2}{8/3}$\\ 
\oddrow \padding $13$&\gap $(1, 1),(-1, 2),(-1, -1),(2, -1)$&$6$&$2$&$\Xsurf{2}{8/3}$\\ 
\evnrow \padding $14$&\gap $(1,1),(-1,2),(-2,1),(-1,-1),(1,-1)$&$4$&$3$&$\Xsurf{3}{3}$\\ 
\oddrow \padding $15$&\gap $(1,1),(-1,2),(-1,-1),(1,-1)$&$7$&$1$&$\Xsurf{1}{10/3}$\\ 
\evnrow \padding $16$&\gap $(1, 1),(-1, 2),(-1, 0),(0, -1),(2, -1)$&$5$&$2$&$\Xsurf{2}{11/3}$\\ 
\oddrow \padding $17$&\gap $(1,0),(1,1),(-1,2),(-2,1),(-1,-1),(0,-1)$&$3$&$3$&$\Xsurf{3}{4}$\\ 
\evnrow \padding $18$&\gap $(1,0),(0,1),(-1,1),(-1,-3)$&$6$&$1$&$\Xsurf{1}{13/3}$\\ 
\oddrow \padding $19$&\gap $(1, 1),(-1, 2),(-1, 1),(0, -1),(2, -1)$&$4$&$2$&$\Xsurf{2}{14/3}$\\ 
\evnrow \padding $20$&\gap $(1,1),(-1,2),(-2,1),(-1,-1),(0,-1)$&$2$&$3$&$\Xsurf{3}{5}$\\ 
\oddrow \padding $21$&\gap $(1,1),(-1,2),(-1,-2)$&$5$&$1$&$\Bsurf{1}{16/3}$\\ 
\evnrow \padding $22$&\gap $(1, 1),(-1, 2),(-1, -1),(0, -1)$&$5$&$1$&$\Xsurf{1}{16/3}$\\ 
\oddrow \padding $23$&\gap $(1, 1),(-1, 2),(0, -1),(2, -1)$&$3$&$2$&$\Xsurf{2}{17/3}$\\ 
\evnrow \padding $24$&\gap $(0,1),(-1,2),(-2,1),(-1,0),(1,-1)$&$4$&$1$&$\Xsurf{1}{19/3}$\\ 
\oddrow \padding $25$&\gap $(0,1),(-1,2),(-2,1),(1,-1)$&$3$&$1$&$\Xsurf{1}{22/3}$\\ 
\evnrow \padding $26$&\gap $(-1,2),(-2,1),(1,-1)$&$2$&$1$&$\Ssurf{1}{25/3}$\\ 
\bottomrule
\end{tabular}
\end{table}

\begin{figure}[ht]
  \caption{$26$ Fano polygons $P\subset N_\QQ$ with singularity
    content $(n,\{k\times\frac{1}{3}(1,1)\})$. See also
    Table~\ref{tab:third_one_one_polygons}.}
\label{fig:third_one_one_polygons}
\centering
\vspace{0.5em}
\renewcommand{\arraystretch}{0.8}
\begin{tabular}{cc}
$1$&
$2$\\
\includegraphics[scale=0.5]{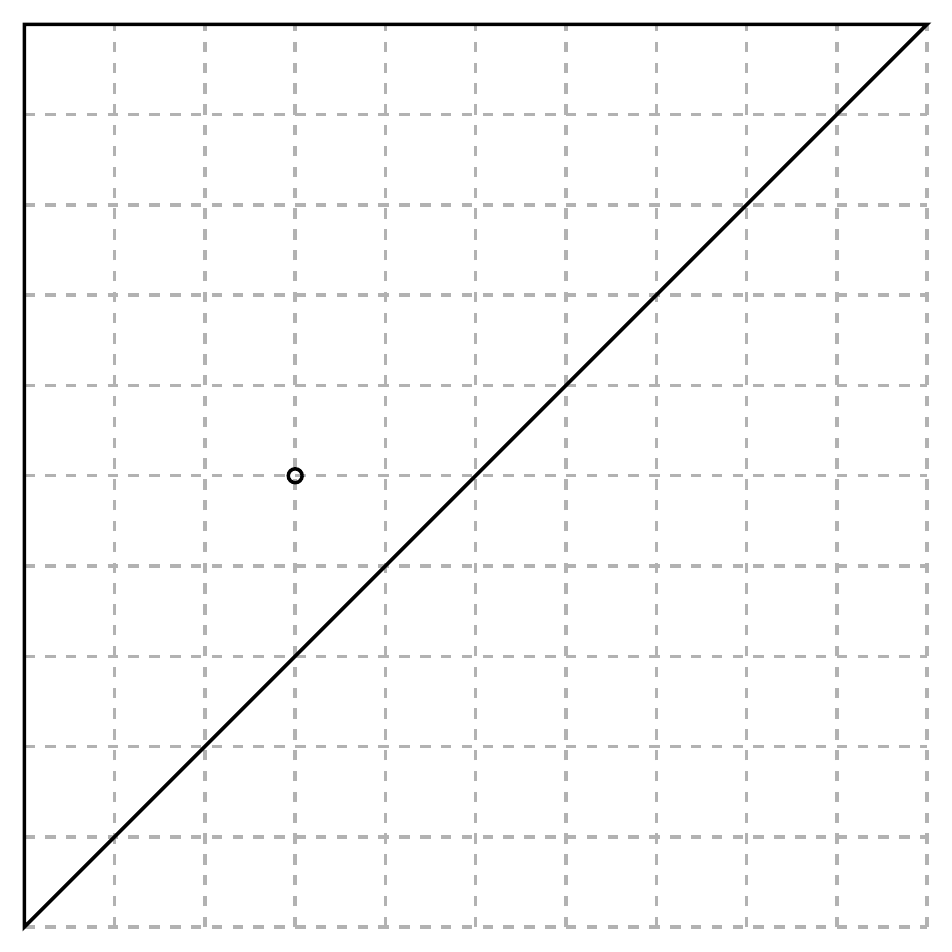}&
\raisebox{78px}{\includegraphics[scale=0.5]{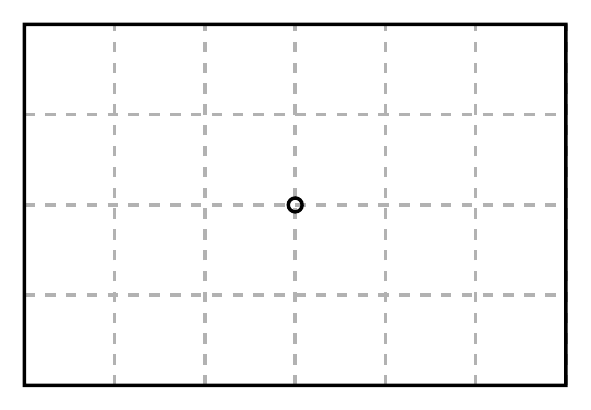}}\\
\end{tabular}
\vgap

\begin{tabular}{ccccc}
$3$&
$4$&
$5$&
$6$&
$7$\\
\includegraphics[scale=0.5]{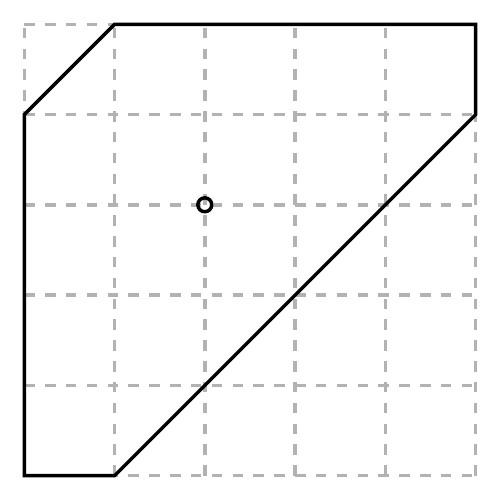}&
\includegraphics[scale=0.5]{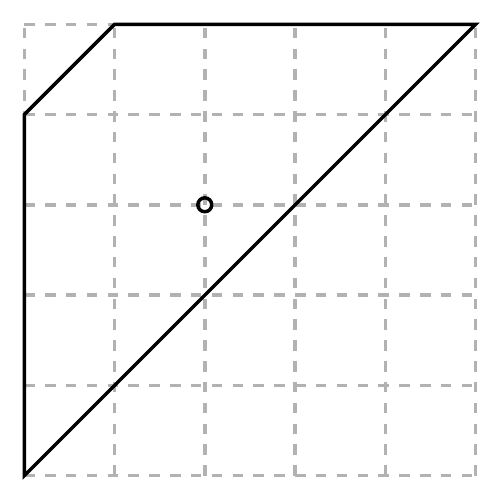}&
\raisebox{13px}{\includegraphics[scale=0.5]{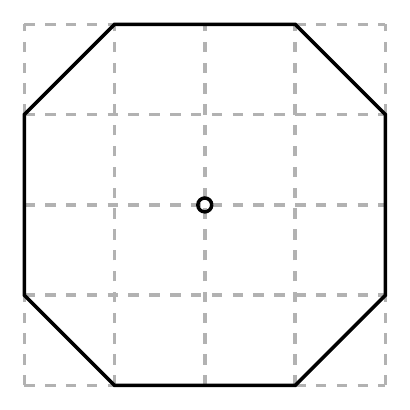}}&
\raisebox{13px}{\includegraphics[scale=0.5]{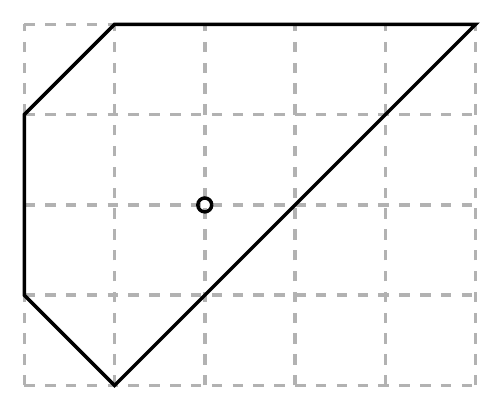}}&
\raisebox{13px}{\includegraphics[scale=0.5]{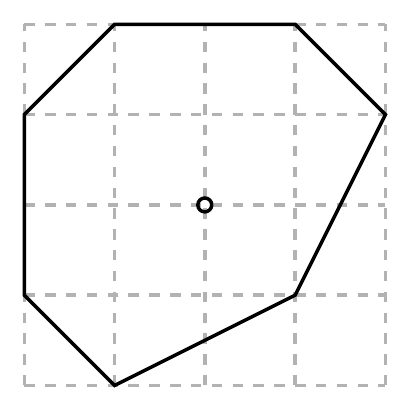}}\\
\end{tabular}
\vgap

\begin{tabular}{cccccc}
$8$&
$9$&
$10$&
$11$&
$12$&
$13$\\
\includegraphics[scale=0.5]{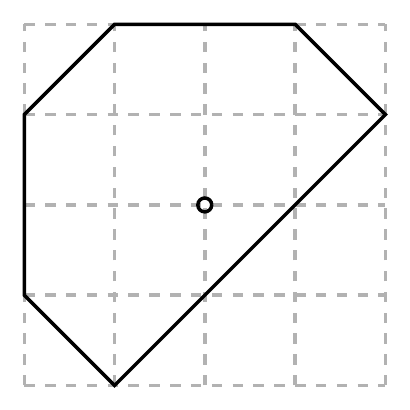}&
\includegraphics[scale=0.5]{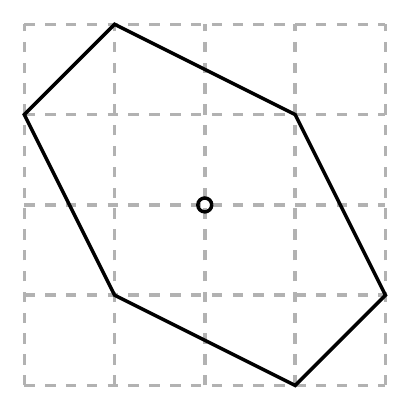}&
\includegraphics[scale=0.5]{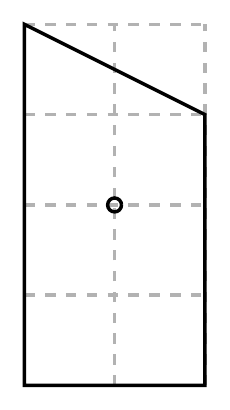}&
\raisebox{13px}{\includegraphics[scale=0.5]{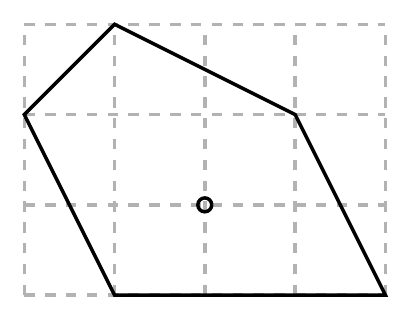}}&
\raisebox{26px}{\includegraphics[scale=0.5]{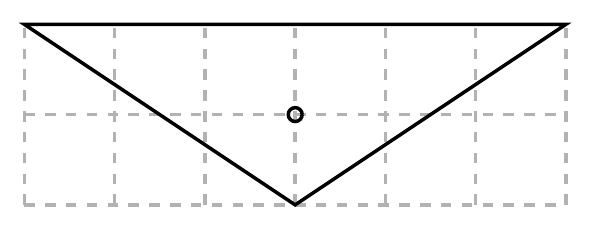}}&
\raisebox{13px}{\includegraphics[scale=0.5]{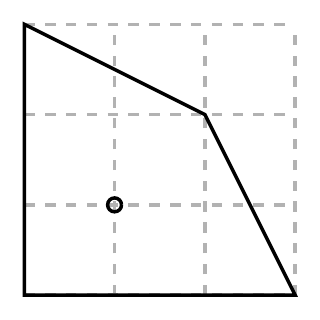}}\\
\end{tabular}
\vgap

\begin{tabular}{ccccccc}
$14$&
$15$&
$16$&
$17$&
$18$&
$19$&
$20$\\
\raisebox{13px}{\includegraphics[scale=0.5]{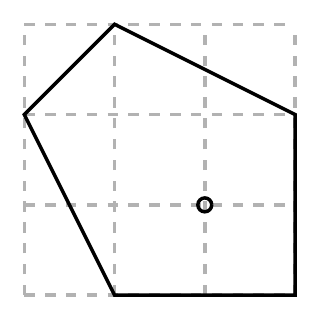}}&
\raisebox{13px}{\includegraphics[scale=0.5]{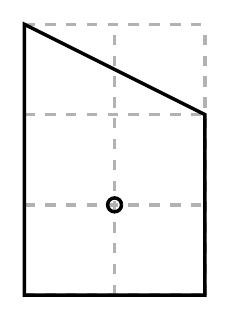}}&
\raisebox{13px}{\includegraphics[scale=0.5]{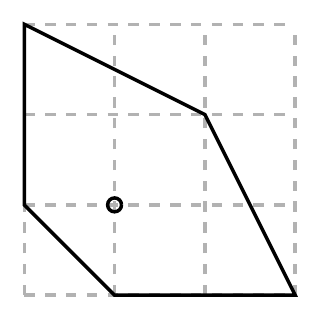}}&
\raisebox{13px}{\includegraphics[scale=0.5]{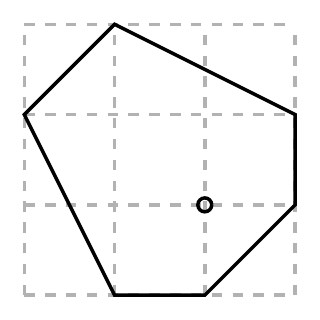}}&
\includegraphics[scale=0.5]{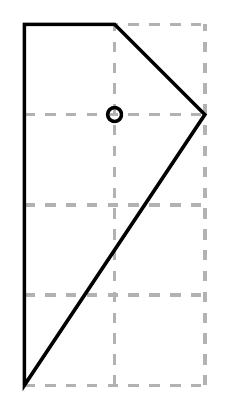}&
\raisebox{13px}{\includegraphics[scale=0.5]{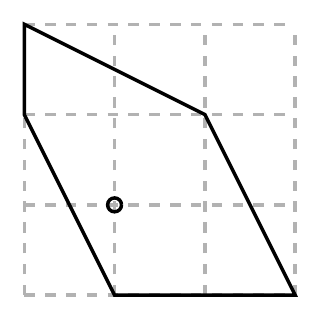}}&
\raisebox{13px}{\includegraphics[scale=0.5]{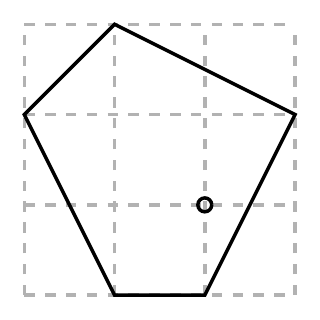}}\\
\end{tabular}
\vgap

\begin{tabular}{ccccccc}
$21$&
$22$&
$23$&
$24$&
$25$&
$26$\\
\includegraphics[scale=0.5]{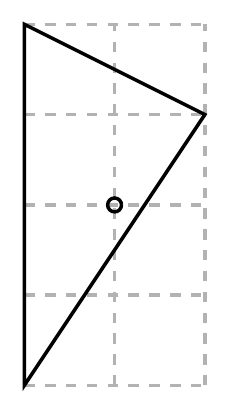}&
\raisebox{13px}{\includegraphics[scale=0.5]{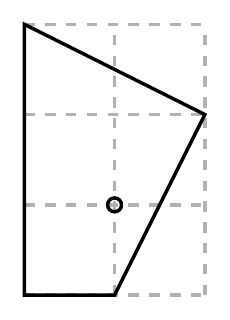}}&
\raisebox{13px}{\includegraphics[scale=0.5]{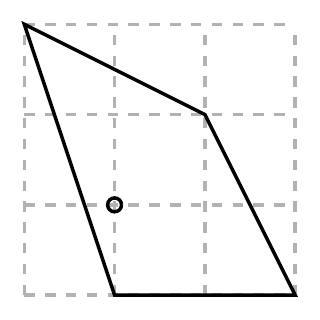}}&
\raisebox{13px}{\includegraphics[scale=0.5]{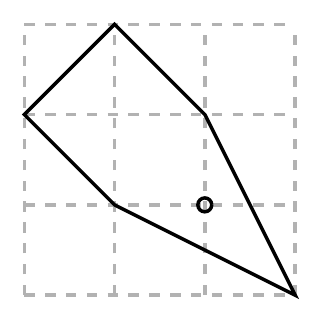}}&
\raisebox{13px}{\includegraphics[scale=0.5]{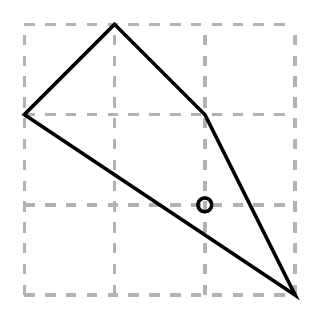}}&
\raisebox{13px}{\includegraphics[scale=0.5]{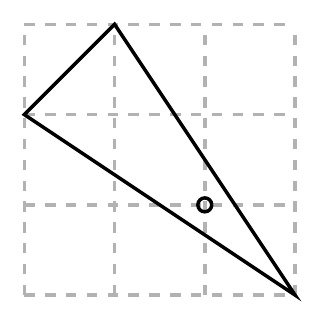}}\\
\end{tabular}
\end{figure}


\bibliography{bib_aclh}
\end{document}